\documentclass[onefignum,onetabnum]{siamart251216}

\usepackage{amsfonts,amsmath,amssymb}
\usepackage{mathrsfs,mathtools,stmaryrd,wasysym}
\usepackage[all]{xy}
\usepackage{multirow}
\usepackage{enumerate}
\usepackage{float}
\usepackage{esint}
\usepackage{graphicx,psfrag}
\usepackage{hyperref}
\DeclareMathAlphabet{\mathpzc}{OT1}{pzc}{m}{it}
\usepackage{lineno}
\usepackage{stmaryrd}
\usepackage{enumitem}
\usepackage{pgf,tikz-cd,pgfplots}
\usepackage{amsopn}

\DeclareFontEncoding{FMS}{}{}
\DeclareFontSubstitution{FMS}{futm}{m}{n}
\DeclareFontEncoding{FMX}{}{}
\DeclareFontSubstitution{FMX}{futm}{m}{n}
\DeclareSymbolFont{fouriersymbols}{FMS}{futm}{m}{n}
\DeclareSymbolFont{fourierlargesymbols}{FMX}{futm}{m}{n}
\DeclareMathDelimiter{\VERT}{\mathord}{fouriersymbols}{152}{fourierlargesymbols}{147}

\numberwithin{equation}{section}

\usepackage{mathrsfs}
\def\ds{\displaystyle}

\def\O{\Omega}

\def\b{\beta}

\newcommand{\bb}{\boldsymbol}

\def\CT{{\mathcal T}}

\def\H{\mathrm H}
\def\Q{\mathbb Q}
\def\I{\mathrm I}
\def\V{\mathrm V}

\def\L{\mathrm L}

\def\W{\mathrm W}

\def\div{\mathop{\mathrm{div}}\nolimits}
\def\curl{\mathop{\mathrm{curl}}\nolimits}

\def\bu{\boldsymbol{u}}

\def\bv{\boldsymbol{v}}

\def\bn{\boldsymbol{n}}

\def\bz{\boldsymbol{z}}

\def\bf{\boldsymbol{\textit{f}}}

\def\b0{\boldsymbol{0}}

\def\bPi{\boldsymbol{\Pi}}

\def\bw{\boldsymbol{w}}

\renewcommand\H{\mathrm{H}}
\renewcommand\L{\mathrm{L}}
\renewcommand\Q{\mathrm{Q}}
\def\bVh{\mathbf{V}_h}

\def\bVV{\mathrm{V}}

\def\bZ{\mathbf{Z}}
\def\bZh{\mathbf{Z}_h}

\def\bV{\mathbf{V}}

\def\div{\mathop{\mathrm{div}}\nolimits}
\def\rot{\mathop{\mathrm{rot}}\nolimits}

\newtheorem{remark}[theorem]{Remark}
\renewcommand{\theequation}{\thesection.\arabic{equation}}

\usepackage{pgf,tikz}
\usetikzlibrary{arrows}

\newcommand{\TheTitle}{A virtual element method for Darcy's problem coupled with a heat equation}
\newcommand{\ShortTitle}{A VEM for Darcy's problem coupled with a heat equation}
\newcommand{\TheAuthors}{D. Amigo and F. Lepe}

\headers{\ShortTitle}{\TheAuthors}

\title{{\TheTitle}\thanks{DA is partially supported by ANID-Chile through grant ACT210087. FL is partially supported by Universidad del B\'io-B\'io through regular project RE2514703. }}
\author{Danilo Amigo\thanks{Departamento de Matem\'atica, Universidad del B\'io-B\'io, Concepci\'on, Chile.
(\email{danilo.amigo2101@alumnos.ubiobio.cl})}
\and
Felipe Lepe\thanks{Departamento de Matem\'atica, Universidad del B\'io-B\'io, Concepci\'on, Chile.
(\email{flepe@ubiobio.cl}).}
}

\ifpdf
\hypersetup{
  pdftitle={\TheTitle},
  pdfauthor={\TheAuthors}
}
\fi

\date{Draft version of \today.}

\begin{document}

\maketitle

\begin{abstract}
In two dimensions, we develop a virtual element method to solve the Darcy's problem coupled with a nonlinear heat equation. This coupled model may allow for thermal diffusion and viscosity as a function of temperature. Under standard discretization assumptions and appropriate assumptions on the data, we prove the well posedness of the proposed numerical scheme. We also derive optimal error estimates under suitable regularity assumptions for the solution and appropriate assumptions on the data. We conclude with a series of numerical tests performed on different families of meshes that complement the theoretical findings.
\end{abstract}



\section{Introduction}\label{sec:intro}

Let $\O\subset\mathbb{R}^2$ be a simply-connected, bounded and open domain with Lipschitz boundary $\Gamma$. In this work, we are interested in analyzing a divergence-free virtual element method (VEM) for the temperature distribution of a fluid modeled by a convection-diffusion equation coupled with a nonlinear Darcy's problem. This model, which is inspired by works like \cite{RASHAD2014134},  can be described by the following nonlinear system of partial differential equations (PDEs):
\begin{equation}\label{eq:DH}
\left\{\begin{array}{rccc}
\nu(T)\bu+\nabla \textsf{p}  &=&  \bb{f}  \text{ in } \Omega,
\\
\div\bu &=& 0 \,\, \text{in} \,\, \Omega, \\
-\div(\kappa(T)\nabla T)+\bu\cdot\nabla T &=& g \,\,\text{in}\,\, \Omega, \\
\bu\cdot\bn&=& 0 \,\, \text{on} \,\, \Gamma, \\
T&=& 0 \,\, \text{on} \,\, \Gamma.
\end{array}\right.
\end{equation}
The unknowns of  system \eqref{eq:DH} are the velocity field $\bu$, the fluid  pressure $\textsf{p}$ and the temperature $T$ of the fluid. The data are the external force $\boldsymbol{f}$, the external heat source $g$, the viscosity coefficient $\nu(\cdot)$, and the thermal diffusion coefficient $\kappa(\cdot)$. We note that $\nu(\cdot)$ and $\kappa(\cdot)$ are coefficients that can depend nonlinearly on the temperature.

Darcy's law is a well known topic in fluid mechanics, where is possible to describe the motion of certain fluid in porous media. Coupling this law with the heat equations gives additional information of the fluid since its viscosity may change under certain temperatures on the media. This easy but important description is what motivates the study of \eqref{eq:DH}, where different formulations and numerical approximations have been introduced through years. An  important  and crucial study of system \eqref{eq:DH} from a mathematical and numerical point of view is contained in \cite{BernardiDH}. Here, existence, uniqueness, and error estimates by means of a finite element method are proved for a velocity-pressure-temperature formulation. The results of  \cite{BernardiDH} have inspired a number of works. We mention, for instance, spectral methods \cite{MR3523581} and  a recent formulation where Dirac measures are considered as forcing terms for the Darcy's equations which has been analyzed in \cite{MR4840402} and the diffusive term associated to the heat equation depends on the temperature, which is not considered on the original model of \cite{BernardiDH}. Convergence and a posteriori error estimates were derived for this model under the approach of standard finite element method. Other contribution for the numerical analysis of \eqref{eq:DH} is provided by \cite{MR4036533}, where a non-stabilized method is analyzed, providing a priori error estimates in two and three dimensions. Also, \cite{MR4041519} is focused in a posteriori error estimates. These references also are interested in the velocity-pressure-temperature formulation. However, other approaches such as mixed formulations are also available. We mention  on this topic the work \cite{MR4497812} where an additional unknown called the pseudoheat flux is incorporated.

The model problem \eqref{eq:DH} has the same condition on the thermal diffusion  $\kappa(\cdot)$ as in  \cite{MR4840402}, where the diffusive term of the heat equation depends on the temperature, which is not considered on \cite{BernardiDH}. Hence, a first step is to prove that the continuous problem is well posed. Since we are considering regular forcing terms, contrary to  \cite{MR4840402}, the continuous analysis of this reference is  not replicable and must be performed for our configuration. Hence, the results of \cite{BernardiDH} need to be adapted when the diffusion on the heat equation depends on the temperature. Hence, with the Brower's fixed point 
we conclude existence and uniqueness of solution for the continuous setting.
The main goal of our paper is to analyze numerically system \eqref{eq:DH} with the virtual element method \cite{MR2997471}. This method has shown important accuracy on the approximation of partial differential equations (PDEs) in different contexts as is shown in the recent book \cite{MR4510898}, demonstrating that  not only is a robust method of approximation, but also can reduce computational costs when we compare it with other methods available on the literature. Certainly, the research on the application of VEMs is in ongoing progress and for coupled non-linear systems the results are well established. For instance we present the following non-exhaustive list of contributions  \cite{MR4896769,MR4673339,MR4552294,MR4727111,MR4721561,MR4300149,MR4827116} where Navier-Stokes, Brinkman-Forchheimer, and  Boussinesq problems have been addressed with VEMs for different type of formulations. This motivates to continue the research on the capabilities of the VEM for coupled problems, particularly on flow models coupled with temperature. Hence, for \eqref{eq:DH} we propose a VEM method to discretize the velocity-pressure-temperature formulation  of the model, where $\H(\text{div})$-VEM spaces as the one in \cite{BMRR} are required to approximate the velocity on Darcy and standard piecewise polynomials for the corresponding pressure, whereas for the temperature is sufficient to consider the classic VEM spaces for the Laplace operator \cite{MR4567234}. This choice of  discrete spaces provide existence and uniqueness for the discrete coupled problem and optimal error estimates for each of the involved unknowns.

\subsection{Outline}
The paper is structured as follows: We begin in Section \ref{sec:notation_and_preliminaries} by introducing the notation that we will use through our work, namely spaces, norms, and  embeddings. Also, we present some assumptions on the the physical coefficients. Next in Section \ref{sec:DH} we introduce the weak formulation for \eqref{eq:DH} and its mathematical analysis, where existence and uniqueness of solution for the coupled problem is studied. Section \ref{sec:vem} is dedicated to the virtual element discretization of the corresponding variational formulation studied in the previous section. Under standard assumptions on the polygonal meshes we prove the well posedness of the discrete scheme, where existence and uniqueness of discrete solutions is proved. In Section \ref{sec:error} we derive a priori error estimates for the velocity, pressure and temperature of the problem which we contrast with a series of numerical tests reported in Section \ref{sec:numericos}.
\section{Notation and preliminary remarks}
\label{sec:notation_and_preliminaries}
Let us establish the notation and the framework within which we will work.

 \subsection{Notation}
 \label{sec:notation}
 In this paper, $\Omega$ is an open and bounded polygonal domain of $\mathbb{R}^{2}$ with Lipschitz boundary $\partial \Omega$. If $\mathscr{X}$ and $\mathscr{Y}$ are normed vector spaces, we write $\mathscr{X} \hookrightarrow \mathscr{Y}$ to denote that $\mathscr{X}$ is continuously embedded in $\mathscr{Y}$. We denote by $\mathscr{X}'$ and $\|\cdot\|_{\mathscr{X}}$ the dual and the norm of $\mathscr{X}$, respectively.
For $ p \in (1,\infty)$, we denote by $q \in (1,\infty)$ its H\"older \emph{conjugate}, which is such that $1/p + 1/q = 1$. 
We use the standard notation for Lebesgue and Sobolev spaces \cite{MR2424078}. The spaces of vector-valued functions and the vector-valued functions themselves are denoted by bold letters. In particular, we use the following notation: $\H:=\H_0^1(\O)$, $\bV:=\H(\div,\O)$, and 
$$\Q=\L_0^2(\O)\coloneqq \left\lbrace \textsf{q}\in\L^2(\O) :  \int_{\O}\textsf{q}=0\right\rbrace.$$
We also introduce the subspaces of $\bV$ given by                                                                                                                                                                                                                                                                                                                                                                                                                                                                                                                                                                                                                                                                                                                                                                                                                      $\bV_0 \coloneqq  \{ \bv \in \bV : \, \bv\cdot\bn|_{\Gamma} = 0\}$ and $\bZ\coloneqq \{\bv \in \bV_0 : \div\bv = 0\}.$ The spaces $\bV$ and $\bVV$ are endowed, respectively, with the following norms $$\|\bv\|_{\H(\div,\O)} \coloneqq  \left( \|\bv\|_{\L^2(\O)^2}^2+\|\div\bv\|_{\L^2(\O)}^2\right)^{1/2}\quad\forall\bv\in\bV,$$ $$\|S\|_{\H^1(\O)} \coloneqq  \left( \|S\|_{\L^2(\O)}^2 + |S|_{\H^1(\O)}^2\right)^{1/2}\quad\forall S\in\H,$$ where $|S|_{\H^1(\O)}\coloneqq \|\nabla S\|_{\L^2(\O)^2}$ denotes the seminorm of $S\in\H$.

On the other hand, we consider the following Sobolev embedding: for $p\geq 1$, there exists $\textsf{C}_p>0$ such that 
\begin{equation}\label{eq:Poincare}
\| S \|_{\L^p(\O)} \leq \textsf{C}_p| S |_{\H^1(\O)} \qquad \forall S \in \H,
\end{equation}
and we observe that for $p=2$, inequality \eqref{eq:Poincare} is nothing but Poincar\'e's inequality. Finally, we also use the continuous embedding (\cite[Theorem 4.12, Case $\mathrm{A}$]{MR2424078}):
\begin{equation}
\label{eq:embedding}
\W^{1,q}(\O) \hookrightarrow \L^{\infty}(\Omega), \quad \forall q > 2,
\end{equation}
which will be useful in the forthcoming analysis.


\subsection{Data assumptions}
\label{sec:data}

We make the following assumptions on the viscosity $\nu(\cdot)$ and the thermal diffusion coefficient $\kappa(\cdot)$:

\begin{itemize}
\item[\textbf{A0})] $\nu(\cdot)$ is strictly positive, bounded, and Lipschitz continuous, i.e., there exist constants $\nu_{*}, \nu^{*}, \nu_{\text{lip}} > 0$ such that
\begin{equation*}
\begin{split}
0<\nu_* \leq \nu(\textsf{x}) \leq \nu^*,
\qquad
|\nu(\textsf{y}) - \nu(\textsf{z})| \leq \nu_{\text{lip}}|\textsf{y} - \textsf{z}|
\qquad \forall \textsf{x},\textsf{y},\textsf{z} \in \mathbb{R}.
\end{split}
\end{equation*}
\item[\textbf{A1})] $\kappa(\cdot)$ is strictly positive, bounded, and Lipschitz continuous, i.e., there exist constants $\kappa_{*}, \kappa^{*}, \kappa_{\text{lip}} > 0$ such that
\begin{equation*}
0 < \kappa_* \leq \kappa(\textsf{x}) \leq \kappa^*,
\qquad 
|\kappa(\textsf{y}) - \kappa(\textsf{z})| \leq \kappa_{\text{lip}}|\textsf{y} - \textsf{z}|
\qquad \forall \textsf{x},\textsf{y},\textsf{z} \in \mathbb{R}.
\end{equation*}
\end{itemize} 

\section{Weak formulation}\label{sec:DH}
We now introduce a weak formulation for  system \eqref{eq:DH} and review existence and uniqueness results. First, it is important to observe that, as is presented in \cite[Section 2.2]{BernardiDH}, since $\bu\in\L^2(\O)^2$, we only have $\bu\cdot\nabla T\in\L^1(\O)$. Therefore, we need to test the nonlinear term by an $\L^\infty$ function. Since we need to find the temperature in the space $\H$, to simplify the analysis we just choose the space $\H\cap\L^{\infty}(\O)$ for the test functions. Under these considerations, the weak formulation of \eqref{eq:DH} reads as follows: Given $\bb{f}\in\L^2(\O)^2$ and $g\in\L^{2}(\O)$, find $(\bu,\textsf{p},T)\in\bV_0\times\Q\times\H$ such that
\begin{equation}\label{eq:DHW}
\left\{\begin{array}{rll}
\displaystyle \int_{\O}\nu(T)\bu\cdot\bv-\int_{\O}\textsf{p}\div\bv&=\displaystyle \int_\O\bb{f}\cdot\bv,\quad \forall\bv\in\bV_0, \\
\displaystyle \int_{\O}\textsf{q}\div\bu&=0,\quad &\forall \textsf{q}\in\Q\\
\displaystyle \int_\O \kappa(T)\nabla T\cdot\nabla S + \int_\O (\bu\cdot\nabla T)S &= \displaystyle \int_{\O}gS,\quad &\forall S\in\H. \\
\end{array}\right.
\end{equation}

To perform the analysis of \eqref{eq:DHW}, we define the following forms 
\begin{equation}\label{eq:bilinear_form_a}
a: \H\times\bV_0\times\bV_0 \rightarrow \mathbb{R}, \quad a(T;\bu,\bv)\coloneqq \displaystyle \int_{\O}\nu(T)\bu\cdot\bv,
\end{equation} 
\begin{equation}\label{eq:bilinear_form_b}
b:\bV_0\times\Q \rightarrow \mathbb{R}, \quad b(\bv,\textsf{q})\coloneqq -\int_{\O}\textsf{q}\div\bv,
\end{equation} 
\begin{equation}\label{eq:bilinear_form_fraka}
\mathfrak{a} : \H\times\H\times\H \rightarrow \mathbb{R}, \quad \mathfrak{a}(R;T,Z)\coloneqq \int_\O \kappa(R)\nabla T \cdot \nabla Z,
\end{equation} 
\begin{equation}\label{eq:bilinear_form_frakc}
\mathfrak{c}:\bV_0\times\H\times(\H\cap\L^{\infty}(\O)) \rightarrow \mathbb{R}, \quad \mathfrak{c}(\bu;T,S)\coloneqq \int_\O (\bu\cdot\nabla T)S.
\end{equation}
On the other hand, we define the functionals $$F : \bV_0 \rightarrow \mathbb{R}, \quad F(\bv)\coloneqq  \int_\O \bb{f}\cdot\bv,\quad\forall\bv\in\bV_0$$ $$G:\H\cap\L^{\infty}(\O), \quad G(S)\coloneqq \int_{\O}gS,\quad\forall S\in\H.$$

With these forms and functionals at hand, we rewrite system \eqref{eq:DHW} as follows: Find $(\bu,\textsf{p},T)\in\bV_0\times\Q\times\H$ such that 
\begin{equation}\label{eq:DHW2}
\left\{\begin{array}{rll}
a(T;\bu,\bv)+b(\bv,\textsf{p})&=\displaystyle F(\bv), &\quad \forall \bv \in \bV_0, \\
b(\bu,\textsf{q})&=0, &\quad \forall \textsf{q} \in \Q, \\
\mathfrak{a}(T;T,S) + \mathfrak{c}(\bu;T,S) &= G(S), &\quad \forall S\in \H\cap\L^{\infty}(\O). \\
\end{array}\right.
\end{equation}
Let us notice that  $a(\cdot;\cdot,\cdot)$, $b(\cdot,\cdot)$, $\mathfrak{a}(\cdot;\cdot,\cdot)$ and $\mathfrak{c}(\cdot;\cdot,\cdot)$ satisfies the following stability properties: 
\begin{itemize}[leftmargin=*]
\item For any $X \in \H$, $a(X;\cdot,\cdot)$ is a coercive and continuous bilinear form: for every $\bv,\bw \in \bV$ we have
\begin{equation*}\label{eq:a_bound}
a(X;\bv,\bw) \leq \nu^*\|\bv\|_{\L^2(\O)^2}\|\bw\|_{\L^2(\O)^2},
\quad
a(X;\bv,\bv) \geq \nu_{*}\|\bv\|_{\L^2(\O)^2}^2.
\end{equation*}
\item For any $X \in \H$, $\mathfrak{a}(X;\cdot,\cdot)$ is a coercive and continuous bilinear form: for every $R, S \in \H$, we have
\begin{equation*}\label{eq:atemp_bound}
\mathfrak{a}(X;R,S) \leq \kappa^*|R|_{\H^1(\O)}|S|_{\H^1(\O)},
\quad
\mathfrak{a}(X;S,S) \geq \kappa_{*}|S|_{\H^1(\O)}^2.
\end{equation*}
\item $\mathfrak{c}(\cdot;\cdot,\cdot)$ is continuous: For every $\bv \in \bV$, $R \in \H$ and $S \in \H\cap\L^\infty(\O)$, we have
\begin{equation*}\label{eq:ctemp_bound}
\mathfrak{c}(\bv;R,S) \leq \|\bv\|_{\L^2(\O)^2} |R|_{\H^1(\O)}\|S\|_{\L^{\infty}(\O)},
\end{equation*}
\item $b(\cdot,\cdot)$ is continuous: For every $\textsf{q} \in \Q$ and $\bv \in \bV$, we have
\begin{equation*}\label{eq:b_bound}
b(\bv,\textsf{q}) \leq \|\bv\|_{\H(\div,\O)}\|\textsf{q}\|_{\L^2(\O)}.
\end{equation*}
Moreover, $b(\cdot,\cdot)$ satisfies an inf-sup condition: There exists a constant $\beta>0$ such that
\begin{equation}\label{eq:infsupcontinua}
\underset{\bv \in \bV}{\sup} \; \dfrac{b(\bv,\textsf{q})}{\|\bv\|_{\H(\div,\O)}} \geq \beta \|\textsf{q}\|_{\L^2(\O)}
\quad
\forall \textsf{q} \in \Q.
\end{equation} 
\end{itemize}

\subsection{Reduced formulation} As is presented in \cite[Section 2.2]{BernardiDH}, an equivalent formulation can be obtained for the nonlinear Darcy-heat system \eqref{eq:DHW2},  where the velocity field $\bu$ and the pressure $\textsf{p}$ of the fluid are functions of the temperature $T$, i.e, $\bu=\bu(T)$ and $\textsf{p}=\textsf{p}(T)$. This formulation is as follows: Find $T\in\H$ such that 
\begin{equation}\label{eq:reducedHeat}
\mathfrak{a}(T;T,S)+\mathfrak{c}(\bu(T);T,S)=G(S), \quad \forall S\in\H\cap\L^{\infty}(\O),
\end{equation}
where $\bu(T)$ is the solution of: Find $(\bu(T),\textsf{p}(T))\in\bV_0\times\Q$ such that 
\begin{equation}\label{eq:reducedDarcy}
\left\{\begin{array}{rll}
a(T;\bu(T),\bv)+b(\bv,\textsf{p}(T))&=\displaystyle F(\bv), \quad &\forall \bv\in\bV_0 \\
b(\bu(T),\textsf{q})&=0, \quad &\forall \textsf{q}\in\Q,
\end{array}\right.
\end{equation}
which is nothing else but the Darcy problem.
\subsection{Existence, stability, and uniqueness results}
Inspired by \cite{BernardiDH}, in this section  we present some important results    for existence, uniqueness and stability of solutions for the nonlinear system \eqref{eq:DHW2}.

We observe that testing by $\bv=\bu(T)$ the first equation of \eqref{eq:reducedDarcy} and using the second equation of it, and invoking the inf-sup condition \eqref{eq:infsupcontinua}, we obtain the following stability bounds \cite[equation (2.22)]{BernardiDH}):
\begin{equation}\label{eq:stabilityDarcy}
\begin{split}
\|\bu(T)\|_{\L^2(\O)^2}&\leq\dfrac{1}{\nu_*}\|\bb{f}\|_{\L^2(\O)^2}, \quad \|\sqrt{\nu(T)}\bu(T)\|_{\L^2(\O)^2}\leq\dfrac{1}{\sqrt{\nu_*}}\|\bb{f}\|_{\L^2(\O)^2}, \\
\|\textsf{p}(T)\|_{\L^2(\O)} &\leq \dfrac{1}{\beta}\left(1 + \dfrac{\nu^*}{\nu_*}\right)\|\bb{f}\|_{\L^2(\O)^2}.
\end{split}
\end{equation}

With the bounds on  \eqref{eq:stabilityDarcy} at hand, it is possible to prove the following convergence result (see \cite[Lemma 2.1]{BernardiDH}).
\begin{proposition}
\label{pro:existence3}
Let $\nu(\cdot)$ and $\kappa(\cdot)$ satisfy \textbf{A0)} and \textbf{A1)}, respectively. Let $\{T_k\}_{k\in\mathbb{N}} \subset \L^2(\Omega)$ be a sequence of functions that converges strongly to $T \in \L^2(\Omega)$. Then the sequence $\{(\mathbf{u}(T_k), \textsf{p}(T_k))\}_{k\in\mathbb{N}} \subset \mathbf{V}_0 \times \Q$ converges weakly to $(\mathbf{u}(T), \textsf{p}(T)) \in \mathbf{V}_0 \times \Q$ and 
\begin{equation*}
\begin{split}
\lim_{k\rightarrow+\infty} \sqrt{\nu(T_k)}\bu(T_k) &= \sqrt{\nu(T)}\bu(T) \;\; \text{strongly in} \;\; \L^2(\O)^2, \\
\lim_{k\rightarrow+\infty} \textsf{p}(T_k) &= \textsf{p}(T) \;\; \text{strongly in} \;\; \L^2(\Omega).
\end{split}
\end{equation*}
\end{proposition}

Next, we present an argument to show the existence of weak solutions of system \eqref{eq:reducedHeat} based in a separability argument. We refer to \cite[Section 2.3]{BernardiDH} for further details. 

Let $q>2$ and consider the space $\W_0^{1,q}(\O)\coloneqq \{S\in\W^{1,q}(\O) : S|_{\Gamma}=0\}$.
The argument that we propose is based in a discretization of \eqref{eq:reducedHeat} for which we prove existence of solutions. First, since $\W_0^{1,q}(\O)$ is separable,  it has a countable basis which we denote by $\Theta\coloneqq \{\theta_n\}_{n\in\mathbb{N}}$. Let $\Theta_k$ be the finite-dimensional space spanned by the basis functions $\{\theta_1,\ldots,\theta_k\}$. A discrete formulation of \eqref{eq:reducedHeat} reads as follows: find $T_k\in\Theta_k$ such that for all $j \in \{1,\ldots,k\}$
\begin{equation}\label{eq:Heatseparable}
\int_\O \kappa(T_k)\nabla T_k\cdot\nabla \theta_j + \int_\O (\bu(T_k)\cdot\nabla T_k)\theta_j = \int_\O g\theta_j,
\end{equation}
where $(\bu(T_k),\textsf{p}(T_k))$ solves \eqref{eq:reducedDarcy} with $T=T_k$. Thanks to the continuous embedding \eqref{eq:embedding}, the nonlinear term is well defined. Now, consider the operator $\mathcal{A}:\Theta_k\rightarrow\Theta_k$ such that for every $T_k \in \Theta_k$, $\mathcal{A}(T_k)$ is the solution of the following problem: Find $\mathcal{A}(T_k)\in\Theta_k$ such that 
\begin{equation}\label{eq:OperatorA}
\int_\O \nabla \mathcal{A}(T_k)\cdot\nabla S_k = \int_\O \kappa(T_k)\nabla T_k\cdot\nabla S_k + \int_\O (\bu(T_k)\cdot\nabla T_k)S_k - \int_\O gS_k,
\end{equation}
for all $S_k\in\Theta_k$. It is easy to check that problem \eqref{eq:OperatorA} is well posed, and therefore $\mathcal{A}$ is well defined. On the other hand, since $\Theta_k$ is finite dimensional, and invoking bounds \eqref{eq:stabilityDarcy}, we obtain the continuity of $\mathcal{A}$. Now, using integration by parts we observe that $$\int_\O (\bu(T_k)\cdot\nabla T_k)T_k = 0 \quad \forall T_k\in\Theta_k,$$ and therefore 
\begin{equation*}
\begin{split}
\int_\O \nabla \mathcal{A}(T_k)\cdot\nabla T_k &= \int_\O \kappa(T_k)|\nabla T_k|^2 - \int_\O gT_k \\
&\geq \kappa_*|T_k|_{\H^1(\O)}^2 - \|g\|_{\L^2(\O)}\|T_k\|_{\L^2(\O)} \\
&\geq \kappa_*|T_k|_{\H^1(\O)}^2 - \textsf{C}_2\|g\|_{\L^2(\O)}|T_k|_{\H^1(\O)}.
\end{split}
\end{equation*}
Then, we obtain for every $T_k\in\Theta_{\text{Ex}}\coloneqq \left\lbrace T_k\in\Theta_k : |T_k|_{\H^1(\O)} = \dfrac{\textsf{C}_2}{\kappa_*}\|g\|_{\L^2(\O)}\right\rbrace$ we have 
$$\int_\O \nabla \mathcal{A}(T_k)\cdot\nabla T_k \geq 0,\qquad k\in\mathbb{N}.$$  
Therefore, invoking Brouwer's fixed point Theorem (\cite[Chap. IV, Corollary 1.1]{MR851383}) we obtain that problem \eqref{eq:Heatseparable} has at least one solution $T_k\in\Theta_k$ which satisfies the stability bound 
\begin{equation}\label{eq:cotaTunif}
|T_k|_{\H^1(\O)} \leq \dfrac{\textsf{C}_2}{\kappa_*}\|g\|_{\L^2(\O)}. 
\end{equation}

This result allows us to show existence of solutions of equation \eqref{eq:reducedHeat} using the density of $\Theta\subset \W_0^{1,q}(\O)$, as is stated in the following theorem.
\begin{theorem}\label{teo:existencia}
Assume that $\nu(\cdot)$ and $\kappa(\cdot)$ satisfy properties \textbf{A0}) and \textbf{A1}), respectively. Then, for any $\bb{f}\in\L^2(\O)^2$ and $g\in\L^2(\O)$, problem \eqref{eq:reducedHeat} has at least one solution $T\in\H$.
\end{theorem}

\begin{proof}
From \eqref{eq:cotaTunif} we observe that sequence $\{T_k\}_{k\in\mathbb{N}}\subset\H$ converges weakly to some function $T\in\H$, and from the continuous embedding \eqref{eq:Poincare} we obtain that $\{T_k\}_{k\in\mathbb{N}}$ converges strongly to $T$ in $\L^p(\O)$ for $p\geq 1$. Then, invoking Proposition \ref{pro:existence3}, we obtain that the sequence $\{(\bu(T_k),\textsf{p}(T_k))\}_{k\in\mathbb{N}}\subset \bV_0\times\Q$ converges weakly to $(\bu(T),\textsf{p}(T))\in\bV_0\times\Q$, $\{\sqrt{\nu(T_k)}\bu(T_k)\}_{k\in\mathbb{N}}$ converges strongly to $\sqrt{\nu(T)}\bu(T)\in\L^2(\O)^2$ and $\{\textsf{p}(T_k)\}_{k\in\mathbb{N}}$ converges strongly to $\textsf{p}(T)\in\L^2(\O)$. 

Now, for a fixed index $j$ in \eqref{eq:Heatseparable}, an integration by parts yields to
\begin{equation*}
\int_{\O}(\bu(T_k)\cdot\nabla T_k)\theta_j = -\int_\O (\bu(T_k)\cdot\nabla\theta_j)T_k.
\end{equation*} 
On the other hand, since $\nabla\theta_j \in \L^q(\O)$ for $q>2$, the strong convergence of $\{T_k\}_{k\in\mathbb{N}}$ to $T\in\L^p(\O)$ for $p\geq1$ and H\"{o}lder's inequality, implies that $\{T_k\nabla\theta_j\}_{k\in\mathbb{N}}$ converges strongly to $T\nabla\theta_j\in\L^2(\O)^2$. This strong convergence, combined with the weak convergence of $\{\bu(T_k)\}_{k\in\mathbb{N}}$ to $\bu(T)\in\L^2(\O)^2$ implies that 
\begin{equation*}
\lim_{k\to+\infty} \int_{\O} (\bu(T_k)\cdot\nabla T_k)\theta_j = - \lim_{k\to+\infty} \int_{\O} (\bu(T_k)\cdot\nabla \theta_j)T_k = -\int_{\O} (\bu(T)\cdot\nabla \theta_j)T.
\end{equation*}
On the other hand, following the proof of \cite[Lemma 2.1]{BernardiDH}, it is possible to show that $\{\kappa(T_k)\nabla T_k\}_{k\in\mathbb{N}}$ converges weakly to $\kappa(T)\nabla T \in \L^2(\O)^2$. Thus we obtain
\begin{equation*}
\lim_{k\to+\infty} \int_{\O}\kappa(T_k)\nabla T_k\cdot\nabla\theta_j = \int_\O \kappa(T)\nabla T\cdot\nabla\theta_j.
\end{equation*}
Then, passing to the limit in \eqref{eq:Heatseparable}, for each $j\in\mathbb{N}$ we obtain
\begin{equation*}
\int_\O \kappa(T)\nabla T\cdot\nabla\theta_j - \int_{\O} (\bu(T)\cdot\nabla \theta_j)T = \int_\O g\theta_j.
\end{equation*}
Now, using the density of the set $\Theta\subset\W_0^{1,q}(\O)$ we have
\begin{equation*}
\int_\O \kappa(T)\nabla T\cdot\nabla S - \int_{\O} (\bu(T)\cdot\nabla S)T = \int_\O gS \quad \forall S\in\W_0^{1,q}(\O),
\end{equation*}
implying that in the distributional sense,  we have the following equivalency:
\begin{equation*}
-\div(\kappa(T)\nabla T) + \div(\bu(T)T) = g\, \Longleftrightarrow  -\div(\kappa(T)\nabla T) + \bu(T)\cdot\nabla T = g.
\end{equation*}
In particular, we observe that $\bu(T)\cdot\nabla T \in \H'\coloneqq \H^{-1}(\O)$. Then, taking duality with $S\in\H\cap\L^{\infty}(\O)$ we obtain (\cite{BernardiDH}):
\begin{equation*}
\int_\O \kappa(T)\nabla T\cdot\nabla S + \int_\O  (\bu(T)\cdot\nabla T)S = \int_\O gS \quad \forall S\in\H,
\end{equation*}
This concludes the proof.
\end{proof}

The following task is to prove uniqueness of solutions of \eqref{eq:DHW}. 
First, given $\bz \in \bV_0$, let us consider the problem: Find $T\in\H$ such that 
\begin{equation}\label{eq:Calordivfree}
\mathfrak{a}(T;T,S) + \mathfrak{c}(\bz;T,S) = G(S), \quad \forall S\in\H\cap\L^{\infty}(\O).
\end{equation}

As is presented in \cite[Section 2.4]{BernardiDH}, a natural choice for the test function is $S=T$, but this is  not possible since $T$ not necessarily belongs to $\L^{\infty}(\O)$. We observe that the existence of solutions of \eqref{eq:Calordivfree} is proved by using similar arguments that those in Theorem \ref{teo:existencia}.
Next, given $\ell > 0$, we define the function $\tau_\ell : \mathbb{R} \rightarrow \mathbb{R}$ where for $t\in\mathbb{R}$, $\tau_\ell(t)$ is given by
$$\tau_\ell(t)=\left\{\begin{array}{lcc} \displaystyle t & \text{if} & |t|\leq\ell, \\ \ell\,\text{sgn}(t) &\text{if}& |t|>\ell. \end{array} \right.$$
It is easy to check that $\tau_\ell$ is continuous on $\mathbb{R}$ and therefore integrable. Let $\sigma_\ell$ its primitive, given by
$$\sigma_{\ell}(t) = \int_0^t \tau_\ell(s)ds.$$

Note that $\tau_\ell \in \W^{1,\infty}(\mathbb{R})$ and for any $S\in\H$, $\tau_{\ell}(S)\in\H$ and almost everywhere in $\O$ we have
$$\nabla\tau_\ell(S)=\left\{\begin{array}{lcc} \displaystyle \nabla S & \text{if} & |S|\leq\ell, \\ 0 &\text{if}& |S|>\ell. \end{array} \right.$$

On the other hand, $\sigma_{\ell}$ is Lipschitz continuous, piecewise $\mathcal{C}^1(\mathbb{R})$, and for any $S\in\H$, $\sigma_{\ell}(S)\in\H$. With these functions at hand, we establish uniqueness of solutions of \eqref{eq:Calordivfree}. 

\begin{theorem}\label{teo:caloruniq}
Let $\kappa(\cdot)$ satisfying assumption \textbf{A1}) and let $\bz\in\bZ$. Assume that \eqref{eq:Calordivfree} has a solution $T_1\in\H$ such that $T_1\in\W^{1,\infty}(\O)$ and
\begin{equation}\label{eq:smalldataTemp}
\kappa_*^{-1}\kappa_{\text{lip}}\textsf{C}_2\|\nabla T_1\|_{\L^{\infty}(\O)^2} < 1.
\end{equation} 
Then, $T_1$ is  unique and satisfies
\begin{equation}\label{eq:cotatemp}
|T_1|_{\H^1(\O)} \leq \dfrac{\textsf{C}_2}{\kappa_*}\|g\|_{\L^2(\O)}.
\end{equation}
\end{theorem}
\begin{proof}
First, since $\tau_\ell(S)\in\L^{\infty}(\O)$ for $S\in\H$, testing in \eqref{eq:Calordivfree} with $S=\tau_\ell(T_1)$ we obtain
\begin{equation*}
\int_\O \kappa(T_1)\nabla T_1\cdot \nabla\tau_\ell(T_1) + \int_\O (\bz\cdot\nabla T_1)\tau_\ell(T_1) = \int_\O g\tau_\ell(T_1).
\end{equation*}
Now, observe that from the definition of $\nabla\tau_\ell$ we obtain
\begin{equation*}
\int_{\O}\kappa(T_1)\nabla T_1\cdot\nabla \tau_\ell(T_1) \geq \kappa_*\int_\O \nabla T_1\cdot \nabla\tau_\ell(T_1) = \kappa_*|\tau_\ell(T_1)|_{\H^1(\O)}^2.
\end{equation*}
On the other hand, its easy to check that $\nabla\sigma_\ell(T_1)=\tau_\ell(T_1)\nabla T_1$. Then,  using integration by parts we obtain
\begin{equation*}
\int_\O (\bz\cdot\nabla T_1)\tau_\ell(T_1) = \int_\O \bz\cdot\nabla\sigma_\ell(T_1) = -\int_\O (\div\bz )\sigma_\ell(T_1) = 0.
\end{equation*}
Next, using Cauchy-Schwarz inequality and \eqref{eq:Poincare} we obtain
\begin{equation*}
\kappa_*|\tau_\ell(T)|_{\H^1(\O)}^2 \leq \|g\|_{\L^2(\O)}\|\tau_{\ell}(T)\|_{\L^2(\O)} \leq \textsf{C}_2\|g\|_{\L^2(\O)}|\tau_\ell(T)|_{\H^1(\O)},
\end{equation*}
and therefore we have
\begin{equation*}
|\tau_\ell(T)|_{\H^1(\O)} \leq \dfrac{\textsf{C}_2}{\kappa_*}\|g\|_{\L^2(\O)}.
\end{equation*}
From the strong convergence of $\tau_{\ell}(T_1)$ to $T_1$ in $\H^1(\O)$, we conclude that $T_1$ satisfies \eqref{eq:cotatemp}. 

Let us focus now on the uniqueness.  Let $T_2\in\H$ be another solution of \eqref{eq:Calordivfree} and setting $T\coloneqq T_1-T_2$, its easy to check that 
\begin{equation*}
\int_\O \kappa(T_2)\nabla T\cdot\nabla S + \int_\O (\bz\cdot\nabla T)S = \int_\O (\kappa(T_2)-\kappa(T_1))\nabla T_1\cdot\nabla S,\quad\forall S\in\H\cap\L^{\infty}(\O).
\end{equation*}
Next, choosing again $S=\tau_\ell(T)$ in the above equation, combined with the fact that $\nabla\sigma_{\ell}(T)=\tau_{\ell}(T)\nabla T$ and $\bz\in\bZ$ (to neglect the corresponding convective term), we obtain 
\begin{equation*}
\int_\O \kappa(T_2)\nabla T\cdot \nabla \tau_{\ell}(T) = \int_\O (\kappa(T_2)-\kappa(T_1))\nabla T_1\cdot\nabla \tau_\ell(T).
\end{equation*}
Then, invoking  assumption \textbf{A1}), using H\"{o}lder's inequality, and \eqref{eq:Poincare}, we deduce that
\begin{align*}
\kappa_*|\tau_\ell(T)|_{\H^1(\O)}^2 &\leq \kappa_{\text{lip}}\int_\O |T||\nabla T_1||\nabla\tau_\ell(T)| \\
&\leq \kappa_{\text{lip}}\|\nabla T_1\|_{\L^{\infty}(\O)^2}|\tau_\ell(T)|_{\H^1(\O)}\|T\|_{\L^2(\O)} \\
&\leq \kappa_{\text{lip}}\textsf{C}_2\|\nabla T_1\|_{\L^\infty(\O)^2}|\tau_{\ell}(T)|_{\H^1(\O)}|T|_{\H^1(\O)}.
\end{align*}
The above estimate in conjunction with the strong convergence of $\tau_{\ell}(T)$ to $T$ in $\H^1(\O)$ implies that 
\begin{equation*}
\kappa_*|T|_{\H^1(\O)} \leq \kappa_{\text{lip}}\textsf{C}_2\|\nabla T_1\|_{\L^\infty(\O)^2}|T|_{\H^1(\O)},
\end{equation*}
where, assuming that \eqref{eq:smalldataTemp} holds we have that $T=0$ and therefore, $T_1=T_2$. This concludes the proof.
\end{proof}

Now, we are in position to prove uniqueness of solutions of coupled system \eqref{eq:DHW2}.

\begin{theorem}\label{eq:uniqcont}
Let us suppose that  $\nu(\cdot)$ and $\kappa(\cdot)$ satisfy assumptions \textbf{A0}) and \textbf{A1}), respectively. In addition to the assumptions of Theorem \ref{teo:existencia}, assume that system \eqref{eq:DHW2} has a solution $(\bu_1,\textsf{p}_1,T_1)\in\bV_0\times\Q\times\H$ such that $T_1\in\W^{1,\infty}(\O)$, $\bu_1\in\L^r(\O)^2$ for some $r>2$ and the following condition holds
\begin{equation}
\label{eq:uniqDHW2} 
\dfrac{\kappa_* - \textsf{C}_2\kappa_{\text{lip}}\|\nabla T_1\|_{\L^{\infty}(\O)^2}}{\nu_*^{-1}\textsf{C}_{\frac{2r}{r-2}}\nu_{\text{lip}}\|\bu_1\|_{\L^r(\O)^2}\|T_1\|_{\L^{\infty}(\O)}} < 1,
\end{equation}
where $\textsf{C}_{\frac{2r}{r-2}}>0$ is the best constant of the Sobolev embedding $\H_0^1(\O) \hookrightarrow \L^{\frac{2r}{r-2}}(\O)$. Then,  
$(\bu_1,\textsf{p}_1,T_1)\in\bV_0\times\Q\times\H$  is unique.
\end{theorem}
\begin{proof}
Let $(\bu_2,\textsf{p}_2,T_2)\in\bV_0\times\Q\times\H$ be another solution of \eqref{eq:DHW2}. With these solutions at hand,  set $(\overline{\bu},\overline{\textsf{p}},\overline{T})\coloneqq (\bu_1-\bu_2,\textsf{p}_1-\textsf{p}_2,T_1-T_2)$. Elementary algebraic calculations yield to the following system
\begin{equation}\label{eq:DHW2reducido}
\left\{\begin{array}{rlc}
a(T_1;\bu_1,\bv)-a(T_2;\bu_1,\bv)+a(T_2;\overline{\bu},\bv)&=0, \\
\mathfrak{a}(T_1;T_1,S)-\mathfrak{a}(T_2;T_1,S) + \mathfrak{c}(\overline{\bu};T_1,S) + \mathfrak{a}(T_2;\overline{T},S) + \mathfrak{c}(\bu_2;\overline{T},S)&= 0, \\
\end{array}\right.
\end{equation}
which holds for every $(\bv,S)\in\bZ\times(\H\cap\L^{\infty}(\O))$. We observe that  for all $\textsf{q}\in\Q$ and $i=1,2$, there  holds that $b(\bu_i,\textsf{q})=0$. 
Testing the first equation of \eqref{eq:DHW2reducido} with $\bv=\overline{\bu}$, using the coercivity of $a(T_2;\cdot,\cdot)$ and assumption \textbf{A0}) we have
\begin{align*}
\nu_*\|\overline{\bu}\|_{\L^2(\O)^2}^2 &\leq |a(T_1;\bu_1,\overline{\bu})-a(T_2;\bu_1,\overline{\bu})| \\
&\leq \nu_{\text{lip}}\int_{\O} |\overline{T}||\bu_1\cdot\overline{\bu}| \\
&\leq \nu_{\text{lip}}\|\overline{T}\|_{\L^{\frac{2r}{r-2}}(\O)}\|\bu_1\|_{\L^r(\O)^2}\|\overline{\bu}\|_{\L^2(\O)^2} \\
&\leq \textsf{C}_{\frac{2r}{r-2}}\nu_{\text{lip}}|\overline{T}|_{\H^1(\O)}\|\bu_1\|_{\L^r(\O)^2}\|\overline{\bu}\|_{\L^2(\O)^2},
\end{align*}
and therefore, we deduce that 
\begin{equation}\label{eq:u1}
\|\overline{\bu}\|_{\L^2(\O)^2} \leq \nu_*^{-1}\textsf{C}_{\frac{2r}{r-2}}\nu_{\text{lip}}|\overline{T}|_{\H^1(\O)}\|\bu_1\|_{\L^r(\O)^2}.
\end{equation}

On the other hand, since $T_1\in\L^{\infty}(\O)$, we derive from the integration by parts formula the following identity
\begin{equation*}
\int_{\O} (\overline{\bu}\cdot\nabla T_1)S = -\int_{\O} (\overline{\bu}\cdot\nabla S)T_1.
\end{equation*}
Now, testing the  second equation of system \eqref{eq:DHW2reducido} with $S=\tau_\ell(\overline{T})$ we deduce that $\mathfrak{c}(\bu_2;\overline{T},\tau_\ell(\overline{T}))=0$.
Then, using the Cauchy-Schwarz inequality and assumption \textbf{A1}) we obtain the following estimate
\begin{multline*}
\kappa_*|\tau_\ell(\overline{T})|_{\H^1(\O)}^2 \leq \kappa_{\text{lip}}\int_\O |\overline{T}||\nabla T_1\cdot\nabla\tau_\ell(\overline{T})| + \int_\O |\overline{\bu}\cdot\nabla \tau_\ell(\overline{T})||T_1| \\
\leq |\tau_\ell(\overline{T})|_{\H^1(\O)}\left(\kappa_{\text{lip}}\|\overline{T}\|_{\L^2(\O)}\|\nabla T_1\|_{\L^{\infty}(\O)^2}+\|\overline{\bu}\|_{\L^2(\O)^2}\|T_1\|_{\L^{\infty}(\O)}\right) \\
\leq |\tau_\ell(\overline{T})|_{\H^1(\O)}\left(\textsf{C}_2\kappa_{\text{lip}}|\overline{T}|_{\H^1(\O)}\|\nabla T_1\|_{\L^{\infty}(\O)^2}+\|\overline{\bu}\|_{\L^2(\O)^2}\|T_1\|_{\L^{\infty}(\O)}\right),
\end{multline*}
and therefore, since  $\tau_\ell(\overline{T})$ converge strongly  to $\overline{T}$, we conclude that 
\begin{equation*}
\kappa_*|\overline{T}|_{\H^1(\O)} \leq \textsf{C}_2\kappa_{\text{lip}}|\overline{T}|_{\H^1(\O)}\|\nabla T_1\|_{\L^{\infty}(\O)^2}+\|\overline{\bu}\|_{\L^2(\O)^2}\|T_1\|_{\L^{\infty}(\O)},
\end{equation*}
or equivalently 
\begin{equation}\label{eq:T1}
\left(\kappa_* - \textsf{C}_2\kappa_{\text{lip}}\|\nabla T_1\|_{\L^{\infty}(\O)^2}\right)|\overline{T}|_{\H^1(\O)} \leq \|\overline{\bu}\|_{\L^2(\O)^2}\|T_1\|_{\L^{\infty}(\O)}.
\end{equation}
Then, replacing \eqref{eq:u1} into \eqref{eq:T1} and using simple algebraic manipulations we obtain
\begin{equation*}
\left(\dfrac{\kappa_* - \textsf{C}_2\kappa_{\text{lip}}\|\nabla T_1\|_{\L^{\infty}(\O)^2}}{\nu_*^{-1}\textsf{C}_{\frac{2r}{r-2}}\nu_{\text{lip}}\|\bu_1\|_{\L^r(\O)^2}\|T_1\|_{\L^{\infty}(\O)}} - 1\right)|\overline{T}|_{\H^1(\O)} \leq 0,
\end{equation*}
where, assuming that \eqref{eq:uniqDHW2} holds, we obtain $\overline{T}=0$ and therefore $T_1=T_2$. Replacing $\overline{T}=0$ into \eqref{eq:u1} we obtain $\overline{\bu}=\boldsymbol{0}$ and then, $\bu_1=\bu_2$. Finally, from \eqref{eq:infsupcontinua}, we obtain easily that $\overline{\textsf{p}}=0$ so $\textsf{p}_1=\textsf{p}_2$. This concludes the proof.
\end{proof}

\section{A virtual  element approximation}
\label{sec:vem}
In this section, we introduce a virtual element approximation for the nonlinear system \eqref{eq:DH} and derive error estimates. For this purpose, we first introduce some notions and basic ingredients, which are inspired in \cite{BEIRAOMIXED}. From now on, we assume that $\Omega$ is a polygonal and bounded domain with Lipschitz boundary.

Let $\{\O_h\}_{h>0}$ be a sequence of decompositions of $\Omega$ into general polygonal elements $E$. Let  $h_E$ be  the diameter of the element $E$ and  $h\coloneqq \max \{ h_E: E\in\O_h \}$ be the mesh-size. For $E \in \O_h$, we denote by $|E|$ its area.

\subsection{Mesh regularity}
Following \cite{BEIRAOMIXED}, we make the following assumptions on the sequence $\{\O_h\}_{h>0}$: There exists $\varrho>0$ such that
\begin{itemize}
\item[\textbf{A2})] $E$ is star-shaped with respect to a ball $B_\textsf{r}$ with radius $\textsf{r} \geq \varrho h_{E}$
\item[\textbf{A3})] the distance between any two vertices of $E$ is $\ell \geq \varrho h_{E}$,
\end{itemize}

Now, given $m,n \in \mathbb{N}$ with $n\leq m$, $s>0$ and $E\in\O_h$, we introduce the sets
\begin{itemize}[leftmargin=*]
\item[$\bullet$] $\mathbb{P}_{m}(E)$ as the set of polynomials of degree $\leq m$ defined over $E$,
\item[$\bullet$] $\mathcal{G}_m(E) \coloneqq  \nabla \mathbb{P}_{m+1}(E)$,
\item[$\bullet$] $\mathcal{G}_m(E)^{\perp} \coloneqq $ the $\L^2(E)$ orthogonal of $\mathcal{G}_m(E)$ in $\mathbb{P}_m(E)^2$,
\item[$\bullet$] $\mathbb{P}_{m\setminus n}(E) \coloneqq $ the $\L^2(E)$ orthogonal to $\mathbb{P}_{n}(E)$ in $\mathbb{P}_{m}(E)$,
\item[$\bullet$] $\mathbb{P}_m(\O_h) \coloneqq  \{q \in \L^2(\O) : q|_E \in \mathbb{P}_m(E) \quad \forall E\in\O_h\}$,
\item[$\bullet$] $\H^{s}(\O_h) \coloneqq  \{v\in\L^2(\O) : v|_E\in\H^s(E) \quad \forall E\in\O_h\}$.
\end{itemize}

For $s\geq1$,  the space $\H^s(\O_h)$ can be endowed with the following broken seminorm
\begin{equation*}
|v|_{\H^s(\O_h)}^2 \coloneqq  \sum_{E\in\O_h} |v|_{\H^s(E)}^2, \quad \forall v\in\H^s(\O_h).
\end{equation*}
The vectorial counterparts of $\mathbb{P}_m(\O_h)$ and $\H^s(\O_h)$ will be denoted by $\mathbb{P}_m(\O_h)^2$ and $\H^s(\O_h)^2$, respectively. 

Next, we introduce the local $\L^2$-orthogonal projection $\Pi_m^{0,E}:\L^2(E) \longrightarrow \mathbb{P}_m(E)$, where for $S\in\L^2(E)$, $\Pi_m^{0,E}S$ is given as the solution of 
\begin{equation*}
\int_E (S-\Pi_m^{0,E}S)q_m= 0, \quad \forall q_m\in\mathbb{P}_m(E).
\end{equation*}
The vectorial version of $\Pi_m^{0,E}$ will be denoted by $\bb{\Pi}_m^{0,E}$. Then, we introduce the operator $\Pi_m^0 : \L^2(\Omega) \longrightarrow \mathbb{P}_m(\O_h)$ which is defined by 
\begin{equation*}
(\Pi_m^0 S)|_E \coloneqq \Pi_m^{0,E}(S|_E), \quad \forall S\in\L^2(\O), \quad \forall E\in\O_h,
\end{equation*}
and its vectorial counterpart will be denoted by $\bb{\Pi}_m^{0}$. For this operator, we have the following result, which is a classical approximation result for polynomials on star-shaped domains (see \cite{MR2373954}).
\begin{lemma}[Bramble-Hilbert]\label{prop:bramblehilbert}
Let $0\leq t\leq s\leq \ell + 1$ and $1\leq q, p\leq \infty$ such that $s-2/p>t-2/q$. Then, for any $E\in\O_h$,
\begin{equation*}
\underset{q_{\ell}\in\mathbb{P}_{\ell}(E)}{\inf} |S-q_\ell|_{\W^{t,q}(E)} \leq Ch_E^{s-t+2/q-2/p}|S|_{\W^{s,p}(E)} \quad \forall S\in\W^{s,p}(E).
\end{equation*}
\end{lemma}

We use the same notation for its vectorial counterpart $\bb{\Pi}_m^0$. Now, the assumptions on the sequence $\{\O_h\}_{h>0}$ allows us to derive the following result.
\begin{proposition}\label{prop:stabPi}
Under assumptions \textbf{A2)} and \textbf{A3}), given $1\leq p\leq +\infty$ and $m\geq 0$ an integer, the following estimate holds
\begin{equation*}
\|\Pi_m^{0,E}S\|_{\L^p(E)} \leq \mathrm{C}_p\|S\|_{\L^p(E)}, \quad \forall S\in\L^2(E)\cap\L^{p}(E).
\end{equation*}
\end{proposition}
\begin{proof}
Let $S\in\L^2(E)\cap\L^{p}(E)$ and $1\leq q \leq +\infty$ be the conjugate exponent of $p$. Since assumptions \textbf{A2}) and \textbf{A3}) hold, for each $E\in\O_h$ we have the following inverse estimate for polynomials (\cite[Theorem 4.5.3]{MR2373954}):
\begin{equation}\label{eq:inversepoly}
\|p\|_{\L^q(E)} \leq Ch_E^{\frac{2}{q}-\frac{2}{p}}\|p\|_{\L^p(E)},
\end{equation}
for all $p\in\mathbb{P}_\ell(E)$, $\ell\geq 0$. Then, using the definition of $\Pi_m^{0,E}$, H\"{o}lder and Cauchy-Schwarz inequalities, combined with estimate \eqref{eq:inversepoly}, yield for each $E\in\O_h$
\begin{align*}
\|\Pi_m^{0,E}S\|_{\L^{p}(E)}^2 &\leq Ch_E^{2\left(\frac{2}{p}-1\right)}\|\Pi_m^{0,E}S\|_{\L^2(E)}^2 \\
&\leq Ch_E^{2\left(\frac{2}{p}-1\right)}\|\Pi_m^{0,E}S\|_{\L^q(E)}\|S\|_{\L^p(E)} \\
&\leq Ch_E^{2\left(\frac{2}{p}-1\right) + \frac{2}{q}-\frac{2}{p}}\|\Pi_m^{0,E}S\|_{\L^p(E)}\|S\|_{\L^p(E)} \\
&\leq \mathrm{C}_p\|\Pi_m^{0,E}S\|_{\L^p(E)}\|S\|_{\L^p(E)},
\end{align*}
implying that $\|\Pi_m^{0,E}S\|_{\L^p(E)} \leq \mathrm{C}_p\|S\|_{\L^p(E)}$, for all $S\in\L^2(E)\cap\L^{p}(E)$.  This concludes the proof.
\end{proof}

Let us  introduce the $\H^{1}(E)$-seminorm projection $\Pi^{\nabla,E}_{m+1}: \H^1(E) \rightarrow \mathbb{P}_{m+1}(E)$, which is defined for every $S\in\H^1(E)$ as the solution of 
\begin{align*}
\label{eq:gradproj}
\left\{\begin{array}{rll}
  \ds\int_{E}\nabla q_{m+1}\cdot \nabla(S  - \Pi^{\nabla,E}_{m+1}S) &=& 0 \quad \forall q_{m+1} \in \mathbb{P}_{m+1}(E),\\\\
\ds\int_{\partial E} (S  - \Pi^{\nabla,E}_{m+1}S)&=&0.
\end{array}\right.
\end{align*}
In the same way to define $\Pi_m^0$, we introduce the operator $\Pi_{m+1}^{\nabla}:\H^1(\O_h)\longrightarrow\mathbb{P}_m(\O_h)$ which is defined by 
\begin{equation*}
(\Pi_{m+1}^\nabla S)|_{E} \coloneqq \Pi_{m+1}^{\nabla,E}(S|_E), \quad \forall S\in\H^1(\O_h), \quad \forall E\in\O_h.
\end{equation*}
\subsection{Virtual space for the velocity}
Let $k\geq0$ be an integer. We introduce the local virtual space $\bV_h^k(E)$ for the velocity variable, which is given by
\begin{multline*}
\bV_h^k(E)\coloneqq \{\bv_h\in\H(\div,E)\cap\H(\rot,E): \bv_h\cdot\bn|_e \in \mathbb{P}_k(e) \;\; \forall e\subset\partial E, \\
\div\bv_h \in \mathbb{P}_k(E)\;\;\text{and}\;\;\rot\bv_h \in \mathbb{P}_{k-1}(E)\}.
\end{multline*}

In what follows, we present some properties of the local virtual space $\bV_h^k(E)$. For further details, we refer to \cite{BEIRAOMIXED}.
\begin{itemize}
 \item [\textbf{(P1)}] \textbf{Polynomial inclusion:} $\mathbb{P}_k(E)^2 \subseteq \bV_h^k(E)$,
 \item [\textbf{(P2)}] \textbf{Degrees of freedom:} the following linear operators $\mathbf{D}_{\mathbf{V}}$ constitute a set of DoFs for the virtual element space $\bV_h^k(E)$:
 \begin{itemize}
 \item [$\mathbf{D}_{\mathbf{V}} \mathbf{1}$:] The edge moments 
 \begin{equation*}
 \int_{e} (\bv_h\cdot\bn)q_k, \quad \forall q_k\in\mathbb{P}_k(e), \quad \forall e\subset \partial E
 \end{equation*}
 \item [$\mathbf{D}_{\mathbf{V}} \mathbf{2}$:] The interior moments
 \begin{equation*}
 \int_E \bv_h\cdot \bb{g}_{k-1}, \quad \forall \bb{g}_{k-1}\in\mathcal{G}_{k-1}(E),
 \end{equation*}
 \item [$\mathbf{D}_{\mathbf{V}} \mathbf{3}$:] The interior moments
 \begin{equation*}
 \int_E \bv_h\cdot \bb{g}_{k}^{\perp}, \quad \forall \bb{g}_{k}^\perp\in\mathcal{G}_k(E)^\perp.
 \end{equation*}
 \end{itemize}
\end{itemize}
Then, the dimension of $\bV_h^k(E)$ is
\begin{equation*}
\text{dim}(\bV_h^k(E)) = (\text{number of edges of $E$})(k+1) + k^2 + 2k.
\end{equation*}
\begin{remark}
According to \cite[Section 3.2]{BEIRAOMIXED}, for each $E\in\O_h$ and $e\subset\partial E$, the DoFs $\mathbf{D}_{\mathbf{V}}$ allows us to compute $\div\bv_h$, $(\bv_h\cdot\bn)|_{e}$ and $\bb{\Pi}_k^{0,E}\bv_h$, $\forall \bv_h\in\bV_h^k(E)$.
\end{remark}
The global virtual space is given by
\begin{equation*}
\bV_h \coloneqq  \{\bv_h\in\bV_0 : \bv_h|_E \in \bV_h^k(E) \quad \forall E\in\O_h\},
\end{equation*}
and its  DoFs  are obtained by collecting the local ones, and the DoFs on the boundary are fixed to zero, according to the definition of $\bV_0$. 

Now, we present optimal approximation properties related to the virtual space $\bV_h$. First, we introduce the Fortin operator $\bb{I}_h^F:\H^1(\O)^2\longrightarrow\bV_h$, which is defined through the degrees of freedom $\mathbf{D}_\mathbf{V}$ by (\cite{BEIRAOMIXED}):
\begin{equation*}
\begin{split}
\displaystyle \int_{e} (\bv - \bb{I}_h^F\bv)\cdot\bn q_k &= 0, \quad \forall \;\; \text{edge} \;\;e, \quad \forall q_k\in\mathbb{P}_k(e), \quad \forall e\subset \partial E,\\
\displaystyle \int_E (\bv-\bb{I}_h^F\bv)\cdot \bb{g}_{k-1} &=0, \quad \forall E\in\O_h, \quad \forall \bb{g}_{k-1}\in\mathcal{G}_{k-1}(E), \\
\displaystyle \int_E (\bv-\bb{I}_h^F\bv)\cdot \bb{g}_{k}^{\perp}&=0, \quad \forall E\in\O_h, \quad \forall \bb{g}_{k}^\perp\in\mathcal{G}_k(E)^\perp.
\end{split}
\end{equation*}
%

With this operator at hand, we have the following interpolation error estimate (see \cite[Proposition 4.1]{10.1093/imanum/drad078}).
\begin{proposition}\label{prop:errorinterp}
Under assumption \textbf{A2}), for every $\bv\in\bV_0\cap\H^s(\O)^2$ with $1\leq s\leq k+1$, the following estimate holds
\begin{equation*}
\|\bv-\bb{I}_h^F\bv\|_{\L^2(\O)^2} \leq Ch^s|\bv|_{\H^s(\O)^2},
\end{equation*}
where $C>0$ is independent of $h$ and $k$.
\end{proposition}

\subsection{Finite element space for the pressure}

To approximate the pressure variable, we introduce the finite element space $\Q_h \coloneqq  \mathbb{P}_k(\O_h)\cap\Q$.
Also, we note that the pair $(\bVh,\Q_h)$ satisfies the following discrete inf-sup condition (\cite[Proposition 3.4]{MR3796371}): there exists $\tilde{\beta}>0$, independent of $h$,
such that
\begin{equation}\label{eq:infsupdiscreta}
\underset{\bv_h \in \bVh}{\sup} \; \dfrac{b(\bv_h,\textsf{q}_h)}{\|\bv_h\|_{\H(\div,\O)}} \geq \tilde{\beta} \|\textsf{q}_h\|_{\L^2(\O)} \quad \forall \textsf{q}_h \in \Q_h.
\end{equation}

\subsection{Virtual element space for the temperature}

Now, for an integer $k\geq 0$ and each $E \in \O_h$, we introduce the local virtual space for the temperature $\H_h^{k+1}(E)$, which is defined by
\begin{multline*}
\H_h^{k+1}(E)\coloneqq  \{S_h \in \H^{1}(E) \cap \mathcal{C}^{0}(\partial E) \colon \Delta S_h \in\mathbb{P}_{k+1}(E),\; S_h|_{e} \in \mathbb{P}_{k+1}(e) \; \forall e \in \partial E, \\
(S_h - \Pi^{\nabla,E}_{k+1} S_h, q_{k+1})_{\L^2(E)} = 0 \,\, \forall q_{k+1} \in \mathbb{P}_{k+1\setminus k-1}(E) \}.
\end{multline*}

Similarly to the virtual space for the velocity, we present properties of the space $\H_h^{k+1}(E)$:

\begin{itemize}
\item[\textbf{(P3)}] \textbf{Polynomial inclusion}: $\mathbb{P}_{k+1}(E) \subseteq \H_h^{k+1}(E)$;
\item[\textbf{(P4)}] \textbf{Degrees of freedom}: the following linear operators $D_V$ constitute a set of DoFs for the virtual element space $\H_h^{k+1}(E)$:
\begin{itemize}
\item[$D_V1$:] the values of $S_h$ at the vertices of $E$,
\item[$D_V2$:] the values of $S_h$ at $k$ distinct points of every edge $e \in \partial E$,
\item[$D_V3$:] the interior moments
\begin{equation*}
\dfrac{1}{|E|} \displaystyle \int_{E} S_hq_{k-1} \quad \forall q_{k-1} \in \mathbb{P}_{k-1}(E).
\end{equation*}
\end{itemize}
\end{itemize}
Then, from \cite[Proposition 2]{MR4567234}, the dimension of $\H_h^{k+1}(E)$:
\begin{equation*}
\text{dim}(\H_h^{k+1}(E)) = (\text{number of edges of E})(k+1) + \dfrac{k(k+1)}{2}.
\end{equation*}
\begin{remark}
According to \cite[Proposition 3.2]{MR3796371}, for each $E\in\O_h$ and for all $S_h\in\H_h^{k+1}(E)$, we can compute $\Pi_{k+1}^{\nabla,E}S_h$ exactly from the degrees of freedom $D_V$. 
\end{remark}
Let us  introduce the global virtual space $\H_h$ to approximate the temperature variable, which is defined by
\begin{equation*}
\H_h \coloneqq \{S_h \in \H: \, S_h|_{E} \in \H_h^{k+1}(E), \quad \forall E \in \O_h\}\cap\L^{\infty}(\O).
\end{equation*}
\begin{remark}
We observe that the standard virtual space to approximate the solution of the Laplacian problem is 
\begin{equation*}
\widetilde{\H}_h \coloneqq \{S_h \in \H: \, S_h|_{E} \in \H_h^{\ell}(E), \quad \forall E \in \O_h\},
\end{equation*}
where $\ell\geq 1$. However, it is not possible to ensure that $\H_h\subset\L^\infty(\O)$ as in the case of the finite element discretization presented in \cite[Section 3.1]{BernardiDH}. Then, since we need $\L^{\infty}$ regularity for the test function in the temperature equation, we consider the space $\H_h$ instead of $\widetilde{\H}_h$.
\end{remark}


\subsection{Discrete forms}

With the discrete spaces $\bV_h$, $\Q_h$, and $\H_h$ at hand, we introduce the discrete versions of the continuous forms involved in the weak problem \eqref{eq:DHW2}.
\\
\begin{itemize}[leftmargin=*]
\item The discrete counterpart of $a(\cdot;\cdot,\cdot)$ is defined by 
\begin{equation*}
a_h:\H_h\times\bV_h\times\bV_h \rightarrow \mathbb{R}, 
\qquad
a_h(X_h;\bv_h,\bw_h) \coloneqq  \ds 
\sum_{E\in\O_h} a_h^E(X_h;\bv_h,\bw_h),
\end{equation*}
where $a_h^{E}:\H_h^{k+1}(E)\times\bV_h^k(E)\times\bV_h^k(E) \rightarrow \mathbb{R}$ is given by
\begin{multline*}
a_h^E(X_h;\bv_h,\bw_h) \coloneqq  
\displaystyle 
\int_E \nu(\Pi_{k+1}^{0,E}X_h)\bb{\Pi}_k^{0,E}\bv_h\cdot\bb{\Pi}_k^{0,E}\bw_h 
\\
+ 
\nu(\Pi^{0,E}_0X_h)S_{1}^{E}((\mathbf{I}-\bPi^{0,E}_k)\bv_h,(\mathbf{I}-\bPi^{0,E}_k)\bw_h),
\end{multline*}
where $S_{1}^{E}:\bV_h^k(E)\times\bV_h^k(E) \rightarrow \mathbb{R}$ is a computable symmetric form that satisfies
\begin{equation}\label{eq:S_v}
\alpha_*\|\bv_h\|_{\L^2(E)}^2 \leq S_{1}^{E}(\bv_h,\bv_h) \leq \alpha^*\|\bv_h\|_{\L^2(E)}^2 \quad \forall \bv_h \in \bV_h^k(E) \cap \ker(\bPi^{0,E}_k),
\end{equation}
where $\alpha_*, \alpha^* >0$ are independent of $h_E$. 
\item The discrete counterpart of $\mathfrak{a}(\cdot;\cdot,\cdot)$ is defined by
\begin{equation}
\mathfrak{a}_h:\H_h\times\H_h\times\H_h \rightarrow \mathbb{R}, 
\qquad 
\mathfrak{a}_h(X_h;R_h,S_h) \coloneqq  \ds \sum_{E\in\O_h} \mathfrak{a}_h^E(X_h;R_h,S_h),
\label{eq:bilinear_form_fraka_h}
\end{equation}
where $\mathfrak{a}_h^{E}:\H_h^{k+1}(E)\times\H_h^{k+1}(E)\times\H_h^{k+1}(E) \rightarrow \mathbb{R}$ is given by
\begin{multline*}
\mathfrak{a}_h^E(X_h;R_h,S_h) \coloneqq \displaystyle \int_E \kappa(\Pi^{0,E}_{k+1}X_h)\bPi^{0,E}_{k}\nabla R_h\cdot \bPi^{0,E}_{k}\nabla S_h \\
+ \kappa(\Pi^{0,E}_0 X_h)S_{2}^{E}((\I-\Pi^{0,E}_{k+1})R_h,(\I-\Pi^{0,E}_{k+1})S_h),
\end{multline*}
and  $S_{2}^{E}:\H_h^{k+1}(E)\times\H_h^{k+1}(E) \rightarrow \mathbb{R}$ is a computable symmetric form that satisfies
\begin{equation}
\widetilde{\alpha}_*|X_h|_{\H^1(E)}^2 \leq S_{2}^{E}(X_h,X_h) \leq \widetilde{\alpha}^*|X_h|_{\H^1(E)}^2 \quad \forall X_h \in \H_h^{k+1}(E) \cap \ker(\Pi^{0,E}_k),
\label{eq:S_T}
\end{equation}
and $\widetilde{\alpha}_*,\widetilde{\alpha}^* >0$ are independent of $h_E$.
\item The discrete counterpart of $\mathfrak{c}(\cdot;\cdot,\cdot)$ is defined by
\begin{equation*}
\mathfrak{c}_{h}:\bV_h\times\H_h\times\H_h \rightarrow \mathbb{R}, 
\qquad 
\mathfrak{c}_{h}(\bv_h;R_h,S_h) \coloneqq \ds \sum_{E\in\O_h} \mathfrak{c}_{h}^{E}(\bv_h;R_h,S_h),
\end{equation*}
where $\mathfrak{c}_{h}^{E}:\bV_h^k(E)\times\H_h^{k+1}(E)\times\H_h^{k+1}(E) \rightarrow \mathbb{R}$ is given by
\begin{equation*}
\mathfrak{c}_{h}^E(\bv_h;R_h,S_h) \coloneqq \displaystyle \int_E (\bPi^{0,E}_k\bv_h \cdot \bPi^{0,E}_{k}\nabla R_h)\Pi^{0,E}_{k+1}S_h.
\end{equation*}
Then, we define 
\begin{equation*}
\mathfrak{c}_{h}^{\textit{skew},E}(\bv_h;R_h,S_h):=\displaystyle\frac{1}{2}(\mathfrak{c}_{h}^{E}(\bv_h;R_h,S_h)-\mathfrak{c}_{h}^{E}(\bv_h;S_h,R_h)),
\end{equation*}
and $\mathfrak{c}_{h}^{\textit{skew}}(\bv_h;R_h,S_h) := \ds \sum_{E\in\O_h} \mathfrak{c}_{h}^{\textit{skew},E}(\bv_h;R_h,S_h)$.
\item The discrete counterpart of $F(\cdot)$ is defined by
\begin{equation*}
F_h : \bV_h \longrightarrow \mathbb{R}, \qquad F_h(\bv_h) \coloneqq \ds \sum_{E\in\O_h} F_h^E(\bv_h),
\end{equation*}
where $F_h^E: \bV_h^k(E)\longrightarrow\mathbb{R}$ is defined by $F_h^E(\bv_h) \coloneqq \displaystyle\int_E \bb{f}\cdot\bPi_k^{0,E}\bv_h.$
\item The discrete counterpart of $G(\cdot)$ is defined by
\begin{equation*}
G_h: \H_h \longrightarrow \mathbb{R}, \qquad G_h(S_h) \coloneqq \ds \sum_{E\in\O_h} G_h^E(S_h),
\end{equation*}
where $G_h^E: \H_h^{k+1}(E)\longrightarrow\mathbb{R}$ is given by $G_h^E(S_h) \coloneqq \displaystyle\int_E g\Pi_{k+1}^{0,E}S_h$.
\end{itemize}

\begin{remark}
As mentioned in \cite{MR3626409,MR3796371}, we do not introduce any approximation of the bilinear form $b$. In fact, we note that $b(\bv_h,\textsf{q}_h)$ for $\bv_h \in \bV_h$ and $\textsf{q}_h\in \Q_h$ is computable from $\mathbf{D}_{\mathbf{V}}\mathbf{1}$,  $\mathbf{D}_{\mathbf{V}}\mathbf{2}$, $\mathbf{D}_{\mathbf{V}}\mathbf{4}$ since $\textsf{q}_h$ is polynomial on each element $E \in \O_h$.
\end{remark}
\begin{remark}
We observe that for $\bv\in\bV_0$ and $R,S\in\H\cap\L^\infty(\O)$, the integration by parts formula yields to
\begin{equation*}
\mathfrak{c}(\bv;R,S) = -\mathfrak{c}(\bv;S,R) \quad \text{and} \quad \mathfrak{c}(\bv;S,S)=0.
\end{equation*}
However, at the discrete level this property does not hold. Then, we approximate the form $\mathfrak{c}(\cdot;\cdot,\cdot)$ with the skew-symmetric form $\mathfrak{c}_h^{\textit{skew}}(\cdot;\cdot,\cdot)$. In fact, given $\bv_h\in\bV_h$ and $R_h,S_h\in\H_h$ we have
\begin{equation*}
\mathfrak{c}_h^\textit{skew}(\bv_h;R_h,S_h)=-\mathfrak{c}_h^\textit{skew}(\bv_h;S_h,R_h) \quad \text{and} \quad \mathfrak{c}_h^\textit{skew}(\bv_h;S_h,S_h)=0.
\end{equation*}
\end{remark}

\subsection{The virtual element method}
With all the ingredients and definitions introduced in the previous sections at hand, we finally are in position to introduce the virtual element discretization for the nonlinear system \eqref{eq:DHW2}. This discretization reads as follows:  find  $(\bu_h,\textsf{p}_h,T_h)\in \bV_h\times \Q_h\times\H_h$ such that
\begin{equation}\label{eq:DHW2discreto}
\left\{\begin{array}{rcll}
a_h(T_h;\bu_h,\bv_h) + b(\bv_h,\textsf{p}_h)&=& F_h(\bv_h), \quad &\forall \bv_h \in \bV_h,
\\
b(\bu_h,\textsf{q}_h) &=& 0, \quad &\forall \textsf{q}\in\Q_h,
\\
\mathfrak{a}_h(T_h;T_h,S_h)
+
\mathfrak{c}_h^{\textit{skew}}(\bu_h;T_h,S_h)
&=& G_h(S_h), \quad &\forall S_h\in\H_h.
\end{array}\right.
\end{equation}
%

\begin{remark}
We observe that, from the second equation of system \eqref{eq:DHW2discreto} and the definition of the space $\bV_h$ that the method preserves the zero divergence.
\end{remark}

To provide an analysis for the discrete problem \eqref{eq:DHW2discreto}, we first present some properties that the discrete forms satisfy:

\begin{itemize}[leftmargin=*]
\item For any $X_h \in \H_h$, $a_h(X_h;\cdot,\cdot)$ is a coercive and continuous bilinear form. More precisely,  there exist $\alpha_1,\alpha_2>0$ such that, for every $\bv_h,\bw_h \in \bV_h$ we have
\begin{equation}\label{eq:ah_bound}
\begin{split}
a_h(X_h;\bv_h,\bw_h) &\leq \alpha_1\nu^*\|\bv_h\|_{\L^2(\O)^2}\|\bw_h\|_{\L^2(\O)^2}, \\
a_h(X_h;\bv_h,\bv_h) &\geq \alpha_{2}\nu_{*}\|\bv_h\|_{\L^2(\O)^2}^2.
\end{split}
\end{equation}
\item For any $X_h \in \H_h$, $\mathfrak{a}_h(X_h;\cdot,\cdot)$ is a coercive and continuous bilinear form, i.e., there  exist $\widetilde{\alpha}_1>0$ and $\widetilde{\alpha}_2>0$ such that, for every $R_h, S_h \in \H_h$, we have
\begin{equation}\label{eq:ahtemp_bound}
\begin{split}
\mathfrak{a}_h(X_h;R_h,S_h) &\leq \widetilde{\alpha}_1\kappa^*|R_h|_{\H^1(\O)}|S_h|_{\H^1(\O)}, \\
\mathfrak{a}_h(X_h;S_h,S_h) &\geq \widetilde{\alpha}_2\kappa_{*}|S_h|_{\H^1(\O)}^2.
\end{split}
\end{equation}
\item $\mathfrak{c}_{h}(\cdot;\cdot,\cdot)$ is continuous: For every $\bv_h \in \bV_h$ and $R_h,S_h \in \H_h$, we have
\begin{equation}\label{eq:chtemp_bound}
\mathfrak{c}_h(\bv_h;R_h,S_h) \leq \mathfrak{C}\|\bv_h\|_{\L^2(\O)^2} |R_h |_{\H^1(\O)}\|S_h\|_{\L^{\infty}(\O)},
\end{equation}
\item $b(\cdot,\cdot)$ is continuous: For every $\textsf{q}_h \in \Q_h$ and $\bv_h \in \bV_h$, we have
\begin{equation}\label{eq:bh_bound}
b(\bv_h,\textsf{q}_h) \leq \|\bv_h\|_{\H(\div,\O)}\|\textsf{q}_h\|_{\L^2(\O)}.
\end{equation} 
\end{itemize}

\subsection{Existence and uniqueness of discrete solutions}
\label{sec:ex_uniq_sol_vem}
In what follows, we derive existence and uniqueness results for the discrete problem \eqref{eq:DHW2discreto}, which are inspired in the techniques used for the continuous problem. Moreover, we obtain a global uniqueness result when the problem data is suitably restricted.
First, we note that as in the continuous case, we split system \eqref{eq:DHW2discreto} as the following discrete reduced problem: Find $T_h\in\H_h$ such that 
\begin{equation*}\label{eq:reducedHeatdisc}
\mathfrak{a}_h(T_h;T_h,S_h)+\mathfrak{c}_h^{\textit{skew}}(\bu_h(T_h);T_h,S_h)=G_h(S_h) \quad \forall S_h\in\H_h,
\end{equation*}
where $\bu_h(T_h)$ is the solution of: Find $(\bu_h(T_h),\textsf{p}_h(T_h))\in\bV_h\times\Q_h$ such that 
\begin{equation}\label{eq:reducedDarcydisc}
\left\{\begin{array}{rlc}
a_h(T_h;\bu_h(T_h),\bv_h)+b(\bv_h,\textsf{p}_h(T_h))&=\displaystyle F_h(\bv_h), \quad \forall \bv_h\in\bV_h \\
b(\bu_h(T_h),\textsf{q}_h)&=0, \quad \forall \textsf{q}_h\in\Q_h.
\end{array}\right.
\end{equation}
We observe that given $T_h\in\H_h$, the inf-sup condition \eqref{eq:infsupdiscreta} implies that system \eqref{eq:reducedDarcydisc} has a unique solution $(\bu_h(T_h),\textsf{p}_h(T_h))\in\bV_h\times\Q_h$ and the following estimates hold
\begin{equation*}\label{eq:stabilityDarcydisc}
\|\bu_h(T_h)\|_{\L^2(\O)^2}\leq\dfrac{1}{\alpha_2\nu_*}\|\bb{f}\|_{\L^2(\O)^2}, \quad \|\sqrt{\nu(T_h)}\bu_h(T_h)\|_{\L^2(\O)^2}\leq\dfrac{1}{\alpha_2\sqrt{\nu_*}}\|\bb{f}\|_{\L^2(\O)^2}, \\
\end{equation*}
\begin{equation*}
\|\textsf{p}_h(T_h)\|_{\L^2(\O)} \leq \dfrac{1}{\tilde{\beta}}\left(1 + \dfrac{\alpha_1\nu^*}{\alpha_2\nu_*}\right)\|\bb{f}\|_{\L^2(\O)^2}.
\end{equation*}
On the other hand, since $\H_h\subset\L^{\infty}(\O)$, we can duplicate the arguments used to show existence of solutions of \eqref{eq:Heatseparable} to show the existence of at least one solution $T_h\in\H_h$ of \eqref{eq:reducedDarcydisc}, which satisfies the bound
\begin{equation*}
|T_h|_{\H^1(\O)} \leq \dfrac{\textsf{C}_2}{\widetilde{\alpha}_2\kappa_*}\|g\|_{\L^2(\O)}.
\end{equation*}

On the other hand, for the uniqueness, we use the same techniques as in the continuous case, but with the difference that the discrete velocity now is an element of the space $\H^1(\O_h)^2$ and the discrete version now involves different projections that must be properly estimated.

\begin{theorem}[uniqueness of discrete solutions]\label{eq:uniqueness} 
Let us assume that assumptions \textbf{A0}) and \textbf{A1}) hold. If the nonlinear system \eqref{eq:DHW2discreto} admits a solution $(\bu_h^1,\textsf{p}_h^1,T_h^1) \in \bZh \times \Q_h \times \H_h$ such that $\bu_h^1 \in \H^1(\O)^2$, $T_h^1 \in \W^{1,\infty}(\O)$ and
\begin{equation}
\label{eq:uniq}
\dfrac{\widetilde{\alpha}_2\kappa_* - \widetilde{\mathrm{C}}\kappa_{\text{lip}}\|\nabla T_h^1\|_{\L^\infty(\O)^2}}{C\mathfrak{C}\alpha_1^{-1}\nu_*^{-1}(\mathrm{C}_2+\mathrm{C}_{\frac{2r}{r-2}})\nu_{\text{lip}}\|\bu_h^1\|_{\H^1(\O)^2}\|T_h^1\|_{\L^\infty(\O)}}<1,
\end{equation}
for each $h>0$, then this solution is unique.
\label{thm:uniqueness}
\end{theorem}

\begin{proof}
Let $(\bu_h^2,\textsf{p}_h^2,T_h^2) \in \bZh\times\Q_h\times\H_h$ be another solution of system \eqref{eq:DHW2discreto} and define
\begin{equation*}
\overline{\bu}_h \coloneqq \bu_h^1-\bu_h^2 \in \bZh,
\qquad
\overline{T}_h \coloneqq T_h^1-T_h^2 \in \H_h,
\qquad
\overline{\textsf{p}}_h \coloneqq \textsf{p}_h^1-\textsf{p}_h^2 \in \Q_h.
\end{equation*}

First, we observe from the definition of $\mathfrak{c}_h^{\textit{skew}}(\cdot;\cdot,\cdot)$ that
\begin{equation}\label{eq:skew1}
\mathfrak{c}_h^{\textit{skew}}(\bu_h^2;\overline{T}_h,\overline{T}_h) = 0 \implies \mathfrak{c}_h^{\textit{skew}}(\bu_h^2;T_h^2,\overline{T}_h) = \mathfrak{c}_h^{\textit{skew}}(\bu_h^2;T_h^1,\overline{T}_h).
\end{equation}

Since  $(\bu_h^1,T_h^1)$ and $(\bu_h^2,T_h^2)$ are solutions of \eqref{eq:DHW2discreto}, we set $\bv_h = \mathbf{0}$ and $S_h = \overline{T}_h \in \H_h$ in the corresponding system and obtain
\begin{equation*}
\mathfrak{a}_h(T_h^1;T_h^1,\overline{T}_h) + \mathfrak{c}_h^{\textit{skew}}(\bu_h^1;T_h^1,\overline{T}_h) = \mathfrak{a}_h(T_h^2;T_h^2,\overline{T}_h) + \mathfrak{c}_h^{\textit{skew}}(\bu_h^2;T_h^2,\overline{T}_h).
\end{equation*}
We add and subtract the term $\mathfrak{a}_h(T_h^2;T_h^1,\overline{T}_h)$ and use \eqref{eq:skew1} to obtain
\begin{equation*}
\mathfrak{a}_h(T_h^2;\overline{T}_h,\overline{T\textit{T}}_h) = \left[\mathfrak{a}_h(T_h^2;T_h^1,\overline{T}_h)-\mathfrak{a}_h(T_h^1;T_h^1,\overline{T}_h)\right] + \mathfrak{c}_h^{\textit{skew}}(\overline{\bu}_h;T_h^1,\overline{T}_h).
\end{equation*}
On the other hand, since $\overline{T}_h\in\L^{\infty}(\O)$, using integration by parts we obtain 
\begin{equation*}
\mathfrak{c}_h^{\textit{skew}}(\overline{\bu}_h;T_h^1,\overline{T}_h) = -\mathfrak{c}_h^{\textit{skew}}(\overline{\bu}_h;\overline{T}_h,T_h^1),
\end{equation*}
which combined with the coercivity property in \eqref{eq:ahtemp_bound} and bound \eqref{eq:chtemp_bound} yield
\begin{multline}\label{temperatura}
\widetilde{\alpha}_2\kappa_{*}|\overline{T}_h|_{\H^1(\O)}^{2} \leq \left\lvert \mathfrak{a}_h(T_h^2;T_h^1,\overline{T}_h)-\mathfrak{a}_h(T_h^1;T_h^1,\overline{T}_h) \right\lvert \\
+ \mathfrak{C}\|\overline{\bu}_h\|_{\L^2(\O)^2}\|T_{h}\|_{\L^\infty(\O)}|\overline{T}_h|_{\H^1(\O)}.
\end{multline}
Now, using the definition of $\mathfrak{a}_h(\cdot;\cdot,\cdot)$ and triangle inequality we obtain
\begin{equation*}
\left\lvert \mathfrak{a}_h(T_h^2;T_h^1,\overline{T}_h)-\mathfrak{a}_h(T_h^1;T_h^1,\overline{T}_h) \right\lvert \leq \eta_1+\eta_2,
\end{equation*}
where $\eta_1$ and $\eta_2$ are defined by
\begin{align*}
\eta_1&\coloneqq\sum_{E\in\O_h} \int_E \lvert \kappa(\Pi_{k+1}^{0,E}T_h^1) - \kappa(\Pi_{k+1}^{0,E}T_h^2)\lvert\lvert\bb{\Pi}_{k}^{0,E}\nabla T_h^1\lvert\lvert\bb{\Pi}_{k}^{0,E}\nabla \overline{T}_h\lvert, \\
\eta_2&\coloneqq\sum_{E\in\O_h}\lvert \kappa(\Pi_{0}^{0,E}T_h^1) - \kappa(\Pi_{0}^{0,E}T_h^2)\lvert \lvert S_2^E((I-\Pi_{k+1}^{0,E})T_h^1,(I-\Pi_{k+1}^{0,E})\overline{T}_h)\lvert.
\end{align*}
The task now is to estimate each of these terms.

\emph{Estimate for $\eta_1$.} First, invoking assumption \textbf{A1}), H\"{o}lder inequality, and Proposition \ref{prop:stabPi} with $p=\infty$,  we obtain
\begin{align*}
\eta_1 &\leq \kappa_{\text{lip}}\sum_{E\in\O_h} \|\Pi_{k+1}^{0,E}\overline{T}_h\|_{\L^2(E)}\|\bb{\Pi}_{k}^{0,E}\nabla T_h^1\|_{\L^{\infty}(E)^2}\|\bb{\Pi}_{k}^{0,E}\nabla \overline{T}_h\|_{\L^2(E)} \\
&\leq \mathrm{C}_\infty\kappa_{\text{lip}}\sum_{E\in\O_h} \|\overline{T}_h\|_{\L^2(E)}\|\nabla T_h^1\|_{\L^{\infty}(E)^2}|\overline{T}_h|_{\H^1(E)} \\
&\leq \mathrm{C}_\infty\kappa_{\text{lip}}\|\overline{T}_h\|_{\L^2(\O)}\|\nabla T_h^1\|_{\L^{\infty}(\O)^2}|\overline{T}_h|_{\H^1(\O)} \leq \textsf{C}_2\mathrm{C}_\infty\kappa_{\text{lip}}|\overline{T}_h|_{\H^1(\O)}^2\|\nabla T_h^1\|_{\L^{\infty}(\O)^2}.
\end{align*}
\emph{Estimate for $\eta_2$.} Invoking once again assumption \textbf{A1}), the scaling property of $S_2^E(\cdot,\cdot)$, Lemma \ref{prop:bramblehilbert}, and the inverse estimate for polynomials \eqref{eq:inversepoly} with $q=\infty$ and $p=2$, we obtain
\begin{align*}
\eta_2 &\leq C\kappa_{\text{lip}}\sum_{E\in\O_h} h_E^{-1}\|\Pi_0^{0,E}\overline{T}_h\|_{\L^2(E)}|(I-\Pi_{k+1}^{0,E})T_h^1|_{\H^1(E)}|(I-\Pi_{k+1}^{0,E})\overline{T}_h|_{\H^1(E)} \\
&\leq C\kappa_{\text{lip}}\sum_{E\in\O_h} h_E^{-1}h_E\|\Pi_0^{0,E}\overline{T}_h\|_{\L^2(E)}\|\nabla T_h^1\|_{\L^\infty(E)^2}|\overline{T}_h|_{\H^1(E)} \\
&\leq C\kappa_{\text{lip}}\|\overline{T}_h\|_{\L^2(\O)}\|\nabla T_h^1\|_{\L^\infty(\O)^2}|\overline{T}_h|_{\H^1(\O)} \leq C\textsf{C}_2\kappa_\text{lip}|\overline{T}_h|_{\H^1(\O)}^2\|\nabla T_h^1\|_{\L^\infty(\O)^2}.
\end{align*}
Thus, we conclude that 
\begin{equation*}
\eta_1+\eta_2 \leq \widetilde{\mathrm{C}}\kappa_\text{lip}|\overline{T}_h|_{\H^1(\O)}^2\|\nabla T_h^1\|_{\L^\infty(\O)^2},
\end{equation*}
and therefore
\begin{equation*}
\widetilde{\alpha}_2\kappa_*|\overline{T}_h|_{\H^1(\O)}^2 \leq \widetilde{\mathrm{C}}\kappa_\text{lip}|\overline{T}_h|_{\H^1(\O)}^2\|\nabla T_h^1\|_{\L^\infty(\O)^2} + \mathfrak{C}\|\overline{\bu}_h\|_{\L^2(\O)^2}\|T_h\|_{\L^\infty(\O)}|\overline{T}_h|_{\H^1(\O)},
\end{equation*}
or, equivalently
\begin{equation}\label{eq:Thbarra}
(\widetilde{\alpha}_2\kappa_* - \widetilde{\mathrm{C}}\kappa_{\text{lip}}\|\nabla T_h\|_{\L^\infty(\O)^2})|\overline{T}_h|_{\H^1(\O)} \leq \mathfrak{C}\|\overline{\bu}_h\|_{\L^2(\O)^2}\|T_h\|_{\L^\infty(\O)}.
\end{equation}

Now the task is to obtain a bound for $\|\overline{\bu}_h\|_{\L^2(\O)^2}$. With this aim in mind, we use the fact that $(\bu_h^1,T_h^1)$ and $(\bu_h^2,T_h^2)$ solve system \eqref{eq:DHW2discreto} to set $\bv_h = \overline{\bu}_h \in \bV_h$ and $S_h=0\in\H_h$ in the corresponding systems and obtain $a_h(T_h^1;\bu_h^1,\overline{\bu}_h) = a_h(T_h^2;\bu_h^2,\overline{\bu}_h)$. 
Adding and subtracting the term $a_h(T_h^2;\bu_h^1,\overline{\bu}_h)$ we obtain
\begin{equation*}
a_h(T_h^2;\overline{\bu}_h,\overline{\bu}_h) = a_h(T_h^2;\bu_h^1,\overline{\bu}_h) - a_h(T_h^1;\bu_h^1,\overline{\bu}_h).
\end{equation*}
Then, invoking \eqref{eq:ah_bound} we have the following estimate
\begin{equation*}\label{uniq2}
\alpha_{1}\nu_*\|\overline{\bu}_h\|_{\L^2(\O)^2}^{2} \leq \lvert a_h(T_h^2;\bu_h^1,\overline{\bu}_h) - a_h(T_h^1;\bu_h^1,\overline{\bu}_h) \lvert \leq \Lambda_1+ \Lambda_2,
\end{equation*}
where $\Lambda_1$ and  $\Lambda_2$ are defined by
\begin{align*}
\Lambda_1&\coloneqq\sum_{E\in\O_h} \int_E \lvert \nu(\Pi_{k+1}^{0,E}T_h^1) - \nu(\Pi_{k+1}^{0,E}T_h^2)\lvert\lvert\bb{\Pi}_{k}^{0,E}\bu_h^1\cdot\bb{\Pi}_{k}^{0,E}\overline{\bu}_h\lvert, \\
\Lambda_2&\coloneqq\sum_{E\in\O_h}\lvert \nu(\Pi_{0}^{0,E}T_h^1) - \nu(\Pi_{0}^{0,E}T_h^2)\lvert \lvert S_1^E((\bb{I}-\bb{\Pi}_{k}^{0,E})\bu_h^1,(\bb{I}-\bb{\Pi}_{k}^{0,E})\overline{\bu}_h)\lvert.
\end{align*}
Similar arguments  for the estimates of $\eta_1$ and $\eta_2$ are needed to estimate  $\Lambda_1$ and $\Lambda_2$.

\emph{Estimate for $\Lambda_1$.} Invoking \textbf{A0}), using H\"{o}lder's  inequality for $r>2$, Proposition \ref{prop:stabPi}, and Poincar\'e inequality \eqref{eq:Poincare}, we arrive to
\begin{align*}
\Lambda_1 &\leq \nu_\text{lip}\sum_{E\in\O_h} \|\Pi_{k+1}^{0,E}\overline{T}_h\|_{\L^{\frac{2r}{r-2}}(E)}\|\bu_h^1\|_{\L^r(E)^2}\|\overline{\bu}_h\|_{\L^2(E)^2} \\
&\leq \nu_\text{lip}\sum_{E\in\O_h}\|\overline{T}_h\|_{\L^{\frac{2r}{r-2}}(E)}\|\bu_h^1\|_{\L^r(E)^2}\|\overline{\bu}_h\|_{\L^2(E)^2} \\
&\leq \nu_\text{lip}\|\overline{T}_h\|_{\L^{\frac{2r}{r-2}}(\O)}\|\bu_h^1\|_{\H^1(\O)^2}\|\overline{\bu}_h\|_{\L^2(\O)^2} \\
&\leq \textsf{C}_{\frac{2r}{r-2}}\nu_\text{lip}|\overline{T}_h|_{\H^1(\O)}\|\bu_h^1\|_{\H^1(\O)^2}\|\overline{\bu}_h\|_{\L^2(\O)^2}.
\end{align*}
\emph{Estimate for $\Lambda_2$}. Using assumption \textbf{A0}), the scaling property of $S_V^E(\cdot,\cdot)$, and Lemma \ref{prop:bramblehilbert}, we obtain
\begin{align*}
\Lambda_2&\leq C\nu_\text{lip}\sum_{E\in\O_h} h_E^{-1}\|\Pi_0^{0,E}\overline{T}_h\|_{\L^2(E)}\|(\bb{I}-\bb{\Pi}_{k}^{0,E})\bu_h^1\|_{\L^2(E)^2}\|(\bb{I}-\bb{\Pi}_{k}^{0,E})\overline{\bu}_h\|_{\L^2(E)^2} \\
&\leq C\nu_\text{lip}\sum_{E\in\O_h} \|\overline{T}_h\|_{\L^2(E)}\|\bu_h^1\|_{\H^1(E)^2}\|\overline{\bu}_h\|_{\L^2(E)^2} \\
&\leq C\nu_\text{lip} \|\overline{T}_h\|_{\L^2(\O)}\|\bu_h^1\|_{\H^1(\O)^2}\|\overline{\bu}_h\|_{\L^2(\O)^2} \\
&\leq C\textsf{C}_{2}\nu_\text{lip} |\overline{T}_h|_{\H^1(\O)}\|\bu_h^1\|_{\H^1(\O)^2}\|\overline{\bu}_h\|_{\L^2(\O)^2}.
\end{align*}
Thus, we obtain 
\begin{equation*}
\Lambda_1 + \Lambda_2 \leq C(\textsf{C}_2+\textsf{C}_{\frac{2r}{r-2}})\nu_{\text{lip}}|\overline{T}_h|_{\H^1(\O)}\|\bu_h^1\|_{\H^1(\O)^2}\|\overline{\bu}_h\|_{\L^2(\O)^2},
\end{equation*}
so we conclude that 
\begin{equation}\label{eq:uhbarra}
\|\overline{\bu}_h\|_{\L^2(\O)} \leq C\alpha_1^{-1}\nu_*^{-1}(\textsf{C}_2+\textsf{C}_{\frac{2r}{r-2}})\nu_{\text{lip}}|\overline{T}_h|_{\H^1(\O)}\|\bu_h^1\|_{\H^1(\O)^2}.
\end{equation}
Then, replacing \eqref{eq:uhbarra} into \eqref{eq:Thbarra} and using simple algebraic manipulations, we obtain
\begin{equation*}
\left(\dfrac{\widetilde{\alpha}_2\kappa_* - \widetilde{\mathrm{C}}\kappa_{\text{lip}}\|\nabla T_h\|_{\L^\infty(\O)^2}}{C\mathfrak{C}\alpha_1^{-1}\nu_*^{-1}(\textsf{C}_2+\textsf{C}_{\frac{2r}{r-2}})\nu_{\text{lip}}\|\bu_h^1\|_{\H^1(\O)^2}\|T_h\|_{\L^\infty(\O)}}-1\right)|\overline{T}_h|_{\H^1(\O)} \leq 0.
\end{equation*}
Therefore, assuming that \eqref{eq:uniq} holds, we deduce that $\overline{T}_h = 0$ so $T_h^1=T_h^2$. Replacing it into \eqref{eq:uhbarra} we obtain $\overline{\bu}_h=0$ so $\bu_h^1 = \bu_h^2$. Finally, using the inf-sup condition \eqref{eq:infsupdiscreta} we obtain consequently that $\overline{\textsf{p}} = 0$ so $\textsf{p}_h^1 = \textsf{p}_h^2$. This concludes the proof.
\end{proof}


\section{Error estimates}
\label{sec:error}
In the present section, we derive error estimates for the discrete system \eqref{eq:DHW2discreto}. In order to perform this analysis, we will make the following regularity assumptions:

\begin{itemize}
\item[\textbf{A4})] The solutions $(\bu,\textsf{p},T) \in \bV_0\times\Q\times\H$ and the data $\boldsymbol{f},g,\kappa(\cdot),\nu(\cdot)$ of system \eqref{eq:DHW2} satisfy the following regularity properties for some $1\leq s\leq k+1$:
\begin{itemize}
\item[i)] $\bu \in \H^{s}(\O)^2$, $\textsf{p} \in \H^{s}(\O)$, and $T\in \W^{1+s,p}(\O)\cap\W^{1,\infty}(\O)$ with $p>2$.
\item[ii)] $\bb{f} \in \H^{s}(\O)^2$ and $g \in \H^{1+s}(\O)$.
\item[iii)] $\nu(T)$ and $\kappa(T)$ belong to $\W^{s,\infty}(\O)$.
\end{itemize}
\end{itemize}

The following result provides the existence of a virtual interpolant for a function $S\in\W^{1+s,p}(\O)$, when  $s\geq 1$ and $p>2$ on the space $\bVV$. To prove this result, we take advantage of the proof of \cite[Proposition 4.2]{MR3340705}, which we adapt for our case.
\begin{lemma}[Virtual interpolant in $\W^{1,p}(\O)$, $p>2$]
\label{estimate2gen}
Under assumptions \textbf{A2}) and \textbf{A3}), let $S\in \H\cap\W^{1+s,p}(\O)$, with $1\leq s\leq k+1$ and $p>2$. Then, there exists $S_I\in \H_h$ such that
\begin{equation*}
\|S-S_I\|_{\L^p(\O)}+h|S-S_I|_{\W^{1,p}(\O)}\leq Ch^{s+1}|S|_{\W^{1+s,p}(\O)},
\end{equation*}
where $C>0$ is independent of $h$.
\end{lemma}
\begin{proof}
Let $S\in\W^{1+s,p}(\O)$ and let $S_{\pi}=\Pi_{k+1}^0S$. Given a polygon $E\in\O_h$, we consider the triangulation $\O_h^E$ obtained by joining each vertex of the element $E$ with the center of the ball for which $E$ is starred. Next, we define $\widetilde{\O}_h:=\cup_{E\in\O_h}\O_h^E$ which corresponds to a shape-regular family of triangulations of $\O$. On the other hand, since $s\geq 1$, we have the inclusion $\W^{1+s,p}(\O)\hookrightarrow \mathcal{C}(\overline{\O})$, which allows us to use  the Lagrange interpolant of degree $k+1$ of $S$ over $\widetilde{\CT}_h$ that we denote by $S_{L}$. Let us recall the following classic estimate (\cite{M2AN_1975}): $$\|S-S_L\|_{\L^p(\O)} + h|S-S_L|_{\W^{1,p}(\O)} \leq Ch^{1+s}|S|_{\W^{1+s,p}(\O)}.$$ Now, for $E\in\O_h$, we define $S_{I}|_E\in\H^1(E)$ as the solution of the following local problem: $$-\Delta S_I = -\Delta S_{\pi} \quad \text{on}\; E, \quad S_I=S_L \quad \text{on} \; \partial E.$$ We observe that $S_I\in\H_{h}^{k+1}(E)$. Then, the above problem can be rewritten as $$-\Delta (S_\pi-S_I) = 0 \quad \text{on}\; E, \quad S_\pi-S_I=S_\pi-S_L \quad \text{on} \; \partial E.$$ Then, we have 
\begin{align*}
|S_\pi-S_I|_{\W^{1,p}(E)} &= \inf\{|R|_{\H^1(E)} : R\in\H^1(E) \; \text{and} \; R=S_\pi-S_L \; \text{on} \; \partial E\} \\
&\leq |S_\pi-S_L|_{\W^{1,p}(E)}.
\end{align*}
Therefore, using triangle inequality and the above inequality, together with error estimates for the Lagrange interpolant, we obtain
\begin{align*}
|S-S_I|_{\W^{1,p}(E)} &\leq |S-S_\pi|_{\W^{1,p}(E)}+|S_\pi-S_I|_{\W^{1,p}(E)} \\ 
&\leq 2|S-S_\pi|_{\W^{1,p}(E)} + |S-S_L|_{\W^{1,p}(E)} \leq Ch_E^{s}|S|_{\W^{1+s,p}(E)}.
\end{align*}
On the other hand, each triangle $T\in\O_h^E$ has one edge on $\partial E$, and since $S_I=S_L$ on $\partial E$, scaling arguments and Poincar\'e inequality yield $$\|S_L-S_I\|_{\L^p(T)} \leq Ch_E|S_L-S_I|_{\W^{1,p}(T)}.$$ Therefore, combining triangle inequality with the above estimates we obtain
\begin{align*}
\|S-S_I\|_{\L^p(E)} &\leq \|S-S_L\|_{\L^p(E)} + \|S_L-S_I\|_{\L^p(E)} \\
&\leq \|S-S_L\|_{\L^p(E)} + Ch_E|S_L-S_I|_{\W^{1,p}(E)} \\
&\leq \|S-S_L\|_{\L^p(E)} + Ch_E|S-S_L|_{\W^{1,p}(E)} + Ch_E|S-S_I|_{\W^{1,p}(E)} \\
&\leq Ch_E^{1+s}|S|_{\W^{1+s,p}(E)}.
\end{align*}
The proof is concluded by summing over all polygons $E\in\O_h$.
\end{proof}

We end this section by proving error estimates for our method. The result follows under the regularity assumptions given before and  smallness assumption on the data. The following lemmas are needed to prove our final result. 
\begin{lemma}\label{Lema1temp}
Let us suppose that \textbf{A1}), \textbf{A2}), \textbf{A3}), and \textbf{A4}) hold. Assume the smallness assumptions given by Theorems \ref{eq:uniqcont} and \ref{eq:uniqueness}. Let $(\bu,\textsf{p},T) \in \bV_0 \times \Q \times \H$ and $(\bu_h,\textsf{p}_h,T_h) \in \bV_h\times\Q_h\times\H_h$ be the unique solutions of problems \eqref{eq:DHW2} and \eqref{eq:DHW2discreto}, respectively, and let $T_I$ and $\bu_I$ be the interpolants of $T\in\H_h$ and $\bu\in\bV_h$ given by Lemmas \ref{estimate2gen}, then we have
\begin{equation*}
\left(\widetilde{\alpha}_2\kappa_* - \mathfrak{A}(T)\right)|E_h|_{\H^1(\O)} \leq \mathfrak{A}(\bu,T,\bb{f},g,\kappa)h + \mathfrak{A}(T)\|\bb{e}_h\|_{\L^2(\O)^2},
\end{equation*}
\end{lemma}
where $E_h \coloneqq T_I - T_h$ and $\bb{e}_h\coloneqq \bu_I-\bu_h$.
\begin{proof}
First we invoke the coercivity bound \eqref{eq:ahtemp_bound}, the definition of $E_h$,  and adding  and subtracting $\mathfrak{a}(T;T,E_h)$, obtaining
\begin{equation*}
 \begin{aligned}
\label{eqtemp1}
\widetilde{\alpha}_{2}\kappa_{*}|E_h|_{1,\O}^{2} & \leq \mathfrak{a}_h(T_h;E_h,E_h) = \mathfrak{a}_h(T_h;T_I,E_h) - \mathfrak{a}_h(T_h;T_h,E_h)
\\
& = \mathfrak{a}_h(T_h;T_I,E_h) - \mathfrak{a}(T;T,E_h) + \mathfrak{a}(T;T,E_h) - \mathfrak{a}_h(T_h;T_h,E_h).
\end{aligned}
\end{equation*}
We now use the third equations of the continuous and discrete systems, \eqref{eq:DHW2} and \eqref{eq:DHW2discreto}, respectively, to arrive at the following estimate:
\begin{multline}
 \widetilde{\alpha}_{2}\kappa_{*}|E_h|_{1,\O}^{2} \leq \left[\mathfrak{a}_h(T_h;T_I,E_h) - \mathfrak{a}(T;T,E_h)\right]
\\
+
\left[ \mathfrak{c}_h^{\textit{skew}}(\bu_h;T_h,E_h) - \mathfrak{c}(\bu;T,E_h)\right]
+ ( g-\Pi_k^0g,E_h)_{\L^2(\O)} =: \I + \I\I + \I\I\I.
\label{eq:I+II+III}
\end{multline}
Hence, we need to estimate the contributions $\I,\I\I$ and $\I\I\I$. To estimate $\I$, we first rewrite $\I$ using the definitions of $\mathfrak{a}(\cdot;\cdot,\cdot)$ and $\mathfrak{a}_h(\cdot;\cdot,\cdot)$ given in \eqref{eq:bilinear_form_fraka} and \eqref{eq:bilinear_form_fraka_h}, respectively, obtaining
\begin{multline*}
\I
=
\displaystyle
\sum_{E\in\O_h} \bigg\{ \int_{E} \kappa(\Pi^{0,E}_{k+1} T_h)
(\bPi^{0,E}_{k} \nabla T_I - \nabla T)
\cdot
\bPi^{0,E}_{k}\nabla E_h
-  \int_{E}\kappa(T)\nabla T\cdot\nabla E_h
\\
+ \int_{E} \kappa(\Pi^{0,E}_{k+1}T_h)\nabla T \cdot \bPi^{0,E}_{k}\nabla E_h
+
\displaystyle
\kappa(\Pi^{0,E}_0 T_h) S_T^{E}((\I-\Pi^{0,E}_{k+1})T_I,(\I-\Pi^{0,E}_{k+1}) E_h)\bigg\}.
\end{multline*}
We construct further differences as follows:
\begin{equation*}
\begin{aligned}
\I
& = \displaystyle
\sum_{E\in\O_h}
\bigg\{
\int_{E} \kappa(\Pi^{0,E}_{k+1} T_h)(\bPi^{0,E}_{k} \nabla T_I - \nabla T)\cdot \bPi^{0,E}_{k}\nabla E_h
\\
& \quad + \displaystyle  \int_E
\left(
\kappa(\Pi^{0,E}_{k+1} T_h) - \kappa(T)
\right)\nabla T
\cdot \bPi^{0,E}_{k} \nabla E_h
-
\displaystyle  \int_E \kappa(T)\nabla T \cdot ( \nabla E_h - \bPi^{0,E}_{k} \nabla E_h)
\\
& \quad
+
\displaystyle \kappa(\Pi^{0,E}_0 T_h) S_T^{E}((\I-\Pi^{0,E}_{k+1})T_I,(\I-\Pi^{0,E}_{k+1}) E_h)
\bigg\}
=: \sum_{E\in\O_h} ( \I_1^E + \I_2^E - \I_3^E + \I_4^E).
\end{aligned}
\end{equation*}
Using the definition of $\bPi^{0,E}_{k}$, $\I_3^E$ can be rewritten as $\I_3^E = \int_E (\mathbf{I} - \bPi^{0,E}_{k})(\kappa(T)\nabla T) \cdot \nabla E_h$. Now, we observe that the terms $\I_1^E$, $\I_3^E$, and $\I_4^E$ can be controlled simultaneously under the assumption \textbf{A0}) and \eqref{eq:S_T}. In fact, we have
\begin{multline}
\label{eq:aux_temp_0}
 \sum_{E\in\O_h}  (\I_1^E - \I_3^E + \I_4^E)
 \leq C
 \sum_{E\in\O_h}
 \left( \kappa^{*}
 \| \bPi^{0,E}_{k} \nabla T_I - \nabla T \|_{\L^2(E)^2}
 \right.
 \\
 \left.
 +
 \| (\mathbf{I} - \bPi^{0,E}_{k})(\kappa(T)\nabla T) \|_{\L^2(E)^2}
 +
 \kappa^{*} |(\I - \Pi^{0,E}_{k+1})T_I |_{\H^1(E)}
 \right)
 |E_h |_{\H^1(E)},
\end{multline}
where we use the estimates $\| \bPi^{0,E}_{k} \nabla E_h \|_{\L^2(E)^2} \leq |E_h |_{\H^1(E)}$, $\| ( \I  - \Pi^{0,E}_{k+1}) E_h \|_{\L^2(E)} \leq \| E_h \|_{\L^2(E)}$ and $\| ( \I  - \bPi^{0,E}_{k+1}) \nabla E_h \|_{\L^2(E)} \leq |E_h|_{\H^1(E)}$. We now use triangle inequality, and Lemmas \ref{estimate2gen} and \ref{prop:bramblehilbert} to obtain
\begin{equation}
\begin{split}
 \| \bPi^{0,E}_{k} \nabla T_I - \nabla T \|_{\L^2(E)^2}
 &\leq
 |T - T_I|_{\H^1(E)}
 +
 \| (\mathbf{I} - \bPi^{0,E}_{k}) \nabla T \|_{\L^2(E)^2}
\\ 
 &\leq
C\left( |T - T_I\|_{\W^{1,p}(E)}
 +
 \| (\mathbf{I} - \bPi^{0,E}_{k}) \nabla T \|_{\L^p(E)^2}\right)
\\ 
&\leq Ch_E^s |T|_{\W^{1+s,p}(E)}.
\end{split}
 \label{eq:aux_temp_1}
\end{equation}
An estimate for $|(\I - \Pi^{0,E}_{k+1})T_I|_{\H^1(E)}$ can be derived in view of similar arguments:
\begin{equation}
\begin{split}
 | (\I - \Pi^{0,E}_{k+1})T_I |_{\H^1(E)}
 &\leq 
 |(\I - \Pi^{0,E}_{k+1}) (T_I-T)|_{\H^1(E)}
 +
 |(\I - \Pi^{0,E}_{k+1}) T|_{\H^1(E)}
 \\
 &\leq
 |T_I-T|_{\H^1(E)}
 +
|(\I - \Pi^{0,E}_{k+1}) T |_{\H^1(E)} 
\\
&\leq C\left(
 |T_I-T|_{\W^{1,p}(E)}
 +
|(\I - \Pi^{0,E}_{k+1}) T |_{\W^{1,p}(E)} \right) \\
&\leq Ch_E^s |T|_{\W^{1+s,p}(E)}.
\label{eq:aux_temp_2}
\end{split}
\end{equation}
Finally, we bound $\| (\mathbf{I} - \bPi^{0,E}_{k})(\kappa(T)\nabla T) \|_{\H^1(E)^2}$ as follows:
\begin{equation}
\begin{split}
\| (\mathbf{I} - \bPi^{0,E}_{k})(\kappa(T)\nabla T) \|_{\H^1(E)} 
&\leq C\|(\mathbf{I} - \bPi^{0,E}_{k})(\kappa(T)\nabla T) \|_{\W^{1,p}(E)^2} \\
&\leq
Ch_E^s | \kappa(T)\nabla T |_{\W^s(E)} \\
&\leq
Ch_E^s \| \kappa(T) \|_{\W^{s,\infty}(E)} | T |_{\W^{s+1,p}(E)}.
\label{eq:aux_temp_3}
\end{split}
\end{equation}
If we substitute the estimates \eqref{eq:aux_temp_1}, \eqref{eq:aux_temp_2}, and \eqref{eq:aux_temp_3} into \eqref{eq:aux_temp_0} we arrive at the bound
\begin{equation}
\label{eq:aux_temp_4}
 \sum_{E\in\O_h}  (\I_1^E - \I_3^E + \I_4^E)
 \leq \mathfrak{A}(\kappa,T) h^s |E_h|_{\H^1(\O)},
 \end{equation}
 where $\mathfrak{A} = \mathfrak{A}(\kappa,T)$ is a positive constant depending on $\kappa$ and $T$.
We now turn to the derivation of a bound for $\I_2^E$. For this purpose, we invoke H\"{o}lder's inequality and the stability of the $\L^2(E)$-projection in $\L^2(E)$ and $\L^p(E)$ with $p>2$ to obtain
 \begin{equation}
 \begin{split}
\I_2^E
&\leq
\kappa_{lip} \|T-\Pi^{0,E}_{k+1}T_h\|_{\L^{\frac{2p}{p-2}}(E)}|T|_{\W^{1,p}(E)}|E_h|_{\H^1(E)} \\
&\leq
\kappa_{lip} \left( \|T- \Pi^{0,E}_{k+1}T \|_{\L^{\frac{2p}{p-2}}(E)} + \| \Pi^{0,E}_{k+1}(T - T_h) \|_{\L^{\frac{2p}{p-2}}(E)}\right)|T|_{\W^{1,p}(E)}|E_h|_{\H^1(E)} \\
&\leq 
\kappa_{lip}
\left( \|T- \Pi^{0,E}_{k+1}T \|_{\L^{\frac{2p}{p-2}}(E)} + \| T - T_h \|_{\L^{\frac{2p}{p-2}}(E)}\right)|T|_{\W^{1,p}(E)}|E_h|_{\H^1(E)}.
\label{eq:aux_temp_5}
\end{split}
\end{equation}
Then, invoking Lemma \ref{prop:bramblehilbert} and \eqref{eq:embedding} it follows that
\begin{equation*}
\|T- \Pi^{0,E}_{k+1}T \|_{\L^{\frac{2p}{p-2}}(E)} \leq Ch_E^{1+s} |T|_{\W^{1+s,p}(E)}.
\end{equation*}
On the other hand, we have $\| T - T_h \|_{\L^\frac{2p}{p-2}(E)} \leq \| T - T_I \|_{\L^\frac{2p}{p-2}(E)} + \| T_I - T_h \|_{\L^\frac{2p}{p-2}(E)}$. Substituting these estimates into \eqref{eq:aux_temp_5} and using \eqref{eq:Poincare} we obtain
\[
\I_2^E \leq
C\kappa_{lip}
(
h_E^{1+s} |T|_{\W^{1+s,p}(E)} + | T - T_I |_{\W^{1,p}(E)} + | E_h |_{\H^1(E)}
)|T|_{\W^{1,p}(E)}|E_h|_{\H^1(E)},
\]
where we also used that $E_h = T_I - T_h$. Therefore, summing over all the elements in $E$ in $\O_h$ we deduce that 
\begin{equation*}
\I_2 \coloneqq \sum_{E\in\O_h} \I_2^E \leq
C\kappa_{lip}
(
h^s |T|_{\W^{1+s,p}(\Omega)} + | E_h |_{\H^1(\Omega)}
)|T|_{\W^{1,p}(\Omega)}|E_h|_{\H^1(\Omega)}.
\end{equation*}
Then, we obtain the estimate 
\begin{equation}
\I_2 \leq \mathfrak{A}(T)|E_h|^2_{\H^1(\O)} + \mathfrak{A}(\kappa,T)h^s|E_h|_{\H^1(\Omega)}.
\label{eq:aux_temp_6}
\end{equation} 
Finally, combining \eqref{eq:aux_temp_4} with \eqref{eq:aux_temp_6} we obtain
\begin{equation}
\label{eq:IT}
\I
=
\sum_{E\in\O_h}  (\I_1^E + \I_2^E - \I_3^E + \I_4^E) \leq
\mathfrak{A}(T)|E_h|^2_{\H^1(\Omega)} + \mathfrak{A}(\kappa,T)h^s|E_h|_{\H^1(\Omega)}.
\end{equation}
Let us now control $\I\I$ in \eqref{eq:I+II+III}. As a first step, we note that since $T\in\W^{1,p}(\Omega)$ for $p>2$ and in light of \eqref{eq:embedding}, we have $\mathfrak{c}(\bu,T;E_h) = \mathfrak{c}^{skew}(\bu,T;E_h)$ because $\bu \in \bZ$. We use this property, the fact that $\mathfrak{c}_{h}^{\textit{skew}}(\bu_h;T_h,E_h) = \mathfrak{c}_{h}^{\textit{skew}}(\bu_h;T_h -T_I,E_h) + \mathfrak{c}_{h}^{\textit{skew}}(\bu_h;T_I,E_h) = \mathfrak{c}_{h}^{\textit{skew}}(\bu_h;T_I,E_h)$, and add and subtract $\mathfrak{c}_{h}^{\textit{skew}}(\bu;T,E_h)$ to rewrite $\I\I$ as follows:
\begin{equation*}
 \I \I
 = \mathfrak{c}_{h}^{\textit{skew}}(\bu_h;T_I - T,E_h)
 +
 \mathfrak{c}_{h}^{\textit{skew}}(\bu_h - \bu;T,E_h)
 +
 \mathfrak{c}_{h}^{\textit{skew}}(\bu;T,E_h)
 -
 \mathfrak{c}^{\textit{skew}}(\bu;T,E_h).
\end{equation*}
Define $\I\I_1:= \mathfrak{c}_{h}^{\textit{skew}}(\bu_h;T_I - T,E_h)
 +
 \mathfrak{c}_{h}^{\textit{skew}}(\bu_h - \bu;T,E_h)$. To bound $\I\I_1$, we first use the definition of $\mathfrak{c}_h^{\textit{skew}}(\cdot;\cdot,\cdot)$, H\"{o}lder inequality, Lemma \ref{prop:stabPi}, \eqref{eq:inversepoly} and \eqref{eq:embedding} to obtain 
\begin{multline*}
\mathfrak{c}_{h}^{\textit{skew}}(\bu_h;T_I - T,E_h) = \dfrac{1}{2}\left(\mathfrak{c}_h(\bu_h;T_I-T,E_h) - \mathfrak{c}_h(\bu_h;E_h,T_I-T) \right) \\
\leq C\|\bu_h\|_{\L^2(\O)^2}\left(|(T-T_I)|_{\W^{1,p}(\O)}\|E_h\|_{\L^{\frac{2p}{p-2}}(\O)} + \|\bPi_k^0\nabla E_h\|_{\L^\infty(\O)^2}\|T-T_I\|_{\L^{2}(\O)}\right) \\
\leq Ch^s\|\bu_h\|_{\L^2(\O)^2}|T|_{\W^{1+s,p}(\O)}|E_h|_{\H^1(\O)} \\
\leq \dfrac{C}{\alpha_2\nu_*}h^s\|\bb{f}\|_{\L^2(\O)^2}|T|_{\W^{1+s,p}(\O)}|E_h|_{\H^1(\O)}.
\end{multline*}
In a similar way, we have
\begin{multline*}
\mathfrak{c}_{h}^{\textit{skew}}(\bu_h - \bu;T,E_h) = \dfrac{1}{2}\left(\mathfrak{c}_h(\bu_h-\bu;T,E_h) - \mathfrak{c}_h(\bu_h-\bu;E_h,T) \right) \\
\leq C\|\bu-\bu_h\|_{\L^2(\O)^2}\left(|T|_{\W^{1,\infty}(\O)}\|E_h\|_{\L^2(\O)} + \|\nabla E_h\|_{\L^2(\O)}\|T\|_{\L^{\infty}(\O)}\right) \\
\leq C\|\bu-\bu_h\|_{\L^2(\O)^2}\|T\|_{\W^{1,\infty}(\O)}|E_h|_{\H^1(\O)} \\
\leq Ch^s\|\bu\|_{\H^s(\O)^2}\|T\|_{\W^{1,\infty}(\O)}|E_h|_{\H^1(\O)} + C\|\bb{e}_h\|_{\L^2(\O)^2}\|T\|_{\W^{1,\infty}(\O)}|E_h|_{\H^1(\O)}.
\end{multline*}
Hence, we derive the following bound for $\textrm{II}_1$:
\begin{equation}
\label{eq:II_1}
\textrm{II}_1
\leq
\mathfrak{A}(T)\|\bb{e}_h\|_{\L^2(\O)^2}|E_h|_{\H^1(\O)} + \mathfrak{A}(\bu,T,\bb{f})h^s|E_h|_{\H^1(\O)}.
\end{equation}
Now, let us define $\mathrm{II}_2:=\mathfrak{c}_{h}^{\textit{skew}}(\bu;T,E_h) - \mathfrak{c}^{\textit{skew}}(\bu;T,E_h)$, and consider its local contributions $\I\I_2^E:= \mathfrak{c}_h^{\textit{skew},E}(\bu;T,E_h)-\mathfrak{c}^{\textit{skew},E}(\bu;T,E_h)$, for $E \in \O_h$. We note that
\begin{multline}
\I\I_2^{E}
=
\frac{1}{2} \left[
\int_E ( \bPi^{0,E}_k \bu \cdot \bPi^{0,E}_{k}\nabla T ) \Pi^{0,E}_{k+1} E_h
-
\int_E  ( \bu \cdot \nabla T ) E_h
\right]
\\
-
\frac{1}{2}
\left[
\int_E ( \bPi^{0,E}_k \bu \cdot \bPi^{0,E}_{k} \nabla E_h ) \Pi^{0,E}_{k+1} T
-
\int_E  ( \bu \cdot \nabla E_h ) T
\right]
=:
\frac{1}{2} \I\I_{2,a}^{E} - \frac{1}{2} \I\I_{2,b}^{E}.
\end{multline}
Then, we have 
\begin{multline*}
\sum_{E\in\O_h} \I\I_{2,a}^{E} = \sum_{E\in\O_h} \sum_{i=1}^2 \int_{E} \left[ \left(\Pi_k^{0,E}u_i \Pi_k^{0,E}\dfrac{\partial T}{\partial x_i}\right)\Pi_{k+1}^{0,E}E_h - \left(u_i\dfrac{\partial T}{\partial x_i}\right)E_h\right] \\
= \sum_{E\in\O_h} \sum_{i=1}^2 \int_{E} \left[ \left(\Pi_k^{0,E}u_i \Pi_k^{0,E}\dfrac{\partial T}{\partial x_i} - u_i\dfrac{\partial T}{\partial x_i}\right)\Pi_{k+1}^{0,E}E_h - \left(u_i\dfrac{\partial T}{\partial x_i}\right)(E_h - \Pi_{k+1}^{0,E}E_h)\right] \\
= \sum_{E\in\O_h} \sum_{i=1}^2 \left\lbrace \int_{E} \left(\Pi_k^{0,E}u_i - u_i\right)\dfrac{\partial T}{\partial x_i}\Pi_{k+1}^{0,E}E_h + \int_E \Pi_k^{0,E}u_i\left(\Pi_k^{0,E}\dfrac{\partial T}{\partial x_i} - \dfrac{\partial T}{\partial x_i}\right)\Pi_{k+1}^{0,E}E_h \right. \\
+ \left. \int_E\left(\Pi_{k}^{0,E}u_i - u_i\right)\dfrac{\partial T}{\partial x_i}(E_h - \Pi_{k+1}^{0,E}E_h) \right. \\ 
\left. + \int_E \Pi_k^{0,E}u_i\left(\Pi_1^{0,E}\dfrac{\partial T}{\partial x_i} - \dfrac{\partial T}{\partial x_i}\right)(E_h-\Pi_{k+1}^{0,E}E_h) \right\rbrace.
\end{multline*}
Next, using H\"{o}lder inequality and Lemmas \ref{prop:bramblehilbert} and \ref{prop:stabPi}, we obtain the following bounds
{\small \begin{equation}\label{1a}
\begin{split}
\int_{E} \left(\Pi_k^{0,E}u_i - u_i\right)\dfrac{\partial T}{\partial x_i}\Pi_{k+1}^{0,E}E_h &\leq Ch_E^s\|u_i\|_{\H^s(E)}\left\|\dfrac{\partial T}{\partial x_i}\right\|_{\L^p(E)}\|E_h\|_{\L^{\frac{2p}{p-2}}(E)}, \\
\int_E \Pi_k^{0,E}u_i\left(\Pi_k^{0,E}\dfrac{\partial T}{\partial x_i} - \dfrac{\partial T}{\partial x_i}\right)\Pi_{k+1}^{0,E}E_h &\leq Ch_E^s\|u_i\|_{\L^2(E)}\left\|\dfrac{\partial T}{\partial x_i}\right\|_{\W^{s,p}(E)}\|E_h\|_{\L^{\frac{2p}{p-2}}(E)}, \\
\int_E\left(\Pi_{k}^{0,E}u_i - u_i\right)\dfrac{\partial T}{\partial x_i}(E_h - \Pi_{k+1}^{0,E}E_h) &\leq Ch_E^s\|u_i\|_{\H^s(E)}\left\|\dfrac{\partial T}{\partial x_i}\right\|_{\L^p(E)}\|E_h\|_{\L^{\frac{2p}{p-2}}(E)}, \\
\int_E \Pi_k^{0,E}u_i\left(\Pi_1^{0,E}\dfrac{\partial T}{\partial x_i} - \dfrac{\partial T}{\partial x_i}\right)(E_h-\Pi_{k+1}^{0,E}E_h) &\leq Ch_E\|u_i\|_{\L^2(E)}\left\|\dfrac{\partial T}{\partial x_i}\right\|_{\W^{1,p}(E)}\|E_h\|_{\L^{\frac{2p}{p-2}}(E)}.
\end{split}
\end{equation}}
Hence, from \eqref{1a} and using \eqref{eq:Poincare}, we deduce that 
\begin{equation}\label{IIa}
\begin{split}
\sum_{E\in\O_h}\sum_{i=1}^2 \dfrac{1}{2}\I\I_{2,a}^E &\leq Ch(\|\bu\|_{\H^s(\O)^2}+\|\bb{f}\|_{\L^2(\O)^2})\|T\|_{\W^{1+s,p}(\O)}|E_h|_{\H^1(\O)} \\
&\leq \mathfrak{A}(\bu,T,\bb{f})h|E_h|_{\H^1(\O)}.
\end{split}
\end{equation}
In a similar way, we have 
\begin{multline*}
\sum_{E\in\O_h} \I\I_{2,b}^{E} = \sum_{E\in\O_h} \sum_{i=1}^2 \int_{E} \left[ \left(\Pi_k^{0,E}u_i \Pi_k^{0,E}\dfrac{\partial E_h}{\partial x_i}\right)\Pi_{k+1}^{0,E}T - \left(u_i\dfrac{\partial E_h}{\partial x_i}\right)T\right] \\
= \sum_{E\in\O_h} \sum_{i=1}^2 \int_{E} \left[ \left(\Pi_k^{0,E}u_i - u_i\right) \Pi_k^{0,E}\dfrac{\partial E_h}{\partial x_i}\Pi_{k+1}^{0,E}T - u_i\left(\dfrac{\partial E_h}{\partial x_i}T - \Pi_k^{0,E}\dfrac{\partial E_h}{\partial x_i}\Pi_{k+1}^{0,E}T\right)\right] \\
= \sum_{E\in\O_h} \sum_{i=1}^2 \left\lbrace \int_{E}  \left(\Pi_k^{0,E}u_i - u_i\right) \Pi_k^{0,E}\dfrac{\partial E_h}{\partial x_i}\Pi_{k+1}^{0,E}T + \int_{E} u_i\Pi_k^{0,E}\dfrac{\partial E_h}{\partial x_i}\left(\Pi_{k+1}^{0,E}T - T\right) \right. \\
\left. + \int_E u_i\left(\Pi_k^{0,E}\dfrac{\partial E_h}{\partial x_i}-\dfrac{\partial E_h}{\partial x_i}\right)T\right\rbrace \\
= \sum_{E\in\O_h} \sum_{i=1}^2 \left\lbrace \int_{E}  \left(\Pi_k^{0,E}u_i - u_i\right) \Pi_k^{0,E}\dfrac{\partial E_h}{\partial x_i}\Pi_{k+1}^{0,E}T + \int_{E} u_i\Pi_k^{0,E}\dfrac{\partial E_h}{\partial x_i}\left(\Pi_{k+1}^{0,E}T - T\right) \right. \\
\left. + \int_E \left(u_i-\Pi_k^{0,E}u_i\right)\left(\Pi_k^{0,E}\dfrac{\partial E_h}{\partial x_i}-\dfrac{\partial E_h}{\partial x_i}\right)T \right. \\
\left. + \int_E \Pi_k^{0,E}u_i\left(\Pi_k^{0,E}\dfrac{\partial E_h}{\partial x_i}-\dfrac{\partial E_h}{\partial x_i}\right)\left(T-\Pi_0^{0,E}T\right)\right\rbrace .
\end{multline*}
Invoking similar arguments to obtain the bounds in \eqref{1a}, we deduce that
{\small \begin{equation}\label{1b}
\begin{split}
\int_{E}  \left(\Pi_k^{0,E}u_i - u_i\right) \Pi_k^{0,E}\dfrac{\partial E_h}{\partial x_i}\Pi_{k+1}^{0,E}T &\leq Ch_E^s\|u_i\|_{\H^s(E)}\left\|\dfrac{\partial E_h}{\partial x_i}\right\|_{\L^2(E)}\|T\|_{\L^\infty(E)}, \\
\int_E u_i\Pi_k^{0,E}\dfrac{\partial E_h}{\partial x_i}\left(\Pi_{k+1}^{0,E}T - T\right) &\leq Ch_E^s\|u_i\|_{\L^2(E)}\left\|\dfrac{\partial E}{\partial x_i}\right\|_{\L^{2}(E)}\|T\|_{\H^{1+s}(E)}, \\
\int_E \left(u_i-\Pi_k^{0,E}u_i\right)\left(\Pi_k^{0,E}\dfrac{\partial E_h}{\partial x_i}-\dfrac{\partial E_h}{\partial x_i}\right)T &\leq Ch_E^s\|u_i\|_{\H^s(E)}\left\|\dfrac{\partial E}{\partial x_i}\right\|_{\L^2(E)}\|T\|_{\L^\infty(E)}, \\
\int_E \Pi_k^{0,E}u_i\left(\Pi_k^{0,E}\dfrac{\partial E_h}{\partial x_i}-\dfrac{\partial E_h}{\partial x_i}\right)\left(T-\Pi_0^{0,E}T\right) &\leq Ch_E\|u_i\|_{\L^2(E)}\left\|\dfrac{\partial E_h}{\partial x_i}\right\|_{\L^{2}(E)}\|T\|_{\W^{1,\infty}(E)}.
\end{split}
\end{equation}}
Therefore, from \eqref{1b} we derive the following estimate
\begin{equation}\label{IIb}
\sum_{E\in\O_h}\sum_{i=1}^2 \dfrac{1}{2}\I\I_{2,b}^E \leq Ch(\|\bu\|_{\H^s(\O)^2}+\|\bb{f}\|_{\L^2(\O)^2})|T|_{\W^{1+s,p}(\O)}|E_h|_{\H^1(\O)},
\end{equation}
and we conclude from \eqref{IIa} and \eqref{IIb} that
\begin{equation}\label{eq:II_2}
\I\I_2 \leq \mathfrak{A}(\bu,T,\bb{f})h|E_h|_{\H^1(\O)}.
\end{equation}
Finally, we derive from \eqref{eq:II_1} and \eqref{eq:II_2} the following bound for $\I\I$:
\begin{equation}\label{estimateII}
\I\I \leq \mathfrak{A}(T)\|\bb{e}_h\|_{\L^2(\O)^2}|E_h|_{\H^1(\O)} + \mathfrak{A}(\bu,T,\bb{f})h|E_h|_{\H^1(\O)}.
\end{equation}
To estimate $\I\I\I$, we invoke Lemma \ref{prop:bramblehilbert} and the definition of $\Pi_{k+1}^{0,E}$ to obtain
\begin{equation}\label{estimateIII}
\I\I\I \leq Ch^{1+s}|g|_{\H^{1+s}(\O)}|E_h|_{\H^1(\O)}.
\end{equation}
Finally, from \eqref{eq:IT}, \eqref{estimateII} and \eqref{estimateIII} we obtain
\begin{equation}\label{estimafinal1}
\left(\widetilde{\alpha}_2\kappa_* - \mathfrak{A}(T)\right)|E_h|_{\H^1(\O)} \leq \mathfrak{A}(\bu,T,\bb{f},g,\kappa)h + \mathfrak{A}(T)\|\bb{e}_h\|_{\L^2(\O)^2},
\end{equation}
where $\widetilde{\alpha}_2\kappa_* - \mathfrak{A}(T)>0$. This concludes the proof
\end{proof}

\begin{lemma}\label{Lema2vel}
Let us suppose that \textbf{A0}), \textbf{A2}), \textbf{A3}), and \textbf{A4}) hold. Assume the smallness assumptions given by Theorems \ref{eq:uniqcont} and \ref{eq:uniqueness}. Let $(\bu,\textsf{p},T) \in \bV_0 \times \Q \times \H$ and $(\bu_h,\textsf{p}_h,T_h) \in \bV_h\times\Q_h\times\H_h$ be the unique solutions of problems \eqref{eq:DHW2} and \eqref{eq:DHW2discreto}, respectively, and let $T_I$ and $\bu_I$ be the interpolants of $T\in\H_h$ and $\bu\in\bV_h$ given by Lemmas \ref{estimate2gen}, then we have
\begin{equation*}
\alpha_2\nu_*\|\bb{e}_h\|_{\L^2(\O)^2} \leq \mathfrak{A}(\bu,T,\bb{f},\nu)h^s + \mathfrak{A}(\bu)|E_h|_{\H^1(\O)},
\end{equation*}
\end{lemma}
where $E_h\coloneqq T_I-T_h$ and $\bb{e}_h\coloneqq\bu_I-\bu_h$.
\begin{proof}
Invoking the coercivity property \eqref{eq:ah_bound} for $a_h(\cdot;\cdot,\cdot)$ and add and subtract the term $a(T;\bu,\boldsymbol{e}_h)$ we obtain
\begin{equation}\label{coercividaderror}
\begin{aligned}
\alpha_2\nu_*\|\boldsymbol{e}_h\|_{\L^2(\O)}^2
&
\leq a_h(T_h;\boldsymbol{e}_h,\boldsymbol{e}_h) = a_h(T_h;\bu_I,\boldsymbol{e}_h)-a_h(T_h;\bu_h,\boldsymbol{e}_h)
\\
&
=
a_h(T_h;\bu_I,\boldsymbol{e}_h)
-
a(T;\bu,\boldsymbol{e}_h)
+
a(T;\bu,\boldsymbol{e}_h)
-
a_h(T_h;\bu_h,\boldsymbol{e}_h)
\\
&
=
a_h(T_h;\bu_I,\bb{e}_h) - a(T;\bu,\bb{e}_h) + \left(\bb{f}-\bPi_k^0\bb{f},\bb{e}_h\right)_{\L^2(\O)^2} = \I\V + \V.
\end{aligned}
\end{equation}
Now, to estimate $\I\V$  we use the definitions of $a_h(\cdot;\cdot,\cdot)$ and $a(\cdot;\cdot,\cdot)$ in \eqref{eq:bilinear_form_a} and  \eqref{eq:bilinear_form_fraka}, respectively, to arrive at the following identities:
{\small \begin{multline*}
\I\V = \sum_{E\in\O_h} \int_E \left(\nu(\Pi_{k+1}^{0,E}T_h)\bPi_k^{0,E}\bu_I\cdot\bPi_k^{0,E}\bb{e}_h - \nu(T)\bu\cdot\bb{e}_h\right) \\
+ \sum_{E\in\O_h}\nu(\Pi_0^{0,E}T)S_1^E((\bb{I}-\bPi_k^{0,E})\bu_I,(\mathbf{I}-\bPi_k^{0,E})\bb{e}_h) \\
= \sum_{E\in\O_h} \int_E \left(\nu(\Pi_{k+1}^{0,E}T_h)\bPi_k^{0,E}(\bu_I-\bu)\cdot\bPi_k^{0,E}\bb{e}_h + \int_E \nu(\Pi_{k+1}^{0,E}T)\bu\cdot\bPi_k^{0,E}\bb{e}_h - \nu(T)\bu\cdot\bb{e}_h\right) \\
+ \sum_{E\in\O_h}\nu(\Pi_0^{0,E}T)S_1^E((\mathbf{I}-\bPi_k^{0,E})\bu_I,(\mathbf{I}-\bPi_k^{0,E})\bb{e}_h) \\
= \sum_{E\in\O_h} \left\lbrace \int_E \nu(\Pi_{k+1}^{0,E}T_h)\bPi_k^{0,E}(\bu_I-\bu)\cdot\bPi_k^{0,E}\bb{e}_h + \int_E \left[\nu(\Pi_{k+1}^{0,E}T_h) - \nu(T)\right]\bu\cdot\bPi_k^{0,E}\bb{e}_h\right. \\ 
\left. + \int_E \nu(T)\bu\cdot(\bPi_k^{0,E}\bb{e}_h - \bb{e}_h) + \nu(\Pi_0^{0,E}T_h)S_V^E((\mathbf{I}-\bPi_k^{0,E})\bu_I,(\mathbf{I}-\bPi_k^{0,E})\bb{e}_h)\right\rbrace \\
= \sum_{E\in\O_h} \I\V_1^E + \I\V_2^E + \I\V_3^E + \I\V_4^E.
\end{multline*}}
Then, using H\"{o}lder's inequality, Lemmas \ref{prop:bramblehilbert} and \ref{prop:stabPi} and triangle inequality, we obtain
{\small \begin{multline*}
\sum_{E\in\O_h}\I\V_1^E + \I\V_2^E + \I\V_4^E \leq C\nu^*\|\bu-\bu_I\|_{\L^2(\O)}\|\bb{e}_h\|_{\L^2(\O)}  \\
+ C\nu^*\|(\mathbf{I}-\bPi_k^0)\bu_I\|_{\L^2(\O)^2}\|(\mathbf{I}-\bPi_k^0)\bb{e}_h\|_{\L^2(\O)^2} + \nu_{\text{lip}}\|T-\Pi_{k+1}^0T_h\|_{\L^6(\O)}\|\bu\|_{\L^3(\O)^2}\|\bb{e}_h\|_{\L^2(\O)^2} \\
\leq C\left(\nu^*h^s|\bu|_{\H^s(\O)^2} + Ch^{1+s}|T|_{\W^{1+s,6}(\O)}|\bu|_{\H^s(\O)^2}\right)\|\bb{e}_h\|_{\L^2(\O)^2} \\
+ |\bu|_{\H^s(\O)^2}\|\bb{e}_h\|_{\L^2(\O)^2}|E_h|_{\H^1(\O)} \\
\leq \mathfrak{A}(\bu,T)h^s\|\bb{e}_h\|_{\L^2(\O)^2} + \mathfrak{A}(\bu)\|\bb{e}_h\|_{\L^2(\O)^2}|E_h|_{\H^1(\O)}.
\end{multline*}}
Next, using the definition of $\bPi_k^{0,E}$ and Lemma \ref{prop:bramblehilbert} we obtain
\begin{multline*}
\sum_{E\in\O_h} \I\V_3^E = \sum_{E\in\O_h} \int_E \nu(T)\bu\cdot\left(\bPi_k^{0,E}\bb{e}_h-\bb{e}_h\right) \\
= \sum_{E\in\O_h} \left(\nu(T)\bu - \bPi_k^{0,E}(\nu(T)\bu)\right)\left(\bPi_k^{0,E}\bb{e}_h-\bb{e}_h\right) \leq \mathfrak{A}(\bu,\nu)h^s\|\bb{e}_h\|_{\L^2(\O)^2}.
\end{multline*}
Hence, we derive the following bound for $\I\V$:
\begin{equation}\label{estimateIV}
\I\V \leq \mathfrak{A}(\bu,T,\nu)h^s\|\bb{e}_h\|_{\L^2(\O)^2} + \mathfrak{A}(\bu)\|\bb{e}_h\|_{\L^2(\O)^2}|E_h|_{\H^1(\O)}.
\end{equation}
To estimate $\V$, we use Lemma \ref{prop:bramblehilbert}, obtaining
\begin{equation}\label{estimateV}
\V \lesssim h^s|\bb{f}|_{\H^s(\O)^2}\|\bb{e}_h\|_{\L^2(\O)^2}.
\end{equation}
From \eqref{coercividaderror}, \eqref{estimateIV} and \eqref{estimateV} we deduce that 
\begin{equation}\label{estimatee_h}
\alpha_2\nu_*\|\bb{e}_h\|_{\L^2(\O)^2} \leq \mathfrak{A}(\bu,T,\bb{f},\nu)h^s + \mathfrak{A}(\bu)|E_h|_{\H^1(\O)}.
\end{equation}
This concludes the proof.
\end{proof}

\begin{lemma}\label{Lema3press}
Let us suppose that \textbf{A0}), \textbf{A1}), \textbf{A2}), \textbf{A3}), and \textbf{A4}) hold. Assume the smallness assumptions given by  Theorems \ref{eq:uniqcont} and \ref{eq:uniqueness}. Let $(\bu,\textsf{p},T) \in \bV_0 \times \Q \times \H$ and $(\bu_h,\textsf{p}_h,T_h) \in \bV_h\times\Q_h\times\H_h$ be the unique solutions of problems \eqref{eq:DHW2} and \eqref{eq:DHW2discreto}, respectively. Let us consider $\textsf{p}_I|_{E} \coloneqq \Pi^{0,E}_{k-1} \textsf{p}$, for each $E\in\O_h$. Then, we have
\begin{equation*}
\beta\|\textsf{e}_h\|_{\L^2(\O)} \leq \mathfrak{A}(\bu,T,\textsf{p},\bb{f},g,\kappa,\nu)h.
\end{equation*}
\end{lemma}
\begin{proof}
First, let $\bv_h \in \bV_h$ and use the definition of $\textsf{e}_h$, namely $\textsf{e}_h = \textsf{p}_I - \textsf{p}_h$ and the first equations of the continuous and discrete systems \eqref{eq:DHW2} and \eqref{eq:DHW2discreto}, respectively, to obtain
\begin{multline*}
b(\bv_h,\textsf{e}_h)=b(\bv_h, \textsf{p}_I)-b(\bv_h, \textsf{p}_h)
=
b(\bv_h, \textsf{p}_I -\textsf{p})
+
b(\bv_h, \textsf{p})-b(\bv_h, \textsf{p}_h)
=
b(\bv_h, \textsf{p}_I- \textsf{p})
\\
+
\left[a_h(T_h;\bu_h,\bv_h)-a(T;\bu,\bv_h)\right]
+
(F-F_h)\bv_h.
\end{multline*}

Let us define $\I:=a_h(T_h;\bu_h,\bv_h)-a(T;\bu,\bv_h)$. Simple algebraic manipulations yield to the following identity
{\small \begin{multline}
\I = \sum_{E\in\O_h} \left\lbrace \int_E \nu(\Pi_{k+1}^{0,E}T_h)\bPi_k^{0,E}\bu_h\cdot\bPi_k^{0,E}\bv_h - \int_E \nu(T)\bu\cdot\bv \right\rbrace  \\
+ \sum_{E\in\O_h} \nu(\Pi_0^{0,E}T_h)S_V^E((\mathbf{I}-\bPi_k^{0,E})\bu_h,(\mathbf{I}-\bPi_k^{0,E})\bv_h) \\
= \sum_{E\in\O_h} \left\lbrace \int_E \nu(\Pi_{k+1}^{0,E}T_h)\left(\bPi_k^{0,E}\bu_h - \bu\right)\cdot\bPi_k^{0,E}\bv_h + \int_E \left(\nu(\Pi_{k+1}^{0,E}T_h) - \nu(T) \right) \bu\cdot\bPi_k^{0,E}\bv_h \right. \\
\left. + \int_E \nu(T)\bu\cdot\left(\bPi_k^{0,E}\bv_h-\bv_h\right) + \nu(\Pi_0^{0,E}T_h)S_V^E((\mathbf{I}-\bPi_k^{0,E})\bu_h,(\mathbf{I}-\bPi_k^{0,E})\bv_h)\right\rbrace \\
=\sum_{E\in\O_h} \I_A^E+\I\I_A^E+\I\I\I_A^E+\I\V_A^E.
\label{eq:K_1}
\end{multline}}
Now, we need to estimate the contributions $\I_A^E, \I\I_A^E, \I\I\I_A^E, \I\V_A^E$. To estimate $\I_A^E$, we invoke assumption \textbf{A0}), Cauchy-Schwarz and triangle inequality, Proposition \ref{prop:stabPi} and Lemma \ref{eq:Poincare}, obtaining
\begin{multline}
\sum_{E\in\O_h} \I_A^E \leq \nu^{*}\sum_{E\in\O_h} \|\bu-\bPi_k^{0,E}\bu_h\|_{\L^2(E)^2}\|\bPi_k^{0,E}\bv_h\|_{\L^2(E)^2} \\
\leq \nu^{*}\sum_{E\in\O_h}\left(\|\bu- \bPi_k^{0,E}\bu\|_{\L^2(E)^2} + \|\bPi_k^{0,E}(\bu-\bu_h)\|_{\L^2(E)^2}\right)\|\bv_h\|_{\L^2(E)^2} \\
\leq C\nu^{*}h^s|\bu|_{\H^s(\O)^2}\|\bv_h\|_{\L^2(\O)^2} \\ 
+ C\nu^{*}\|\bu-\bu_h\|_{\L^2(\O)^2}\|\bv_h\|_{\L^2(\O)^2}.
\label{eq:IA}
\end{multline}
To estimate $\I\I_A^E$, we use the Lipschitz-continuity of $\nu(\cdot)$, H\"{o}lder and triangle inequality, Proposition \ref{prop:stabPi} and Lemma \ref{prop:bramblehilbert} in order to obtain
{\small \begin{multline}
\sum_{E\in\O_h} \I\I_A^E \leq \nu_{\text{lip}}\sum_{E\in\O_h} \|T-\Pi_{k+1}^{0,E}T_h\|_{\L^{\frac{2p}{p-2}}(E)}\|\bu\|_{\L^p(E)^2}\|\bPi_k^{0,E}\bv_h\|_{\L^{2}(E)^2} \\
\leq \nu_{\text{lip}}\sum_{E\in\O_h}\left(\|T- \Pi_{k+1}^{0,E}T\|_{\L^{\frac{2p}{p-2}}(E)} + \|\Pi_{k+1}^{0,E}(T-T_h)\|_{\L^{\frac{2p}{p-2}}(E)^2}\right)\|\bu\|_{\L^p(E)^2}\|\bv_h\|_{\L^2(E)^2} \\
\leq C\textsf{C}_2\nu_{\text{lip}}\|\bu\|_{\H^s(\O)^2}|T|_{\W^{1+s,p}(\O)}h^{s+1}\|\bv_h\|_{\L^2(\O)^2} + \\
C\textsf{C}_{\frac{2p}{p-2}}\nu_{\text{lip}}\|\bu\|_{\L^2(\O)^2}|T-T_h|_{\H^1(\O)}\|\bv_h\|_{\L^2(\O)^2}.
\label{eq:IIA}
\end{multline}}
For $\I\I\I_A^E$, we proceed as before: Using the definition of $\bPi_k^{0,E}$ and Lemma \ref{prop:bramblehilbert} we obtain
\begin{equation}\label{eq:IIIA}
\begin{split}
\sum_{E\in\O_h} \I\I\I_A^E &= \sum_{E\in\O_h} \int_E \nu(T)\bu\cdot\left(\bPi_k^{0,E}\bv_h-\bv_h\right) \\
&= \sum_{E\in\O_h} \left(\nu(T)\bu - \bPi_k^{0,E}(\nu(T)\bu)\right)\left(\bPi_k^{0,E}\bv_h-\bv_h\right) \\
&\leq C\|\nu\|_{\W^{s,\infty}(\O)}|\bu|_{\H^s(\O)^2}h^s\|\bv_h\|_{\L^2(\O)^2}.
\end{split}
\end{equation}
Now, to estimate $\I\V_A^E$ we invoke Proposition \ref{prop:stabPi}, Lemma \ref{prop:bramblehilbert}, Cauchy-Schwarz and triangle inequality, obtaining 
\begin{equation}\label{eq:IVA}
\begin{split}
\sum_{E\in\O_h} \I\V_A^E &\leq \nu^*\sum_{E\in\O_h} \|\bu_h-\bPi_k^{0,E}\bu_h\|_{\L^2(E)^2}\|\bv_h-\bPi_k^{0,E}\bv_h\|_{\L^2(E)^2} \\
&\leq \nu^*\sum_{E\in\O_h} \left(2\|\bu-\bu_h\|_{\L^2(E)^2}+\|\bu-\bPi_k^{0,E}\bu\|_{\L^2(E)^2}\right)\|\bv_h\|_{\L^2(E)^2} \\
&\leq C\nu^*h^s|\bu|_{\H^s(\O)^2}\|\bv_h\|_{\L^2(\O)^2} + C\nu^*\|\bu-\bu_h\|_{\L^2(\O)^2}\|\bv_h\|_{\L^2(\O)^2}.
\end{split}
\end{equation}
Therefore, gathering \eqref{eq:IA}, \eqref{eq:IIA}, \eqref{eq:IIIA} and \eqref{eq:IVA} we conclude that 
\begin{equation}\label{Ifinal}
\I \leq \mathfrak{A}(\bu,T,\nu)h^s\|\bv_h\|_{\L^2(\O)^2} + C\left(\|\bu-\bu_h\|_{\L^2(\O)^2}+\mathfrak{A}(\bu)|T-T_h|_{\H^1(\O)}\right)\|\bv_h\|_{\L^2(\O)^2}.
\end{equation}
Now, the terms $\I\I:=b(\bv_h,\textsf{p}_I-\textsf{p})$ and $\I\I\I:= (F-F_h)\bv_h$ can be easily estimated. In fact, for $\I\I$ we use Cauchy-Schwarz and Lemma \ref{prop:bramblehilbert}, obtaining
\begin{equation}\label{IIfinal}
\I\I \leq \|\bv_h\|_{\H(\div,\O)}\|\textsf{p}-\textsf{p}_I\|_{\L^2(\O)} \leq \mathfrak{A}(\textsf{p})h^s\|\bv_h\|_{\H(\div,\O)},
\end{equation}
whereas for $\I\I\I$ we proceed in a similar way, obtaining
\begin{equation}\label{IIIfinal}
\I\I\I \leq \|\bb{f}-\bPi_k^{0,E}\bb{f}\|_{\L^2(\O)^2}\|\bv_h\|_{\L^2(\O)^2} \leq \mathfrak{A}(\bb{f})h^s\|\bv_h\|_{\L^2(\O)^2}. 
\end{equation}
Therefore, gathering \eqref{eq:infsupdiscreta}, \eqref{Ifinal}, \eqref{IIfinal} and \eqref{IIIfinal} we deduce that 
\begin{equation}\label{presfinal}
\beta\|\textsf{e}_h\|_{\L^2(\O)} \leq \mathfrak{A}(\bu,T,\textsf{p},\nu,\bb{f})h^s + C\|\bu-\bu_h\|_{\L^2(\O)^2} \\ 
+\mathfrak{A}(\bu)|T-T_h|_{\H^1(\O)}.
\end{equation}
\end{proof}

Now, we are in position to prove the main result of this section.

{\begin{theorem}[error estimates]
Let us suppose that assumptions \textbf{A0}), \textbf{A1}), \textbf{A2}), \textbf{A3}), and \textbf{A4}) hold. Assume the smallness assumptions given by Theorems \ref{eq:uniqcont} and \ref{eq:uniqueness}. Let $(\bu,\textsf{p},T) \in \bV_0 \times \Q \times \H$ and $(\bu_h,\textsf{p}_h,T_h) \in \bV_h\times\Q_h\times\H_h$ be the unique solutions of problems \eqref{eq:DHW2} and \eqref{eq:DHW2discreto}, respectively. If\begin{equation*}
\label{eq:condicionerror}
\begin{split}
\widetilde{\alpha}_2\kappa_* - \mathfrak{A}(T)>0,\\
\alpha_2\nu_* - \mathfrak{A}(\bu,T)\left(\widetilde{\alpha}_2\kappa_* - \mathfrak{A}(T)\right)>0,
\end{split}
\end{equation*}
then the following a priori error estimates hold
\begin{equation*}
\begin{aligned}
\|\bu-\bu_h\|_{0,\O}+|T-T_h|_{1,\O}
& \leq
\mathfrak{A}(\bu,T,\bb{f},g,\kappa,\nu)h
\\
\| \textsf{p}- \textsf{p}_h\|_{0,\O}
& \leq
\mathfrak{A}(\bu,T,p,\bb{f},g,\kappa,\nu)h.
\end{aligned}
\end{equation*}
\label{thm:error_estimates}
\end{theorem}
\begin{proof}
For this, we combine Lemmas \ref{Lema1temp}, \ref{Lema2vel} and Lemma \ref{Lema3press}. We consider $\bu_I \in \bV_h$ and $T_I \in \H_h$ as the interpolants of $\bu \in \bV$ and $T \in \H$ given by Lemmas \ref{prop:errorinterp} and \ref{estimate2gen}, respectively. We also introduce $ \textsf{p}_I \in \Q_h$ as follows: for each $E \in \O_h$, $ \textsf{p}_I|_{E} \coloneqq \Pi^{0,E}_{k-1} \textsf{p}$. As a direct consequence of Lemmas \ref{prop:errorinterp}, \ref{prop:bramblehilbert} and Lemma \ref{estimate2gen} combined with \eqref{eq:embedding}, we obtain
\begin{equation*}\label{eq:int1}
\begin{split}
\|\bu-\bu_I\|_{\L^2(\O)^2} &\leq Ch^{s}|\bu|_{\H^s(\O)^2}, \quad |T-T_I|_{\H^1(\O)} \leq Ch^{s}|T|_{\W^{s+1,p}(\O)}, \\
\quad \| \textsf{p}- \textsf{p}_I\|_{\L^2(\O)} &\leq Ch^{s}| \textsf{p}|_{\H^s(\O)}.
\end{split}
\end{equation*}
Let us define $E_h\coloneqq T_I-T_h$, $\bb{e}_h\coloneqq \bu_I-\bu_h$ and $\textsf{e}_h\coloneqq \textsf{p}_I-\textsf{p}_h$, and recall from Lemmas \ref{Lema1temp} and \ref{Lema2vel} the estimates \eqref{estimafinal1} and \eqref{estimatee_h}. Now, replacing \eqref{estimafinal1} into \eqref{estimatee_h}, we obtain
\begin{equation*}
\left(\alpha_2\nu_* - \mathfrak{A}(\bu,T)\left(\widetilde{\alpha}_2\kappa_* - \mathfrak{A}(T)\right)\right)\|\bb{e}_h\|_{\L^2(\O)^2} \leq \mathfrak{A}(\bu,T,\bb{f},g,\kappa,\nu)h,
\end{equation*}
where, defining $\mathfrak{C}(\bu,T) \coloneqq\left(\alpha_2\nu_* - \mathfrak{A}(\bu,T)\left(\widetilde{\alpha}_2\kappa_* - \mathfrak{A}(T)\right)\right)^{-1}$, we deduce that 
\begin{equation}\label{errore_h}
\|\bb{e}_h\|_{\L^2(\O)^2} \leq \mathfrak{C}(\bu,T)\mathfrak{A}(\bu,T,\bb{f},g,\kappa,\nu)h.
\end{equation}
Next, replacing \eqref{errore_h} into \eqref{estimafinal1} we obtain
\begin{equation}\label{errorT_h}
|E_h|_{\H^1(\O)} \leq \mathfrak{A}(\bu,T,\bb{f},g,\kappa,\nu)h.
\end{equation}
Finally, invoking Lemma \ref{Lema3press}, we combine \eqref{errore_h} and \eqref{errorT_h} with \eqref{presfinal} in order to obtain
\begin{equation*}\label{errorp_h}
\beta\|\textsf{e}_h\|_{\L^2(\O)} \leq \mathfrak{A}(\bu,T,\textsf{p},\bb{f},g,\kappa,\nu)h.
\end{equation*}
Finally, using triangle inequality we conclude the proof.
\end{proof}

\section{Numerical experiments}
\label{sec:numericos}

In this section, we report some numerical experiments to evaluate the performance of the proposed VEM. Our goal is to compute the experimental convergence rates in the norms used in the theoretical analysis. The results of this section were obtained using a MATLAB code with $k=0$, where the nonlinear problem \eqref{eq:DHW2discreto} was solved using the fixed-point iteration described below in \textbf{Algorithm 1}. The refinement parameter $N$ used to characterize each mesh is the number of elements on each edge of $\Omega$. In Figure \ref{FIG:meshes} we present the meshes considered for this test, which are: \\
\begin{itemize}
\item $\CT_h$: Triangles,
\item $\mathcal{Q}_h$: Uniform squares,
\item $\mathcal{NC}_h$: Non-convex polygons,
\item $\mathcal{DQ}_h$: Deformed squares,
\item $\mathcal{H}_h$: Deformed hexagons.
\end{itemize}

\begin{figure}[h]
	\begin{center}
		\begin{minipage}{13cm}
                          \centering\includegraphics[height=4cm, width=4cm]{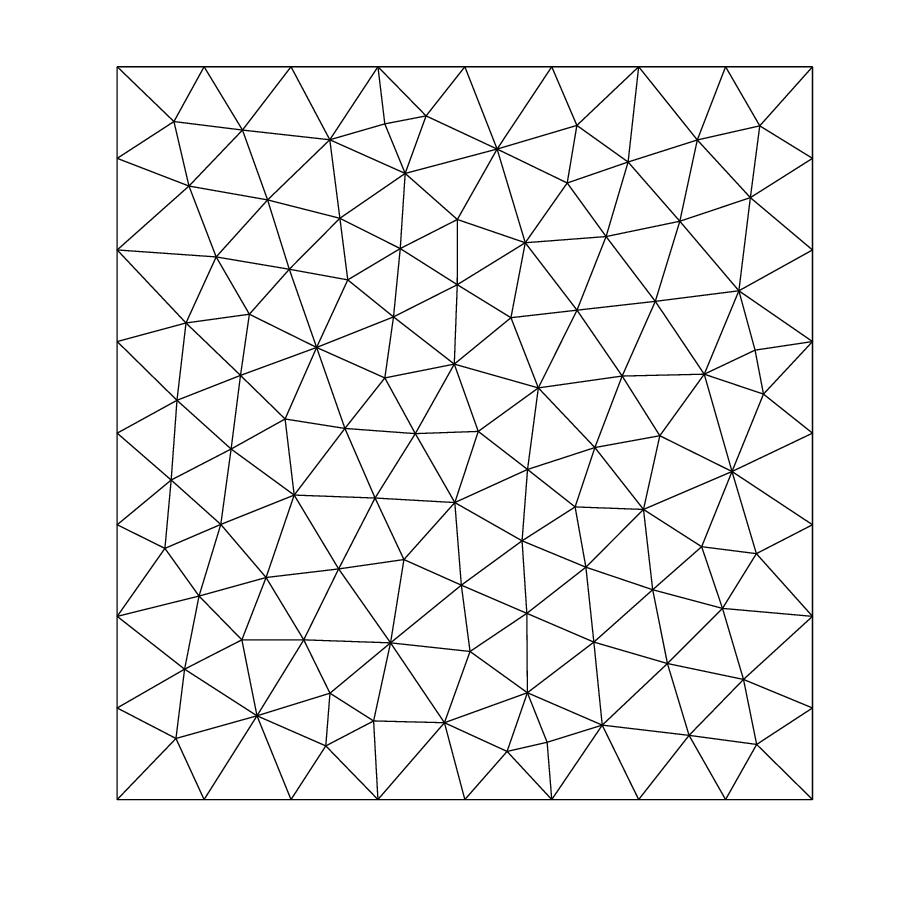}
                         \centering\includegraphics[height=4cm, width=4cm]{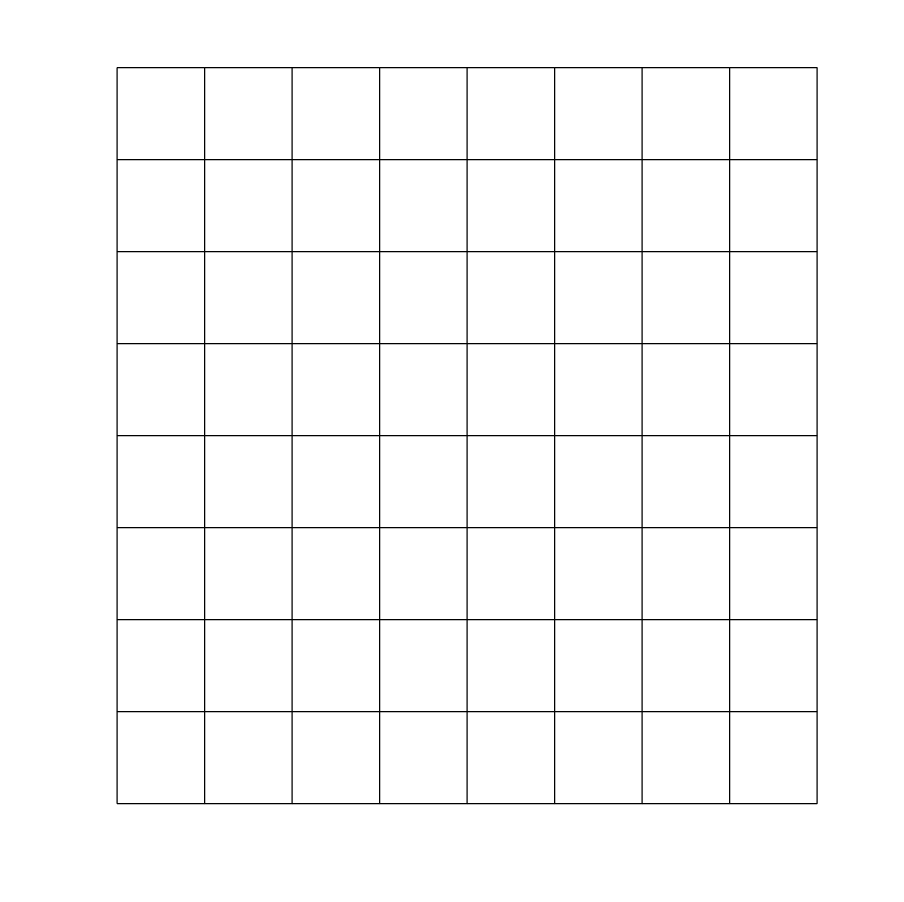}
                          \centering\includegraphics[height=4cm, width=4cm]{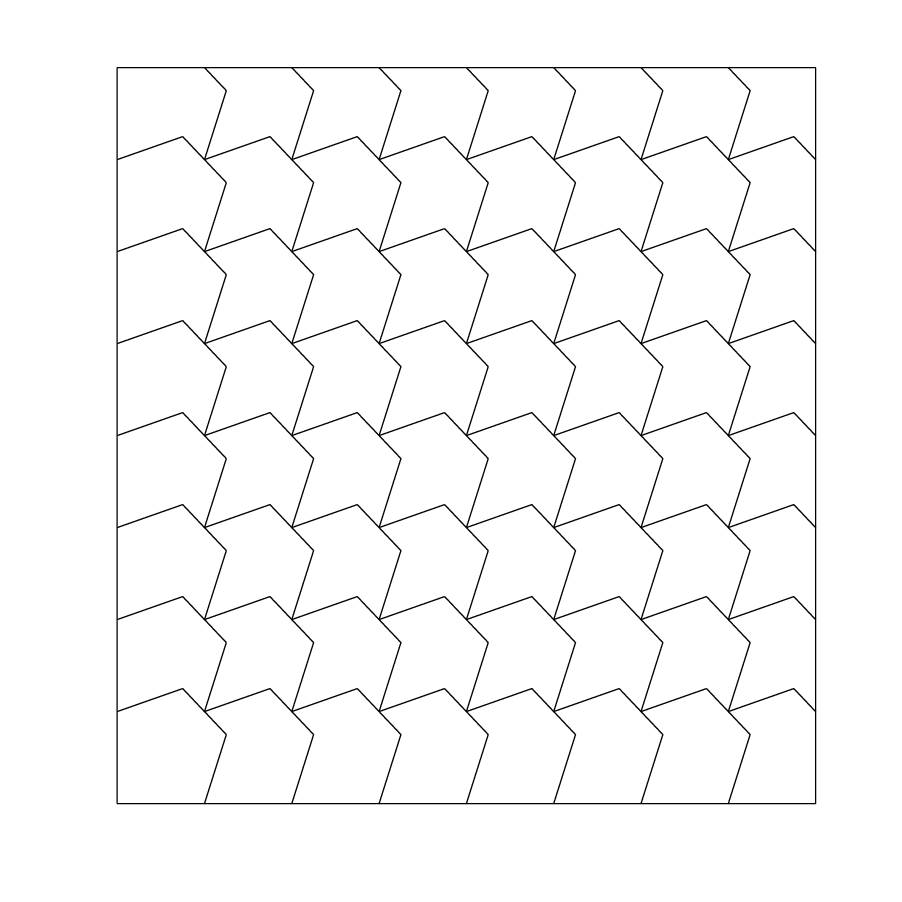}\\
                          \centering\includegraphics[height=4cm, width=4cm]{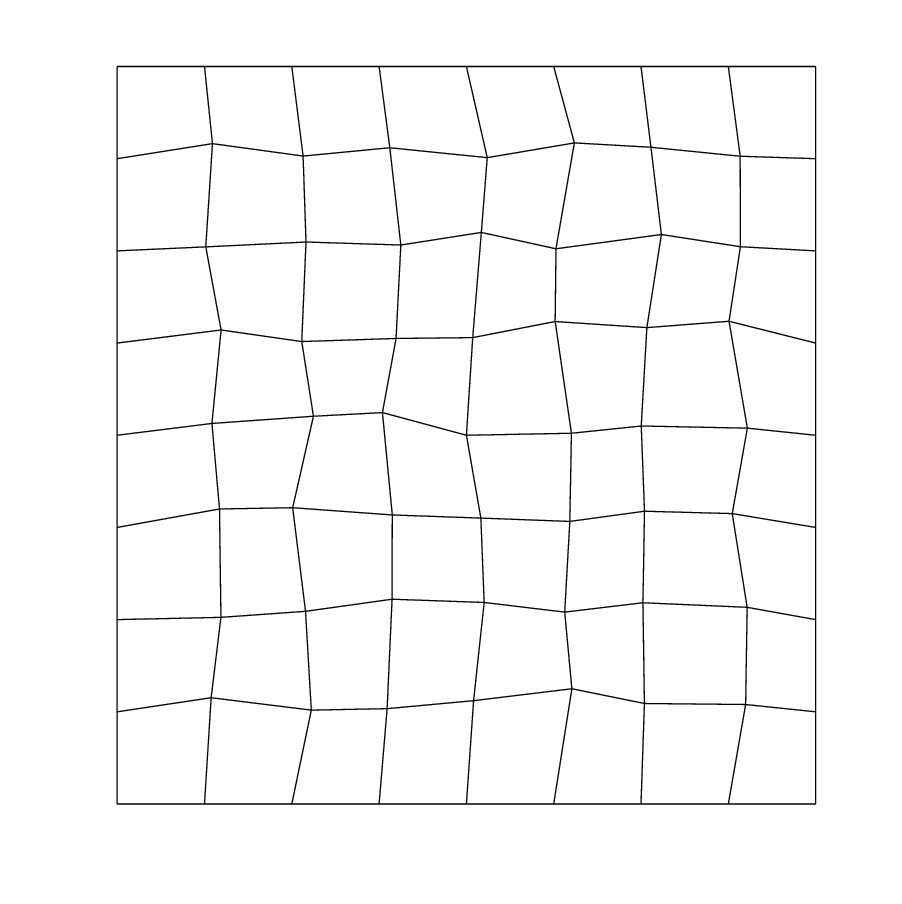}
                          \centering\includegraphics[height=4cm, width=4cm]{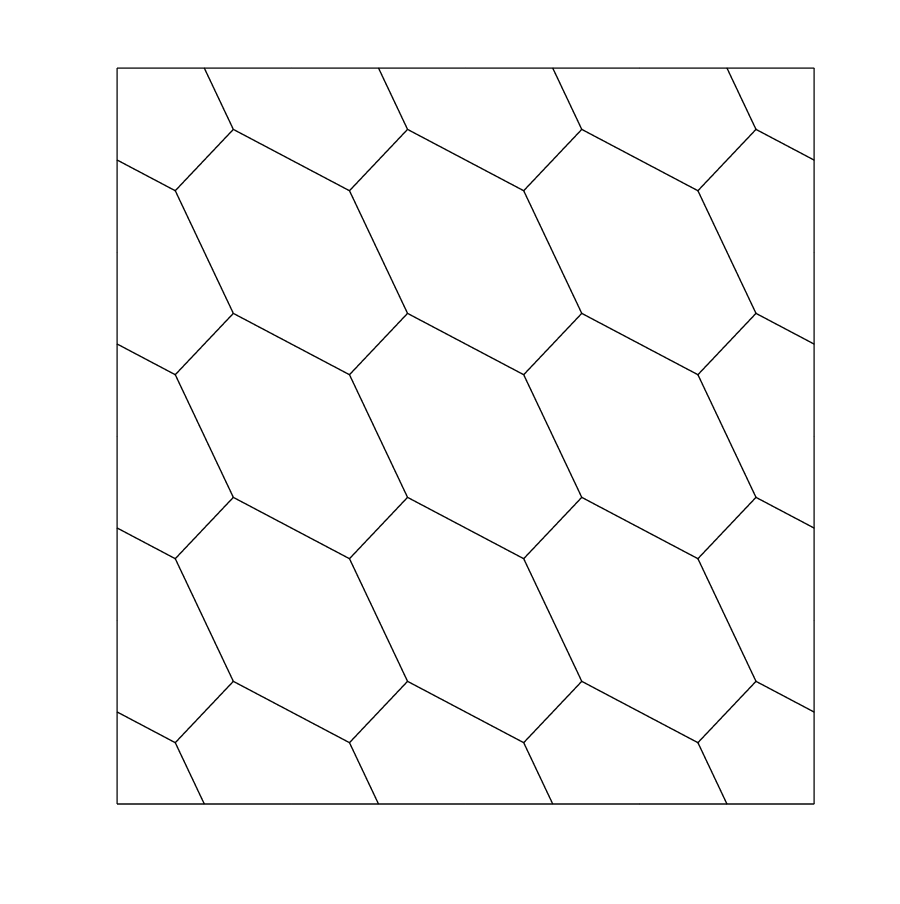}
                   \end{minipage}
		\caption{Initial meshes. From top left to bottom right: $\CT_h$, $\mathcal{Q}_h$, $\mathcal{NC}_h$, $\mathcal{DQ}_h$ and $\mathcal{H}_h$ with $N=8$.
}
		\label{FIG:meshes}
	\end{center}
\end{figure}

Let us describe the fixed-point iteration used to solve the coupled problem \eqref{eq:DHW2discreto}. \\

\noindent \textbf{Algorithm 1: Fixed-point iteration}

\noindent\rule{13cm}{0.3pt} \\
\noindent Input: Initial mesh $\O_h$, initial guess $(\bu_h^0,\textsf{p}_h^0,T_h^0)\in \bV_h \times\Q_h\times\bVV_h$, the coefficients $\nu(\cdot)$ and $\kappa(\cdot)$, $\boldsymbol{f}_h \in \L^2(\O)^2$, $g_h\in\L^{2}(\O)$, and $\textsf{tol}=10^{-6}$.
\begin{enumerate}
\item[1:] For $n\geq0$, find $(\bu_h^{n+1},\textsf{p}_{h}^{n+1})\in\bV_h\times\Q_h$ such that
\begin{equation*}
\begin{split}
a_h(T_h^n;\bu_h^{n+1},\bv_h) + c_{N,h}^{\textit{skew}}(\bu_h^n;\bu_h^{n+1},\bv_h)
+
c_{F,h}(\bu_h^n;\bu_h^{n+1},\bv_h) \\
+
d_h(\bu_h^{n+1},\bv_h)
+
b(\bv_h,\textsf{p}_h^{n+1})&= (\boldsymbol{f}_h,\bv_h)_{0,\O},
\\
b(\bu_h^{n+1},\textsf{q}_h) &= 0,
\end{split}
\end{equation*}
for every $(\bv_h,\textsf{q}_h)\in\bV_h\times\Q_h$. Then, $T_h^{n+1} \in \bVV_h$ is found as the solution of
\begin{equation*}
\mathfrak{a}_h(T_h^{n};T_h^{n+1},S_h)+\mathfrak{c}_h^{\textit{skew}}(\bu_h^{n+1};T_h^{n+1},S_h)=(g_h,S_h)_{0,\O},
\end{equation*}
for every $S_h\in\bVV_h$.
\item[2:] If $|(\bu_h^{n+1},\textsf{p}_h^{n+1},T_h^{n+1})-(\bu_h^{n},\textsf{p}_h^n,T_h^n)|>\textsf{tol}$, set $n\leftarrow n+1$ and go to step 1. Otherwise, return $(\bu_h,\textsf{p}_h,T_h)=(\bu_h^{n+1},\textsf{p}_h^{n+1},T_h^{n+1})$. Here, $|\cdot|$ denotes the Euclidean norm.
\end{enumerate}
\noindent\rule{13cm}{0.4pt}

For this algorithm, we need to describe the bilinear forms $S_{1}^{E}(\cdot,\cdot)$ and $S_{2}^{E}(\cdot,\cdot)$ that satisfy \eqref{eq:S_v} and \eqref{eq:S_T}, respectively. Proceeding as in \cite{MR2997471}, for the temperature we consider the so-called  dofi--dofi stabilization, which is defined as follows: for each $E\in\O_h$, we denote by $\vec{T}_h$ and $\vec{R}_h$ the real-valued vectors containg the values of the local degrees of freedom associated with $T_h$ and $R_h \in \H_h$, respectively. Then, we set
\[
S_{2}^{E}(T_h,S_h):=\vec{T}_h\cdot \vec{S_h},
\]
whereas for the velocity, defining $N_\textsf{edges}$ as the number of edges of the polygon $E\in\O_h$, we consider the following stabilization form (\cite{BMRR}):
\[
S_{1}^{E}(\bu_h,\bv_h):=|E|\sum_{k=1}^{N_\textsf{edges}} \left(\int_{e_k}\bu_h\cdot\bn\right)\left(\int_{e_k}\bv_h\cdot\bn\right).
\]

To compute the errors between the exact and approximated solutions we define, for the velocity and temperature, the following computable quantities:
\[
\textsf{e}({\bu}):=\displaystyle\left(\displaystyle\sum_{E\in\O_h}\|\bu-\bPi_0^{0,E}\bu_h\|_{0,E}^2\right)^{\frac{1}{2}},
\quad
\textsf{e}(T):=\displaystyle\left(\displaystyle\sum_{E\in\O_h}|T-\Pi_1^{0,E}T_h|_{1,E}^2\right)^{\frac{1}{2}},
\]
respectively, whereas for the pressure we consider $\textsf{e}(\textsf{p}):=\|\textsf{p}-\textsf{p}_h\|_{0,\O}$.


\subsection{Square Domain} We begin by testing the method on the square $\O:=(0,3)^2$. The viscosity coefficient $\nu(\cdot)$ and the thermal diffusivity coefficient $\kappa(\cdot)$ are given as follows:
\begin{enumerate}
\item Test 1: $\nu(T)=T+1$ and $\kappa(T)=3$.
\item Test 2: $\nu(T)=T+1$ and $\kappa(T)=1+\exp(-T)$.
\item Test 3: $\nu(T)=0.1+\exp(T/100)$ and $\kappa(T)=2+\sin(T)$.
\end{enumerate}
As manufactured solutions to problem \eqref{eq:DH},	 we choose
\begin{align*}
  \bu(x_1,x_2)&=\curl(\psi), \qquad \psi(x_1,x_2)=\exp(-5(x-1)^2-5(y-1)^2)\\
  \textsf{p}(x_1,x_2)&=\cos(\pi x_1/3)\cos(\pi x_2/3), \\
  T(x_1,x_2)&=x_1^2x_2^2(x_1-3)^2(x_2-3)^2.
\end{align*}

On the following tables, \textsf{it} and \textsf{DoF} denote the number of iterations of the fixed-point iteration and total degrees of freedom associated to the spaces $\bV_h, \Q_h$ and $\H_h$, respectively. We begin by reporting the obtained results for the configuration of Test 1. 
\begin{table}[h]
 \caption{Test 1: Computed errors and convergence rates for different polygonal meshes.}
 \label{tabla1}
 \begin{center}
 \resizebox{0.75\textwidth}{!}{
 \begin{tabular}{ccccccccccccc}
 & \textsf{DoF} & $N$ & $\textsf{it}$ & $\textsf{e}(\bu)$ & $r(\bu)$& $\textsf{e}(T)$ & $r(T)$ & $\textsf{e}(\textsf{p})$ & $r(\textsf{p})$ \\ 
\hline \hline
\rule{0pt}{3ex}
\multirow{5}{*}{$\CT_h$} & 669  & 8  & 1 & 0.37859 & - & 10.16365 & - & 0.31143 & - \\
& 2675 & 16 & 1 & 0.18095 & 1.07 & 5.12609 & 0.99 & 0.15877 & 0.97 \\
& 10587 & 32 & 1 & 0.09257 & 0.97 & 2.63609 & 0.96 & 0.07971 & 0.99 \\
& 42497 & 64 & 1 & 0.04608 & 1.01 & 1.32794 & 0.99 & 0.03944 & 1.01 \\
& 171309 & 128 & 1 & 0.02296 & 1.00 & 0.66411 & 1.00 & 0.01955 & 1.01 \\
\hline \hline
\rule{0pt}{3ex}
\multirow{5}{*}{$\mathcal{Q}_h$} & 289 & 8 & 1 & 0.73951 & - & 17.99615 & - & 0.54818 & - \\
& 1089 & 16 & 1 & 0.33279 & 1.15 & 9.40367 & 0.94 & 0.27341 & 1.00 \\
& 4225 & 32 & 1 & 0.16616 & 1.00 & 4.76212 & 0.98 & 0.13945 & 0.97 \\
& 16641 & 64 & 1 & 0.08308 & 1.00 & 2.38889 & 1.00 & 0.07016 & 0.99 \\
& 66049 & 128 & 1 & 0.04154 & 1.00 & 1.19544 & 1.00 & 0.03514 & 1.00 \\
\hline \hline
\rule{0pt}{3ex}
\multirow{5}{*}{$\mathcal{NC}_h$} & 513 & 8 & 1 & 1.03981 & - & 18.80783 & - & 0.57982 & - \\
& 2049 & 16 & 1 & 0.35096 & 1.57 & 9.51943 & 0.98 & 0.27503 & 1.08 \\
& 8193 & 32 & 1 & 0.16665 & 1.07 & 4.77680 & 0.99 & 0.13974 & 0.98 \\
& 32769 & 64 & 1 & 0.08309 & 1.00 & 2.39072 & 1.00 & 0.07023 & 0.99 \\
& 131073 & 128 & 1 & 0.04154 & 1.00 & 1.19566 & 1.00 & 0.03515 & 1.00 \\
\hline \hline
\rule{0pt}{3ex}
\multirow{5}{*}{$\mathcal{DQ}_h$} & 289 & 8 & 1 & 0.73540 & - & 18.23153 & - & 0.56442 & - \\
& 1089 & 16 & 1 & 0.33402 & 1.14 & 9.49409 & 0.94 & 0.27473 & 1.04 \\
& 4225 & 32 & 1 & 0.16735 & 1.00 & 4.79847 & 0.98 & 0.14037 & 0.97 \\
& 16641 & 64 & 1 & 0.08355 & 1.00 & 2.40387 & 1.00 & 0.07056 & 0.99 \\
& 66049 & 128 & 1 & 0.04176 & 1.00 & 1.20200 & 1.00 & 0.03533 & 1.00 \\
\hline \hline
\rule{0pt}{3ex}
\multirow{5}{*}{$\mathcal{H}_h$} & 545 & 8 & 1 & 0.71700 & - & 18.45831 & - & 0.51780 & - \\
& 1857 & 16 & 1 & 0.33488 & 1.10 & 9.48519 & 0.96 & 0.26241 & 0.98 \\
& 6785 & 32 & 1 & 0.16621 & 1.01 & 4.77338 & 0.99 & 0.13580 & 0.95 \\
& 25857 & 64 & 1 & 0.08308 & 1.00 & 2.39036 & 1.00 & 0.06914 & 0.97 \\
& 100865 & 128 & 1 & 0.04154 & 1.00 & 1.19562 & 1.00 & 0.03487 & 0.99 \\
\hline \hline
 \end{tabular}}
 \end{center}
 \end{table}

From Table \ref{tabla1}, we observe that the results are consistent with Theorem \ref{thm:error_estimates} and also with the numerical test in \cite[Section 5]{BernardiDH}. We remark that these results are independent of the mesh under consideration on the experiment, proving that for any polygonal mesh, the method is capable of capturing the physical properties of the model. We remark that only one interation was needed on Algorithm 1 to achieve the results.
For Test 2, the results are reported in Table \ref{tabla2}. From this table we observe that once again, the proposed method is convergent, independent of the polygonal mesh, and the rates of convergence for each of the variables is the expected according to the theoretical results.
Finally, in Table \ref{tabla3} we present the computed errors and convergence rates obtained for Test 3 for the polygonal meshes presented in Figure \ref{FIG:meshes}. Here, the convergence rates are optimal, according to Theorem \ref{thm:error_estimates}. Finally, in Figure \ref{fig:plots} we present the exact and approximated solutions for this test, when the mesh is $\mathcal{Q}_h$. For the other meshes, the plots show no differences.


 \begin{table}[h]
 \caption{Test 2: Computed errors and convergence rates for different polygonal meshes.}
 \label{tabla2}
 \begin{center}
 \resizebox{0.75\textwidth}{!}{
 \begin{tabular}{ccccccccccccc}
 & \textsf{DoF} & $N$ & $\textsf{it}$ & $\textsf{e}(\bu)$ & $r(\bu)$& $\textsf{e}(T)$ & $r(T)$ & $\textsf{e}(\textsf{p})$ & $r(\textsf{p})$ \\ 
\hline \hline
\rule{0pt}{3ex}
\multirow{5}{*}{$\CT_h$} & 669  & 8  & 1 & 0.37859 & - & 10.35088 & - & 0.31143 & - \\
& 2675 & 16  & 1 & 0.18095 & 1.07 & 5.14972 & 1.01 & 0.15877 & 0.97  \\
& 10587 & 32  & 1 & 0.09257 & 0.97 & 2.63891 & 0.96 & 0.07971 & 0.99 \\
& 42497 & 64 & 1 & 0.04608 & 1.01 & 1.32813 & 0.99 & 0.03944 & 1.01 \\
& 171309 & 128 & 1 & 0.02296 & 1.00 & 0.66412 & 1.00 & 0.01955 & 1.01 \\
\hline \hline
\rule{0pt}{3ex}
\multirow{5}{*}{$\mathcal{Q}_h$} & 289 & 8 & 1 & 0.73951 & - & 18.48968 & - & 0.54818 & - \\
& 1089 & 16 & 1 & 0.33279 & 1.15 & 9.47961 & 0.96 & 0.27341 & 1.00 \\
& 4225 & 32 & 1 & 0.16616 & 1.00 & 4.77209 & 0.99 & 0.13945 & 0.97 \\
& 16641 & 64 & 1 & 0.08308 & 1.00 & 2.39016 & 1.00 & 0.07016 & 0.99 \\
& 66049 & 128 & 1 & 0.04154 & 1.00 & 1.19559 & 1.00 & 0.03514 & 1.00 \\
\hline \hline
\rule{0pt}{3ex}
\multirow{5}{*}{$\mathcal{NC}_h$} & 513 & 8 & 1 & 1.03981 & - & 18.97966 & - & 0.57982 & - \\
& 2049 & 16 & 1 & 0.35096 & 1.57 & 9.54372 & 0.99 & 0.27503 & 1.08 \\
& 8193 & 32 & 1 & 0.16665 & 1.07 & 4.77996 & 1.00 & 0.13974 & 0.98 \\
& 32769 & 64 & 1 & 0.08309 & 1.00 & 2.39112 & 1.00 & 0.07023 & 0.99 \\
& 131073 & 128 & 1 & 0.04154 & 1.00 & 1.19571 & 1.00 & 0.03515 & 1.00 \\
\hline \hline
\rule{0pt}{3ex}
\multirow{5}{*}{$\mathcal{DQ}_h$} & 289 & 8 & 1 & 0.73497 & - & 18.78999 & - & 0.54078 & - \\
& 1089 & 16 & 1 & 0.33527 & 1.13 & 9.56666 & 0.97 & 0.27554 & 0.97 \\
& 4225 & 32 & 1 & 0.16714 & 1.00 & 4.80726 & 0.99 & 0.14034 & 0.97 \\
& 16641 & 64 & 1 & 0.08351 & 1.00 & 2.40351 & 1.00 & 0.07056 & 0.99 \\
& 66049 & 128 & 1 & 0.04177 & 1.00 & 1.20200 & 1.00 & 0.03533 & 1.00 \\
\hline \hline
\rule{0pt}{3ex}
\multirow{5}{*}{$\mathcal{H}_h$} & 545 & 8 & 1 & 0.71700 & - & 18.71139 & - & 0.51780 & - \\
& 1857 & 16 & 1 & 0.33488 & 1.10 & 9.51920 & 0.98 & 0.26241 & 0.98 \\
& 6785 & 32 & 1 & 0.16621 & 1.01 & 4.77772 & 0.99 & 0.13580 & 0.95 \\
& 25857	& 64 & 1 & 0.08308 & 1.00 & 2.39090 & 1.00 & 0.06914 & 0.97 \\
& 100865 & 128 & 1 & 0.04154 & 1.00 & 1.19569 & 1.00 & 0.03487 & 0.99 \\
\hline \hline
 \end{tabular}}
 \end{center}
 \end{table}


 \begin{table}[h]
 \caption{Test 3: Computed errors and convergence rates for different polygonal meshes.}
 \label{tabla3}
 \begin{center}
 \resizebox{0.75\textwidth}{!}{
 \begin{tabular}{ccccccccccccc}
 & \textsf{DoF} & $N$ & $\textsf{it}$ & $\textsf{e}(\bu)$ & $r(\bu)$& $\textsf{e}(T)$ & $r(T)$ & $\textsf{e}(\textsf{p})$ & $r(\textsf{p})$ \\ 
\hline \hline
\rule{0pt}{3ex}
\multirow{5}{*}{$\CT_h$} & 669  & 8  & 1 & 0.37010 & - & 10.36679 & - & 0.30790 & - \\
& 2675 & 16 & 1 & 0.18027 & 1.04 & 5.12431 & 1.02 & 0.15880 & 0.96 \\
& 10587 & 32 & 1 & 0.09253 & 0.96 & 2.63582 & 0.96 & 0.07972 & 0.99 \\
& 42497 & 64 & 1 & 0.04608 & 1.01 & 1.32794 & 0.99 & 0.03944 & 1.02 \\
& 171309 & 128 & 1 & 0.02296 & 1.00 & 0.66411 & 1.00 & 0.01955 & 1.01 \\
\hline \hline
\rule{0pt}{3ex}
\multirow{5}{*}{$\mathcal{Q}_h$} & 289 & 8 & 1 & 0.66501 & - & 19.11512 & - & 0.48543 & - \\
& 1089 & 16 & 1 & 0.33232 & 1.00 & 9.39918 & 1.02 & 0.27143 & 0.84 \\
& 4225 & 32 & 1 & 0.16616 & 1.00 & 4.76180 & 0.98 & 0.13939 & 0.96 \\
& 16641 & 64 & 1 & 0.08308 & 1.00 & 2.38885 & 1.00 & 0.07016 & 0.99 \\
& 66049 & 128 & 1 & 0.04154 & 1.00 & 1.19543 & 1.00 & 0.03514 & 1.00 \\
\hline \hline
\rule{0pt}{3ex}
\multirow{5}{*}{$\mathcal{NC}_h$} & 513 & 8 & 1 & 0.67016 & - & 18.94194 & - & 0.49146 & - \\
& 2049 & 16 & 1 & 0.33253 & 1.01 & 9.51700 & 0.99 & 0.27261 & 0.85 \\
& 8193 & 32 & 1 & 0.16617 & 1.00 & 4.77668 & 0.99 & 0.13967 & 0.96 \\
& 32769 & 64 & 1 & 0.08308 & 1.00 & 2.39071 & 1.00 & 0.07023 & 0.99 \\
& 131073 & 128 & 1 & 0.04154 & 1.00 & 1.19566 & 1.00 & 0.03515 & 1.00 \\
\hline \hline
\rule{0pt}{3ex}
\multirow{5}{*}{$\mathcal{DQ}_h$} & 289 & 8 & 1 & 0.66905 & - & 18.67337 & - & 0.48969 & - \\
& 1089 & 16 & 1 & 0.33325 & 1.01 & 9.49251 & 0.98 & 0.27450 & 0.84 \\
& 4225 & 32 & 1 & 0.16723 & 0.99 & 4.79820 & 0.98 & 0.14040 & 0.97 \\
& 16641 & 64 & 1 & 0.08350 & 1.00 & 2.40319 & 1.00 & 0.07055 & 0.99 \\
& 66049 & 128 & 1 & 0.04175 & 1.00 & 1.20195 & 1.00 & 0.03533 & 1.00 \\
\hline \hline
\rule{0pt}{3ex}
\multirow{5}{*}{$\mathcal{H}_h$} & 545 & 8 & 1 & 0.66485 & - & 18.83858 & - & 0.44172 & - \\
& 1857 & 16 & 1 & 0.33233 & 1.00 & 9.48427 & 0.99 & 0.25626 & 0.79 \\
& 6785 & 32 & 1 & 0.16616 & 1.00 & 4.77327 & 0.99 & 0.13523 & 0.92 \\
& 25857 & 64 & 1 & 0.08308 & 1.00 & 2.39034 & 1.00 & 0.06908 & 0.97 \\
& 100865 & 128 & 1 & 0.04154 & 1.00 & 1.19562 & 1.00 & 0.03486 & 0.99 \\
\hline \hline 
 \end{tabular}}
 \end{center}
 \end{table}

\begin{figure}[H]
	\begin{center}
		\begin{minipage}{13cm}
                          \centering\includegraphics[height=3.7cm, width=3.7cm]{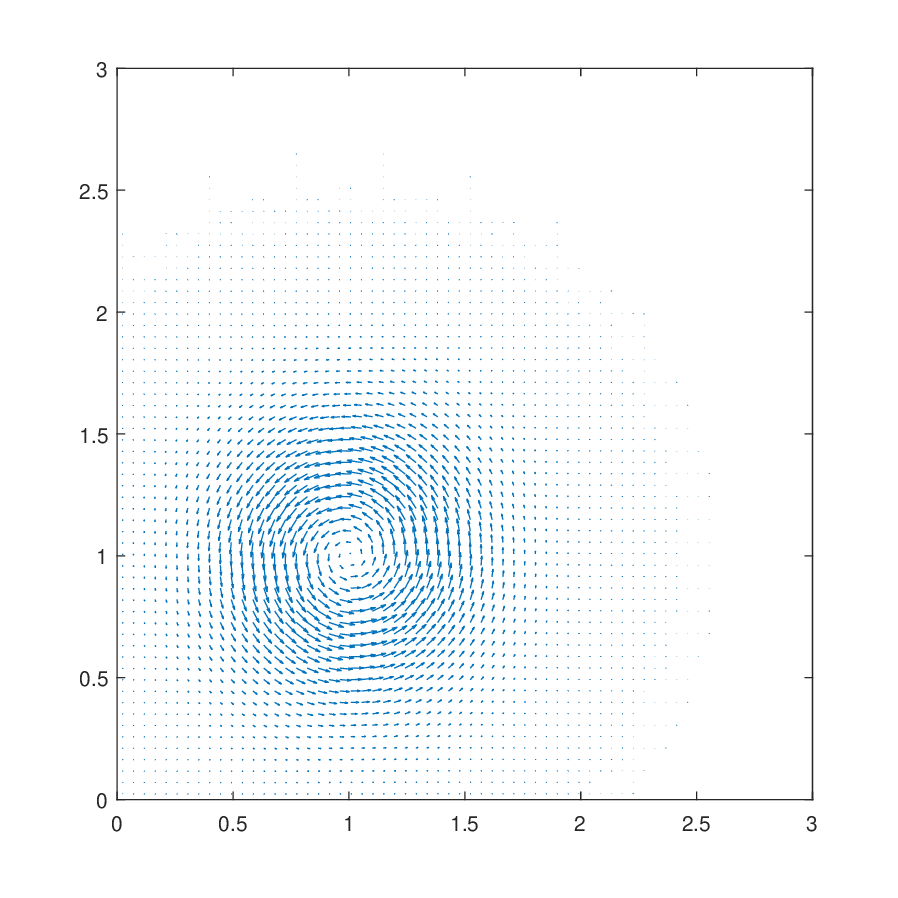}
                          \centering\includegraphics[height=3.7cm, width=3.7cm]{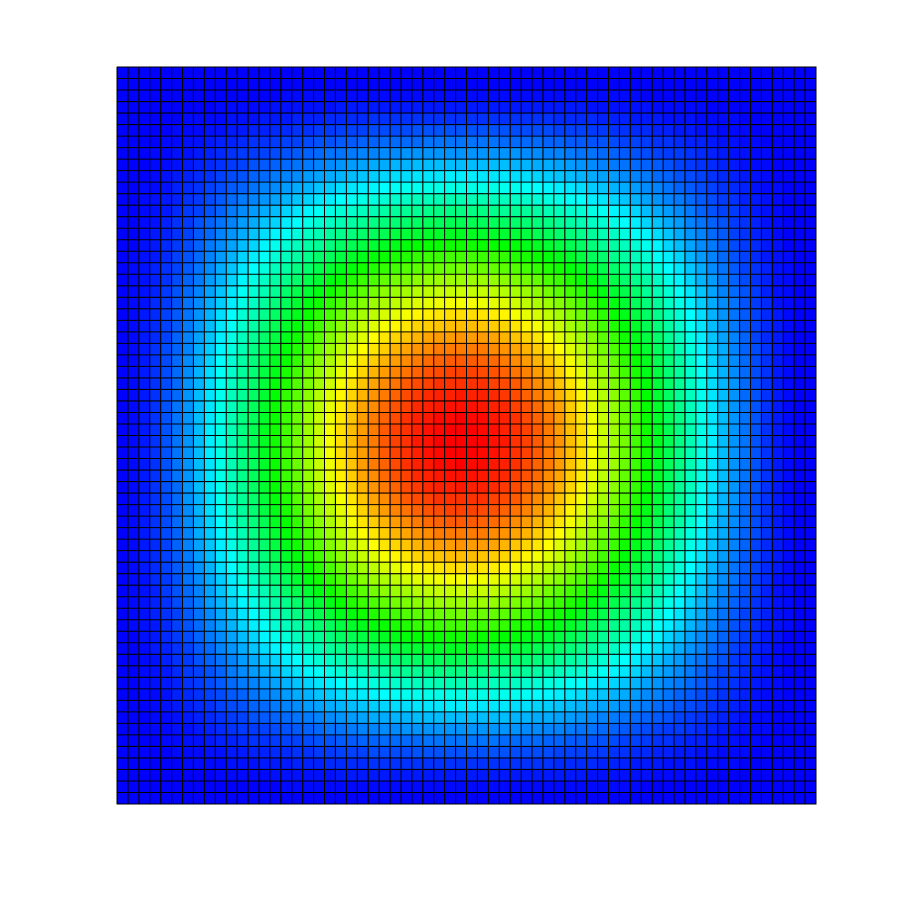}
                          \centering\includegraphics[height=3.7cm, width=3.7cm]{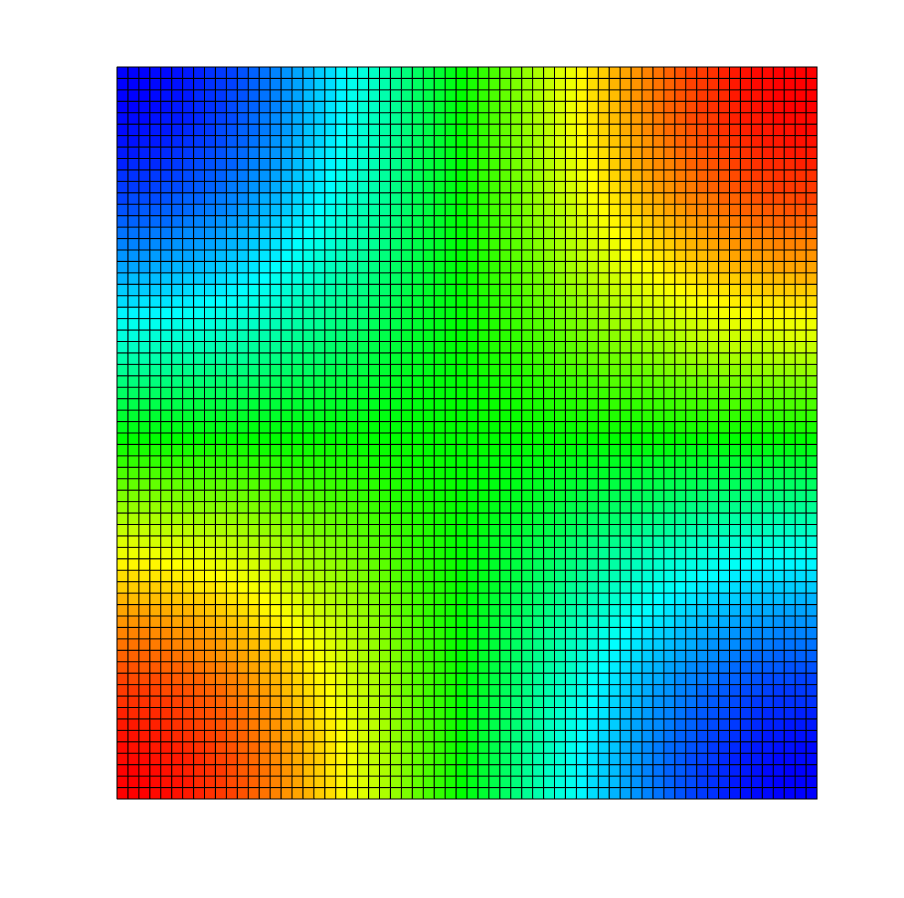}\\
                          \centering\includegraphics[height=3.7cm, width=3.7cm]{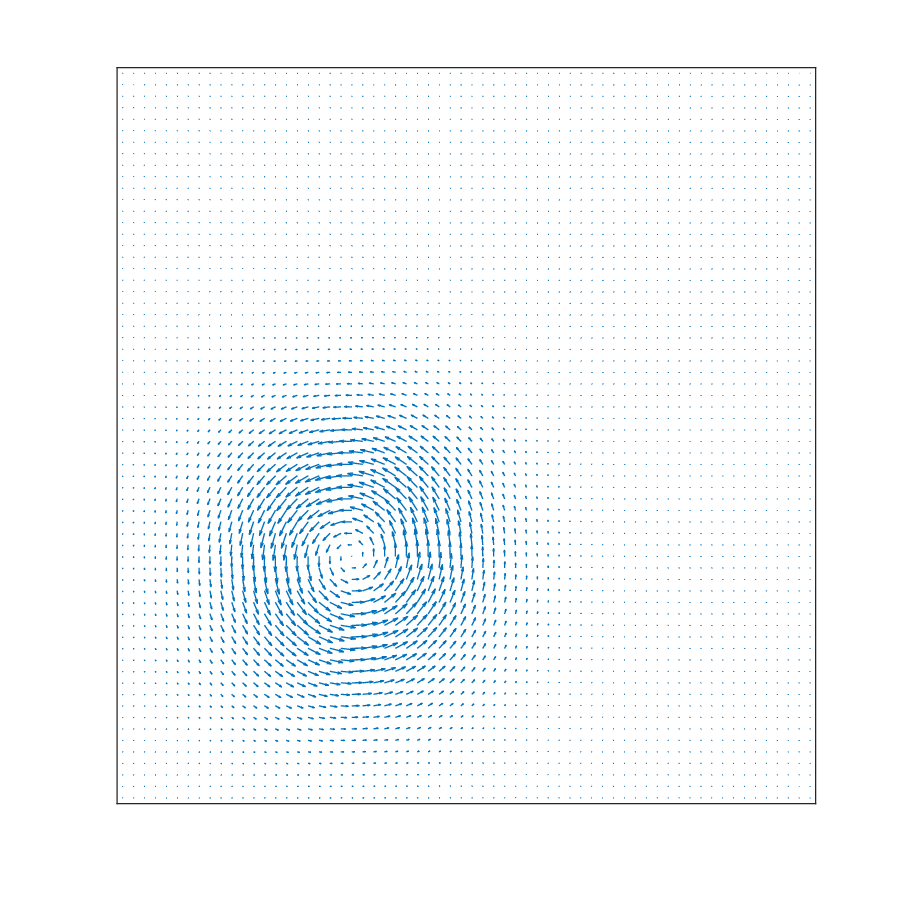}
                          \centering\includegraphics[height=3.7cm, width=3.7cm]{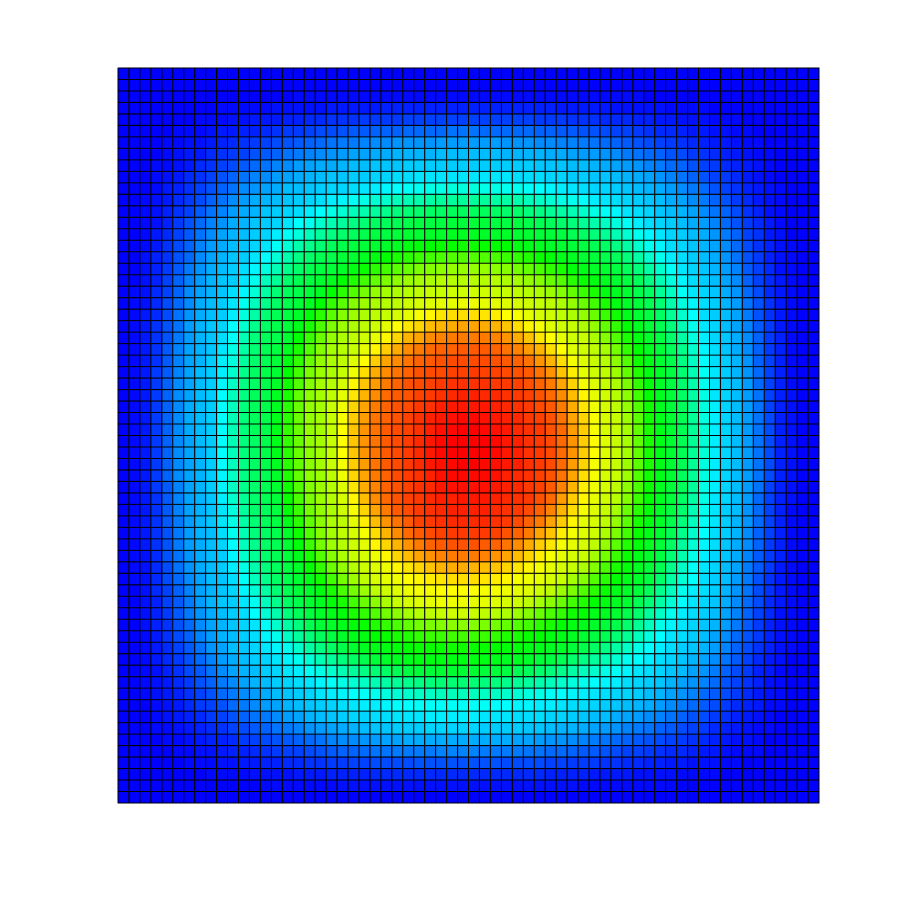}
                          \centering\includegraphics[height=3.7cm, width=3.7cm]{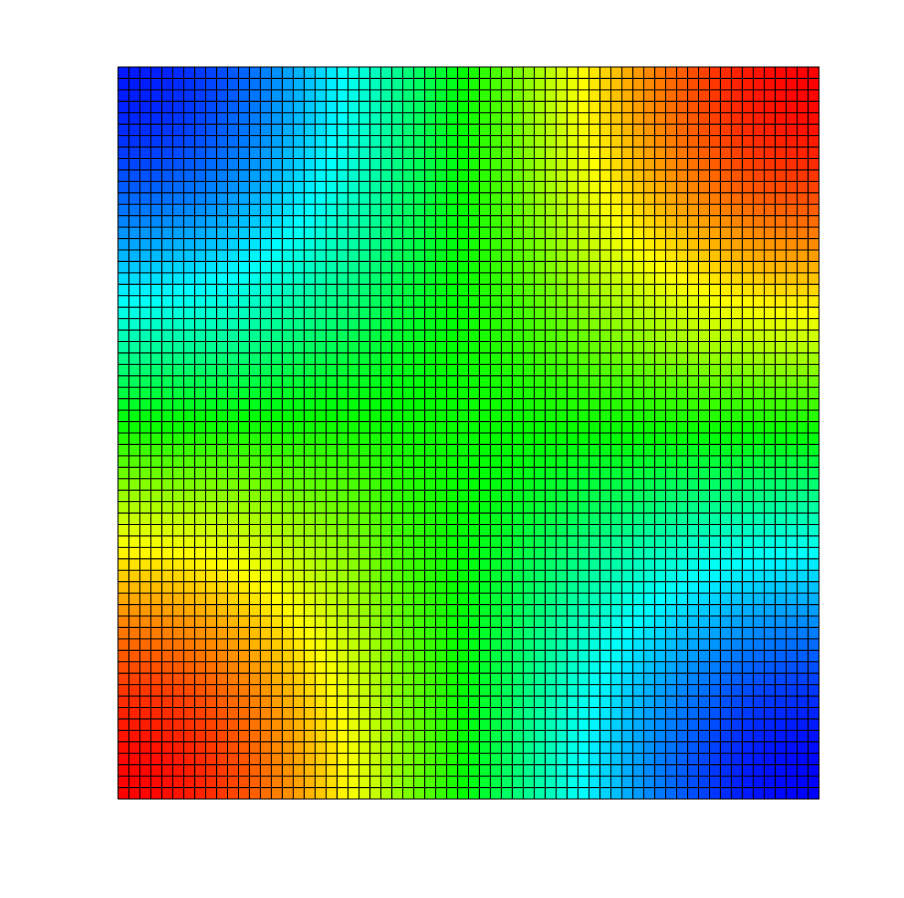}
                   \end{minipage}
		\caption{Exact and approximated solutions for Test 3. First column: $|u|$, $T$ and $\textsf{p}$; second column: $|u_h|$, $T_h$ and $\textsf{p}_h$, plotted using mesh $\mathcal{Q}_h$.}
		\label{fig:plots}
	\end{center}
\end{figure}

Our next aim is to prove computationally, that the proposed VEM scheme is divergence-free at discrete level. To carry out this task, we compute the $\L^2$ norm of the discrete divergence $\div\bu_h$ on different polygonal meshes and refinement levels. In Table \ref{tabla4} we present the quantity $\|\div\bu_h\|_{\L^2(\O)}$ for different polygonal meshes which clearly,  reveals that the method is capable to preserve the incompressibility condition at discrete level for any geometrical configuration of the mesh.
\begin{table}[h]
 \caption{$\|\div \bu_h\|_{0,\O}$ computed for different polygonal meshes.}
 \label{tabla4}
 \begin{center}
 \begin{tabular}{ccccccc}
 \hline \hline 
 \rule{0pt}{3ex}
  $N$ & $\CT_{h}$ & $\mathcal{Q}_h$ & $\mathcal{NC}_h$ & $\mathcal{DQ}_h$  &$\mathcal{H}_h$. \\ \hline \hline 
  \rule{0pt}{3ex}
  8 & 6.49982e-18 & 1.25969e-17 & 2.16493e-17 & 1.40424e-17 & 5.08689e-03 \\
 16 &1.18450e-18 &  3.17447e-18 &  1.11157e-17   & 2.96106e-18  & 6.40267e-04 \\
  32  & 3.15053e-19&   7.28291e-19 &  2.95751e-18 &  7.62836e-19&   7.8849e-05\\
  64  &6.83867e-20  & 1.69481e-19  & 7.68963e-19  & 2.22270e-19   &8.96693e-06\\
    128  & 1.66510e-20 & 4.18205e-19  & 2.08273e-19  & 5.61828e-20 &  9.88791e-07 \\\hline\hline
 \end{tabular}
 \end{center}
 \end{table}
\subsection{Influence of the viscosity}
Now our aim is to observe the performance of the VEM when the viscosity changes. More precisely, we focus our attention for sufficiently small values of this parameter and its influence on the convergence of the method. In order to obtain the most detailed information about this fact, we compute the velocity, pressure and temperature for small values of $\nu$ and different meshes. These results are reported in the following error curves.
\begin{figure}[H]
	\begin{center}
                           \centering\includegraphics[height=4.4cm, width=4.2cm]{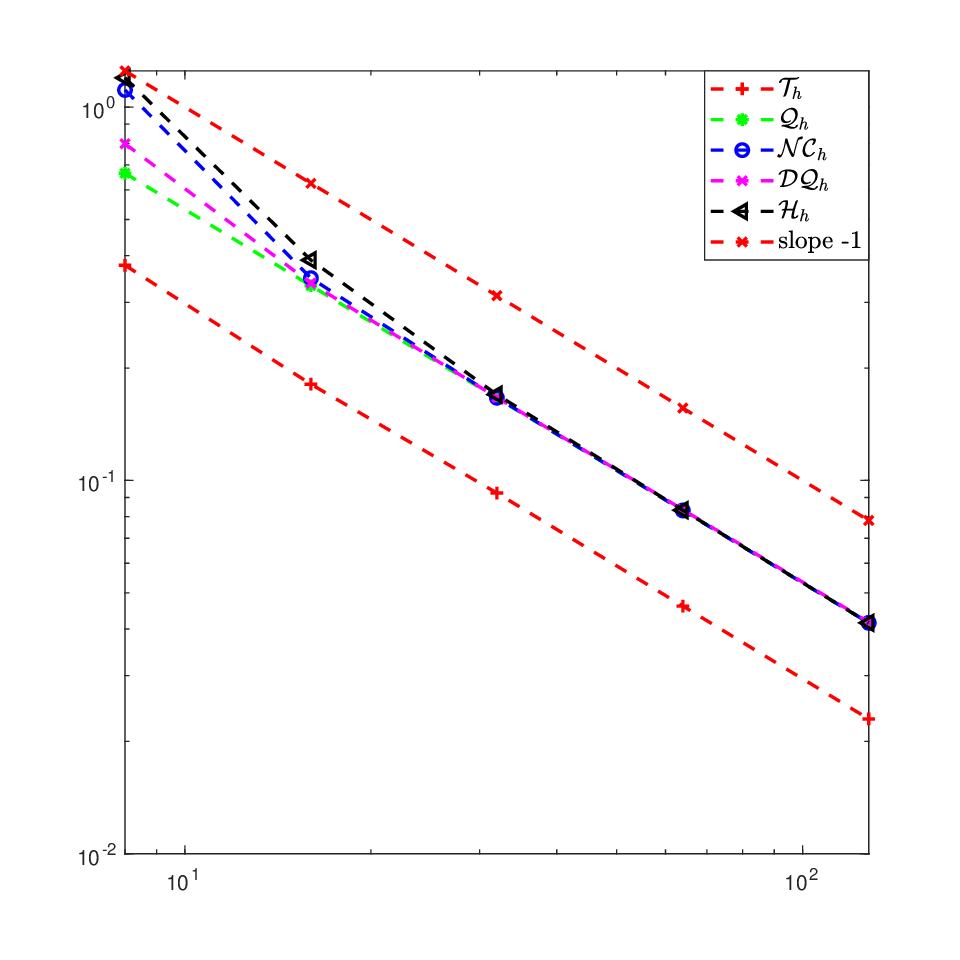}
                          \centering\includegraphics[height=4.4cm, width=4.2cm]{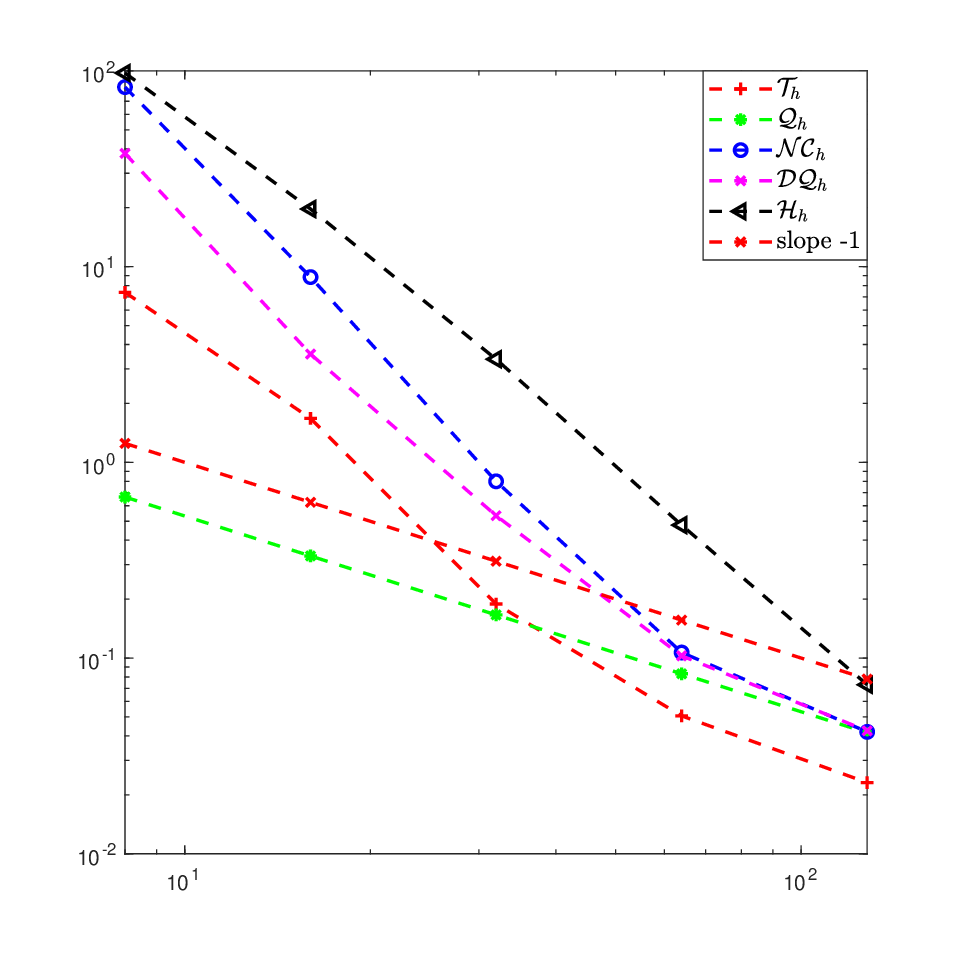}
                       \centering\includegraphics[height=4.4cm, width=4.2cm]{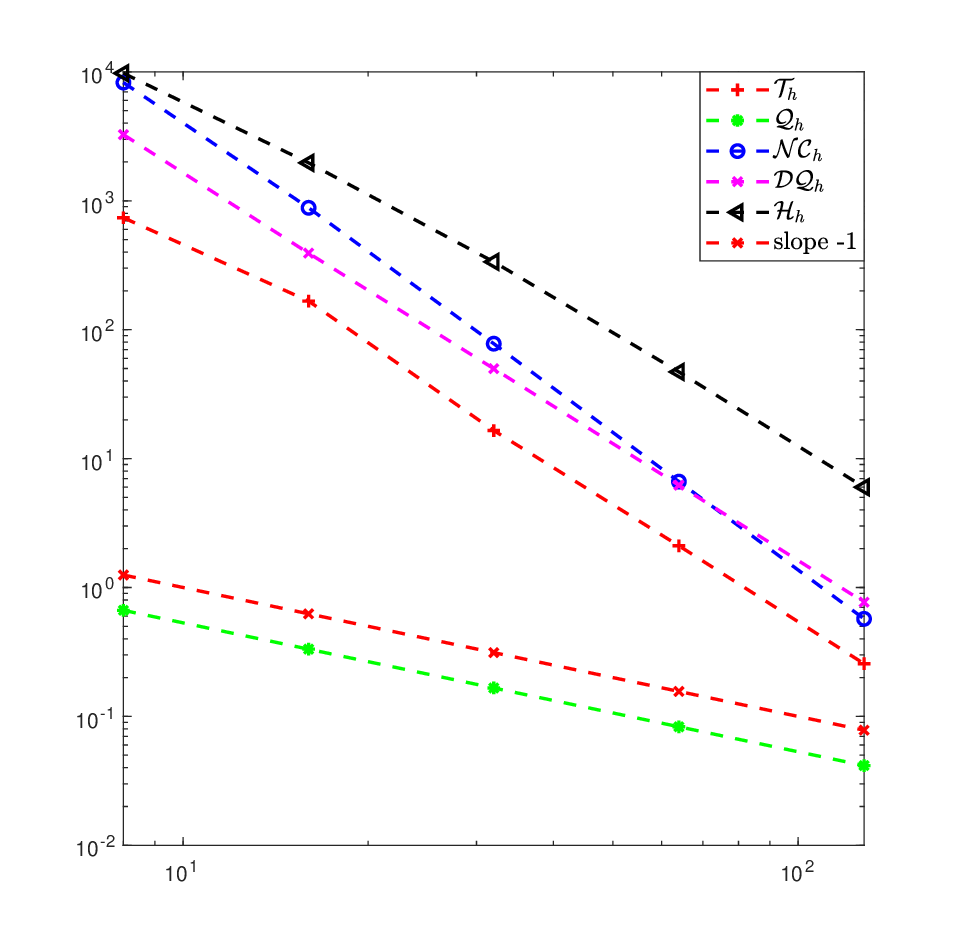}
 		\caption{Test 4: From left to right: velocity error for $\nu =10^{-2},10^{-4},10^{-6}$ for meshes $CT_h$, $\mathcal{Q}_h$, $\mathcal{NC}_h$, $\mathcal{DQ}_h$ and $\mathcal{H}_h$.}
 		\label{FIG:error_vel}
 	\end{center}
 \end{figure}

\begin{figure}[H]
	\begin{center}
                           \centering\includegraphics[height=4.4cm, width=4.2cm]{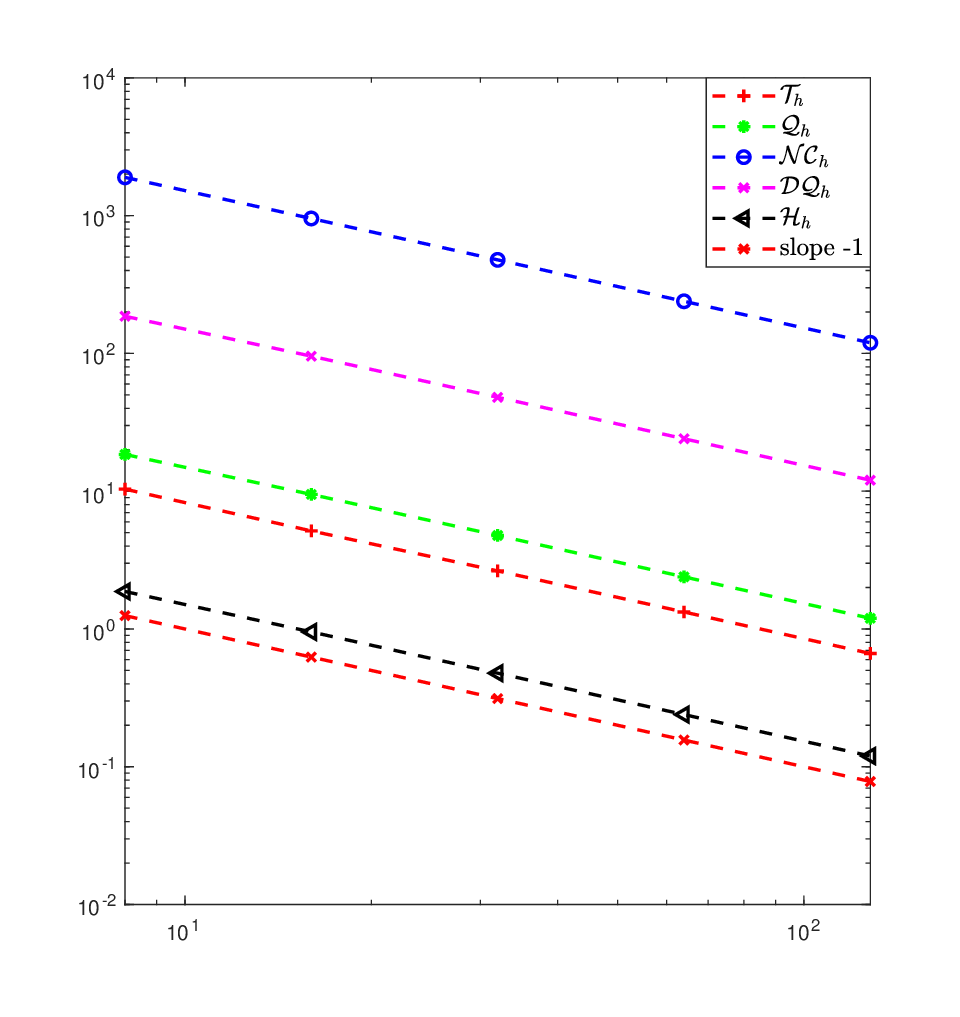}
                           \centering\includegraphics[height=4.4cm, width=4.2cm]{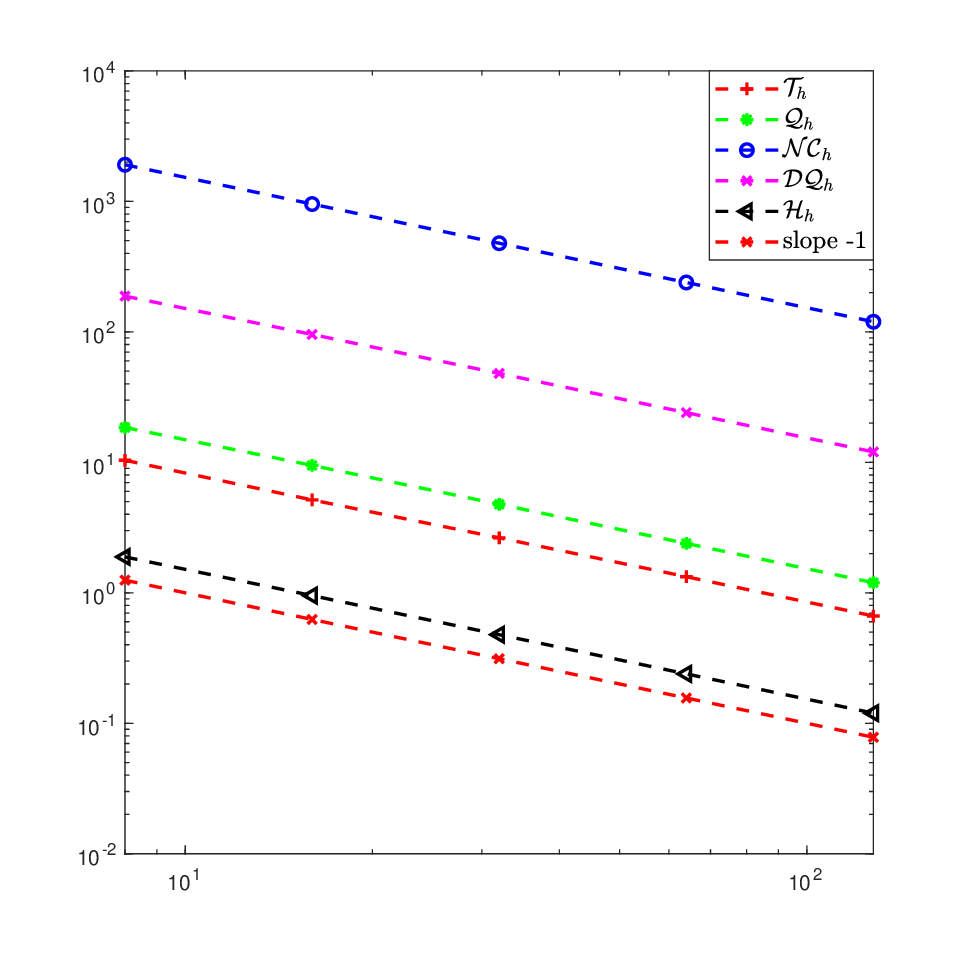}
                           \centering\includegraphics[height=4.4cm, width=4.2cm]{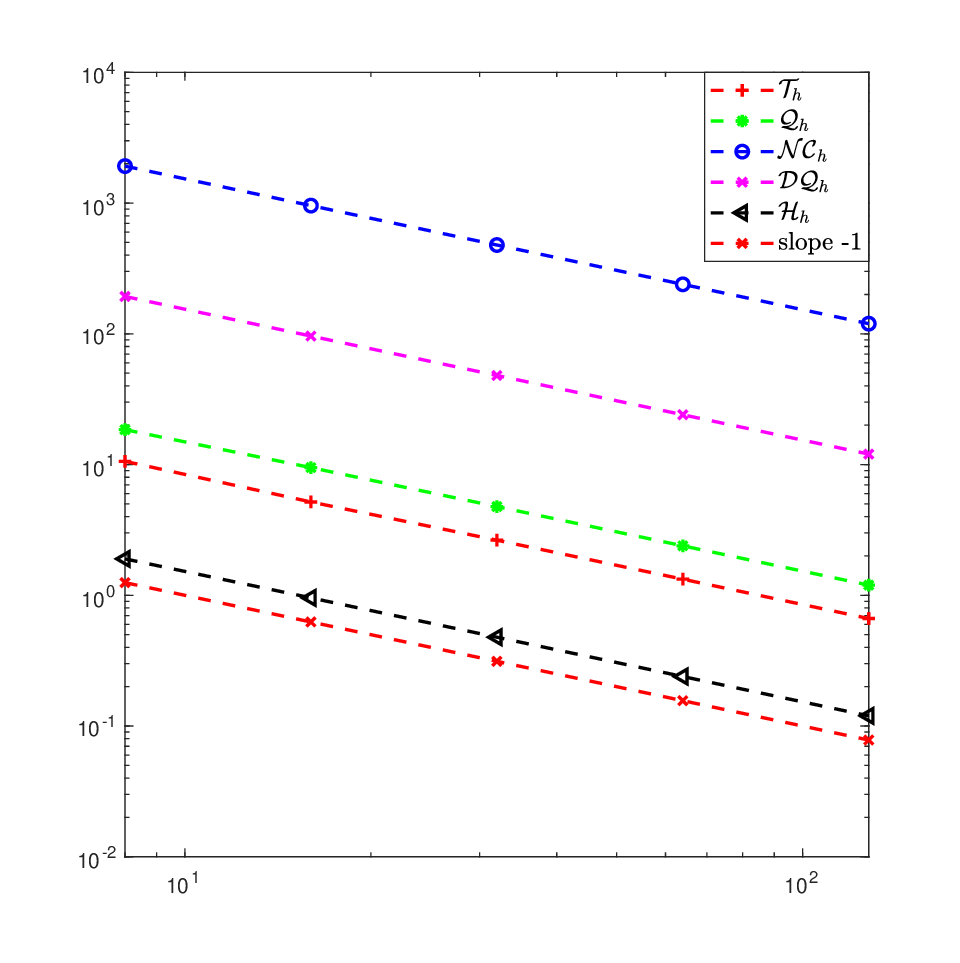}
 		\caption{Test 4: From left to right: temperature error for $\nu =10^{-2},10^{-4},10^{-6}$ for meshes $\CT_h$, $\mathcal{Q}_h$, $\mathcal{NC}_h$, $\mathcal{DQ}_h$ and $\mathcal{H}_h$.}
 		\label{FIG:error_temp}
 	\end{center}
 \end{figure}
 
 \begin{figure}[H]
	\begin{center}
                           \centering\includegraphics[height=4.4cm, width=4.2cm]{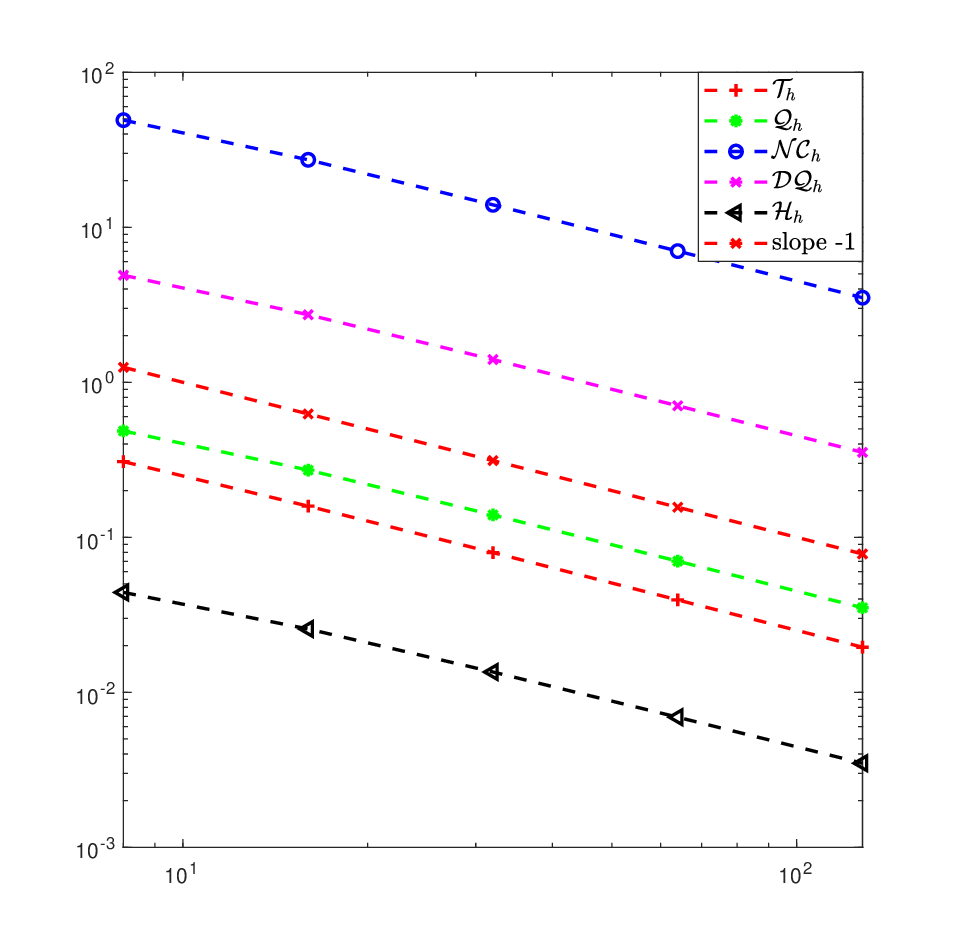}
                           \centering\includegraphics[height=4.4cm, width=4.2cm]{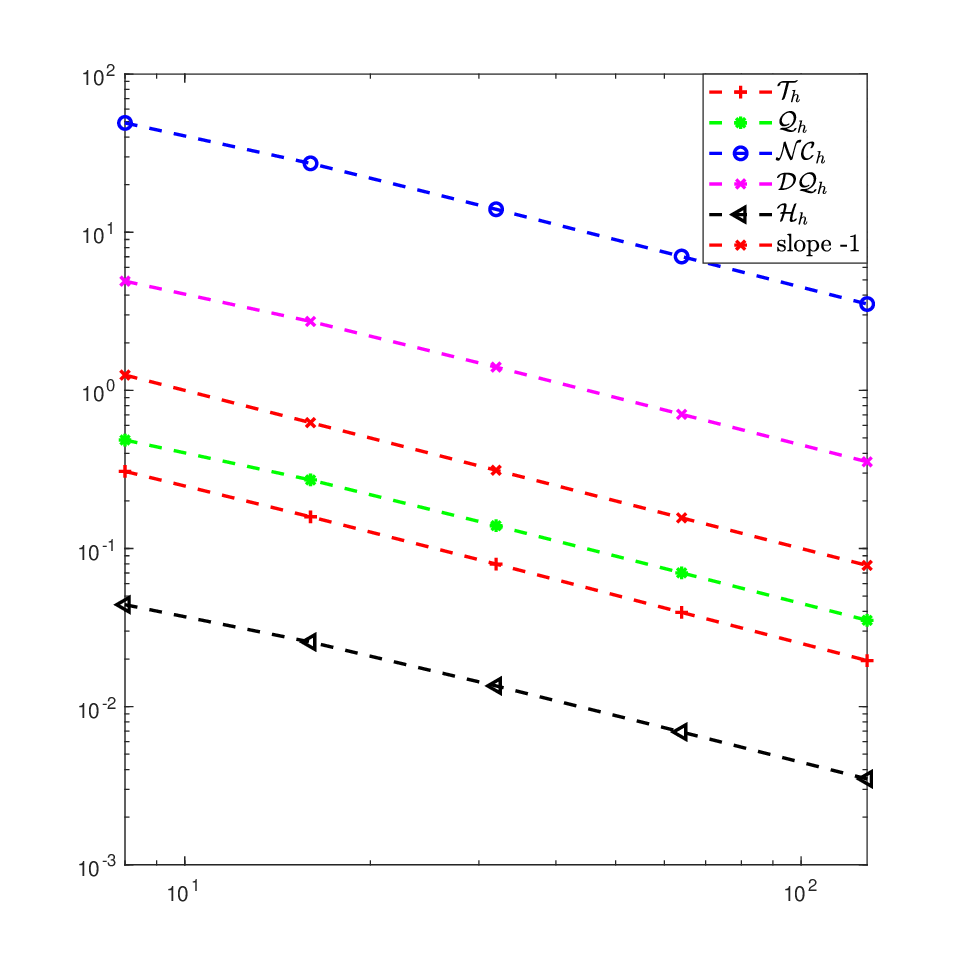}
                           \centering\includegraphics[height=4.4cm, width=4.2cm]{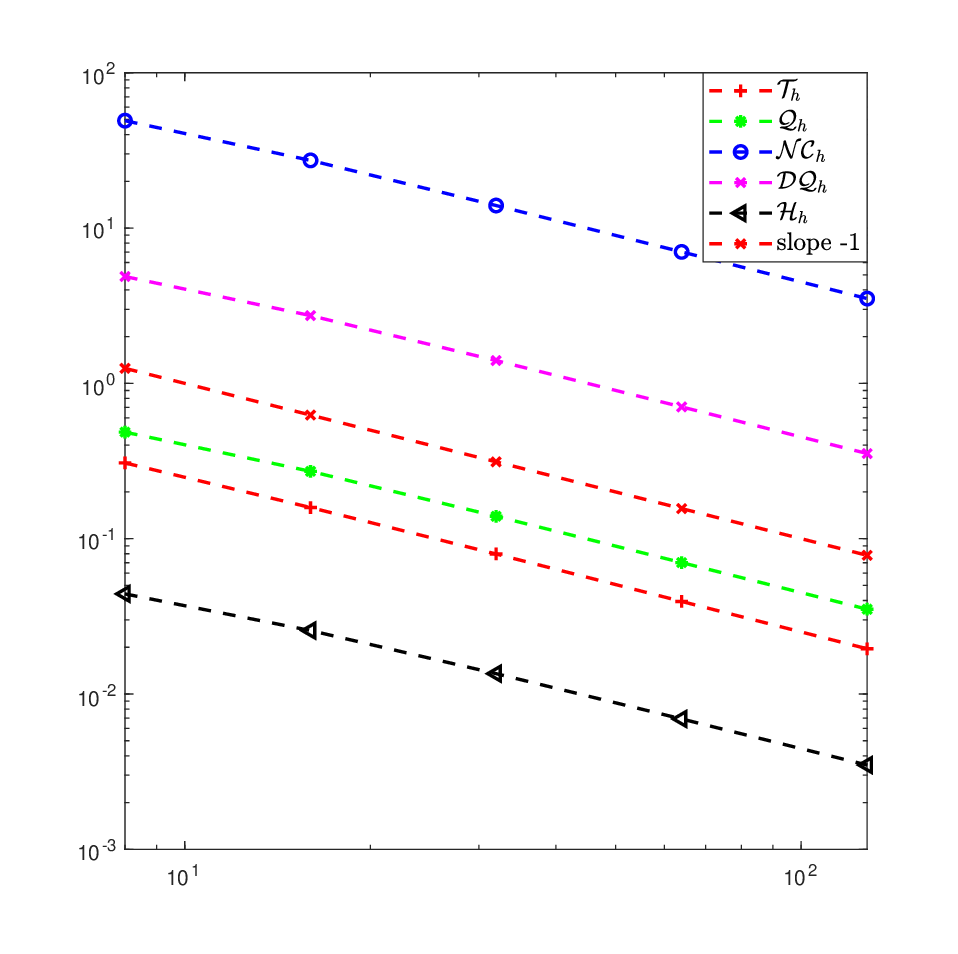}
 		\caption{Test 4: From left to right: pressure error for $\nu =10^{-2},10^{-4},10^{-6}$ for meshes $\CT_h$, $\mathcal{Q}_h$, $\mathcal{NC}_h$, $\mathcal{DQ}_h$ and $\mathcal{H}_h$.}
 		\label{FIG:error_press}
 	\end{center}
 \end{figure}
 
 From Figures \ref{FIG:error_vel}, \ref{FIG:error_temp} and \ref{FIG:error_press} we observe that when $\nu\rightarrow 0$, the order of convergence for the velocity is deteriorated, contrary to the pressure and temperature, which remain stable. We remark that this phenomenon is observed for all the polygonal meshes that we consider.
\textbf{Declaration of competing interest.}
The authors declare that they have no known competing financial interests or personal relationships that could have appeared
to influence the work reported in this paper.

\textbf{Aknowledgments.}
The authors are deeply grateful to Prof. Gonzalo Rivera (Universidad de Los Lagos, Chile) for his comments and suggestions on the computational experiments. This paper is part of Danilo Amigo's Ph.D. thesis in the Ph.D. Program in Applied Mathematics at Universidad del Bío-Bío.

\textbf{Data availability.}
Data will be made available on request.

\bibliographystyle{amsplain}
\bibliography{bib_LOQ}

\end{document}